\documentclass[12pt,a4paper,twoside]{amsart}

\usepackage[all]{xypic}
\usepackage{latexsym,amsmath,amscd,amssymb,amsthm,epsfig,times,stmaryrd}
\usepackage{graphicx}
\usepackage{pinlabel}

\usepackage{color}
\definecolor{maus}{rgb}{0.2,0.2,0.2}

\newtheorem{thm}{Theorem}[section]
\newtheorem{prop}[thm]{Proposition}
\newtheorem{lem}[thm]{Lemma}
\newtheorem{cor}[thm]{Corollary}

\newtheorem*{tauthomotopythm*}{Theorem \ref{t:atoral homotopy}}
\newtheorem*{colinstabthm*}{Theorem \ref{t:colin stab}}

\newtheorem*{thm*}{Theorem}

\theoremstyle{definition}
\newtheorem{defn}[thm]{Definition}
\newtheorem{ex}[thm]{Example}
\newtheorem{rem}[thm]{Remark}
\newtheorem*{ex*}{Example}

\newcommand{\thmref}[1]{Theorem~\ref{#1}}
\newcommand{\lemref}[1]{Lemma~\ref{#1}}
\newcommand{\propref}[1]{Proposition~\ref{#1}}
\newcommand{\defref}[1]{Definition~\ref{#1}}
\newcommand{\exref}[1]{Example~\ref{#1}}
\newcommand{\remref}[1]{Remark~\ref{#1}}
\newcommand{\corref}[1]{Corollary~\ref{#1}}
\newcommand{\figref}[1]{Figure~\ref{#1}}
\newcommand{\secref}[1]{Section~\ref{#1}}

\newcommand{\lra}{\longrightarrow}

\newcommand{\TT}{\mathcal{ T}}

\newcommand{\GG}{\mathcal{ G}}
\newcommand{\FF}{\mathcal{ F}}
\newcommand{\II}{\mathcal{ I}}

\newcommand{\LL}{\mathcal{ L}}

\newcommand{\Q}{\mathbb{ Q}}
\newcommand{\R}{\mathbb{ R}}
\newcommand{\HH}{\mathbb{ H}}

\newcommand{\Z}{\mathbb{ Z}}

\newcommand{\N}{\mathbb{N}}

\newcommand{\mlabel}[1]{
                        \label{#1}}

\newcommand{\eing}[1]{\big|_{#1}}

\newcommand{\ww}{\wedge}

\newcommand{\eps}{\varepsilon}

\newcommand{\id}{\mathrm{id}}

\newcommand{\lmt}{\longmapsto}

\DeclareMathAccent{\ring}{\mathalpha}{operators}{"17}

\author{T. Vogel}

\title{Uniqueness of the contact structure approximating a foliation}

\address{T.~Vogel, Max-Planck-Institut f{\"u}r Mathematik, Vivatsgasse 7, 53111 Bonn, Germany}

\email{tvogel@mpim-bonn.mpg.de}

\date{\today} 


\begin{document}

\color{maus}



\begin{abstract}
According to a theorem of Eliashberg and Thurston a $C^2$-foliation on a closed $3$-manifold can be $C^0$-approximated by  contact structures unless all leaves of the foliation are spheres. Examples on the $3$-torus show that every neighbourhood of a foliation can contain non-diffeomorphic contact structures.  

In this paper we show uniqueness up to isotopy of the contact structure in a small neighbourhood of the foliation when the foliation has no torus leaf and is not a foliation without holonomy on parabolic torus bundles over the circle. This allows us to associate invariants from contact topology to foliations. As an application we show that the space of taut foliations in a given homotopy class of plane fields is not connected in general. 
\end{abstract}

\maketitle

\tableofcontents

\section{Introduction and results}

The purpose of this paper is to determine which foliations on closed $3$-manifolds have the property that all positive contact structures in a sufficiently small neighbourhood of the foliation are isotopic (for definitions and basic results see \secref{s:def}). According to the following theorem of Y.~Eliashberg and W.~Thurston \cite{confol} most foliations can be approximated by contact structures: 
\begin{thm}[Eliashberg, Thurston]   \mlabel{t:elth}
Let $\FF$ be an oriented $C^2$-foliation by surfaces on a closed oriented $3$-manifold. If $\FF$ is not isomorphic to the foliations by spheres on $S^2\times S^1$, then every $C^0$-neighbourhood of $\FF$ contains a positive contact structure.  
\end{thm}
It can be shown quite easily \cite{confol}  that the foliation by the first factor on $M=S^2\times S^1$ cannot be approximated by a contact structure, i.e. there is a $C^0$-neighbourhood of the foliation which does not contain a contact structure. 

\thmref{t:elth} provides a first link between foliations and contact structures. Before the appearance of \cite{confol} these fields developed independently. The approximation theorem allows one to obtain potentially interesting contact structures from construction of foliations. For example, the work of D.~Gabai \cite{gabai} on constructions of foliations from sutured manifold decompositions provides a rich source of interesting contact structures. Via this construction there is a connection between sutured manifolds and gauge theory \cite{km}. The most prominent application of this circle of ideas is in the proof of the Property-P-conjecture by P.~Kronheimer and T.~Mrowka \cite{km2}.  

In view of \thmref{t:elth} it is natural to ask to what extent the foliation determines the contact structures up to isotopy in sufficiently small neighbourhoods. The following well-known example shows that the isotopy type of a contact structure in a small neighbourhood of a foliation is not completely determined by the foliation. 

\begin{ex} \mlabel{ex:T3}
Let $\FF$ be the foliation of $T^2\times S^1=\R^3/\Z^3$ by tori corresponding to the first factor. Then for $0\neq\eps\to 0$ and $k\neq0$ the contact planes $\xi_k$ defined by the $1$-forms
$$
\alpha_{k,\eps} := dt+\eps(\cos(2\pi kt)dx_1-\sin(2\pi kt)dx_2)
$$
converge to $T\FF$. Different $\eps$ yield isotopic contact structures which we therefore denote by $\xi_k$. According to Y.~Kanda \cite{ka} the contact structure $\xi_k$ is isotopic to $\xi_l$ if and only if $k=l$. They are distinguished by their Giroux torsion (we review the definition of this invariant in \defref{d:torsion}). 
\end{ex}
However, the question whether or not torus leaves are the only source of ambiguity was raised by V.~Colin as Question 5.9 in \cite{colin habil}. Also, in the paper \cite{hkm} by K.~Honda, W.~Kazez and G.~Mati{\'c} these authors suggest\footnote{second paragraph on page 306 of \cite{hkm}} that \smallskip
\begin{center}
{\sl `` \ldots contact topology may ultimately be a \\discrete version of foliation theory.''}
\end{center}
\smallskip
One piece of evidence for this is the following theorem from \cite{hkm}. \begin{thm}[Honda-Kazez-Mati{\'c}, \cite{hkm}] 
Let $\psi$ be an orientation preserving pseudo-Anosov diffeomorphism of a hyperbolic surface and $M_\psi$ the surface fibration over the circle with monodromy $\psi$. 

There is a unique tight contact structure $\xi$ on $M_\psi$ such that $\langle e(\xi),[\Sigma]\rangle = 2-2g$. In particular, there is a $C^0$-neighbourhood $U$ of the foliation by fibers of $M_\psi\lra S^1$ in the space of plane fields so that $\xi\simeq\xi'$ for all pairs of positive contact structures $\xi,\xi'\in U$. 
\end{thm}
 
The following theorem answers Colin's question affirmatively up to a small set of exceptions. It can be viewed as confirmation of the above remark from \cite{hkm}.
\begin{thm} \mlabel{t:unique}
Let $\FF$ be a coorientable $C^2$-foliation (or a $C^2$-con\-foliation) on a closed oriented $3$-manifold satisfying the following conditions:
\begin{itemize}
\item[(i)] $\FF$ has no torus leaf,
\item[(ii)] $\FF$ is not a foliation by planes,
\item[(iii)] $\FF$ is not a foliation by cylinders. 
\end{itemize}
Then there is a $C^0$-neighbourhood $U$ of $\FF$ in the space of plane fields and a contact structure $\xi$ in $U$ such that every positive contact structure in $U$ is isotopic to $\xi$.
\end{thm}

According to a theorem of H.~Rosenberg \cite{ros-planes}, $C^2$-foliations by planes exist only on the $3$-torus. Later G.~Hector \cite{hector} proved that foliations by cylinders exist only on parabolic $T^2$-bundles over the circle. This shows that the foliations in (ii), (iii) of \thmref{t:unique} are very special. Thus torus leaves are essentially the only source of non-uniqueness of the isotopy classes of contact structure which are sufficiently close to a given confoliation. (Note that by the Reeb stability theorem all coorientable foliations on $3$-manifolds with spherical leaves are equivalent to the product foliation on $S^2\times S^1$. Therefore spherical leaves play no particular role in \thmref{t:unique}.) Foliations which satisfy the assumptions of \thmref{t:unique} will be called {\em atoral} and the isotopy class of positive contact structures in the neighbourhood of \thmref{t:unique2} {\em approximates} $\FF$. 

When a foliation has torus leaves, then every neighbourhood contains non-isotopic contact structures distinguished by their  Giroux torsion. If the torus leaves satisfy a certain stability condition, then the Giroux torsion is the only source of ambiguity of the contact structures in small neighbourhoods of $\FF$. In order to state the corresponding theorem  we need the following definition.
\begin{defn}
Two contact structures $\xi',\xi''$ are stably equivalent with respect to a finite collection of pairwise disjoint embedded tori if the following conditions hold:
\begin{itemize}
\item[(i)]  It is possible to isotope the tori and to choose a contact form $\alpha$ such that the restriction of $\alpha$ to the isotoped tori is closed (such tori are called {\em pre-Lagrangian}).
\item[(ii)] $\xi'$ and $\xi''$ become isotopic after a contact structure $$
(T^2\times[0,1],\ker(\cos(2\pi k(t+t_0))dx_1-\sin(2\pi k(t+t_0)) dx_2))
$$
with suitable parameters $k>0$ and $t_0\in\R$ is inserted along the pre-Lagrangian tori. 
\end{itemize}
\end{defn}

\begin{thm} \mlabel{t:unique2}
If the foliation (or confoliation) $\FF$ satisfies only the weaker hypothesis
\begin{itemize}
\item[(i')] All torus leaves of $\xi$ have attractive holonomy
\end{itemize}
and (ii), (iii) of \thmref{t:unique}, then there is a $C^0$-neighbourhood $U$ of $\FF$ such that any two contact structures $\xi', \xi''$ in $U$ are stably equivalent with respect to the torus leaves of $\FF$.
\end{thm}
As explained in \secref{s:tori} condition (i') can be generalized somewhat further. We explain the details  in \secref{s:fixing closed tori}. Examples show that foliations with torus leaves violating (i') do not satisfy the conclusion in the theorem above. However, the examples known to the author in which this happens are rather special. In view of potential applications  of \thmref{t:unique2}The characterization of those foliations with torus leaves which violate (i') but still satisfy the conclusion of \thmref{t:unique2} is an interesting open problem (for example \propref{p:tight approx}). 

\thmref{t:unique} allows us to associate invariants from contact topology (for example the contact invariant from Heegard-Floer theory) to atoral foliations. Combining \thmref{t:elth} and \thmref{t:unique} we obtain
\begin{tauthomotopythm*}
Let $\FF_t, t\in [0,1]$, be a $C^0$-continuous family of atoral $C^2$-foliations. Then the positive contact structures $\xi_0$ respectively $\xi_1$ approximating $\FF_0$ respectively $\FF_1$ are isotopic. 
\end{tauthomotopythm*} 
This provides an obstruction for finding a path of atoral foliations connecting two atoral foliations. This is of interest since the work of H.~Eynard \cite{eynard} shows that two atoral foliations are homotopic through foliations as soon as the two foliations are homotopic as plane fields. The foliations in the homotopy constructed by H.~Eynard contain Reeb components and  therefore violate the hypothesis of \thmref{t:unique}. In many interesting cases the class of atoral foliation on a manifold coincides with the class of taut foliations. In \exref{ex:2311} we show that the Brieskorn homology sphere $\Sigma(2,3,11)$ has a taut foliation $\FF$ such that $\FF$ is not homotopic to the foliation $\overline{\FF}$ (this is $\FF$ with the opposite coorientation) through foliations without Reeb components although $\FF$ and $\overline{\FF}$ are homotopic as oriented plane fields.  Other applications of \thmref{t:unique} and \thmref{t:unique2} can be found in \secref{s:further}. 

Finally, let us note that when it is possible to prove a parametric version of \thmref{t:unique} without too much additional difficulty, then we will do so. The parametric versions are \thmref{t:transitive} and \thmref{t:sack contract}, they cover foliations with holonomy which do not have closed leaves.   

\subsection{Some ideas for the proof of the uniqueness result}

The main  tools used in this paper stem from \cite{col} by V.~Colin, \cite{confol} by Y.~Eliashberg and W.~Thurston, \cite{gi-bif} by E.~Giroux and \cite{hkm} by K.~Honda, W.~Kazez and G.~Mati{\'c}. 

Just like the proof of \thmref{t:elth} the proof of the uniqueness theorem deals with minimal sets and the rest of the manifold in separate steps. These steps are treated in different order in the proofs of \thmref{t:elth} and \thmref{t:unique}.  First, we fix a pair of neighbourhoods $\widehat{N}\supset N$ of the set of closed leaves and particular curves with linear holonomy (Sacksteder curves, cf. \secref{s:sack intro}) and we choose a $C^0$-neighbourhood of $\FF$ such that the restriction of every contact structure in the $C^0$-neighbourhood to $\widehat{N}$ is tight. Given two contact structures $\xi,\xi'$ in an even smaller neighbourhood of $\FF$ we first deform $\xi$ so that the resulting contact structure $\widehat{\xi}$ coincides with $\xi'$ outside of $N$   such that the contact structures remains tight on $\widehat{N}$ throughout the deformation. Then we use classification results for tight contact structures in order to show that $\widehat{\xi}$ and $\xi'$ are isotopic on $\widehat{N}$.  A somewhat different procedure has to be used when $\FF$ is a foliation without holonomy. 

The first step follows the structure of the proof of the following theorem of V.~Colin.  
\begin{colinstabthm*}[V.~Colin, \cite{col}]
Let $\xi$ be a contact structure on the closed $3$-manifold $M$. Then there is a $C^0$-neighbourhood of $\xi$ in the space of smooth plane fields so that every contact structure in $U$ is isotopic to $\xi$. 
\end{colinstabthm*}
Since we start with a confoliation and not with a contact structure several modifications are needed. As in the proof of \thmref{t:colin stab} one starts with a polyhedral decomposition of $M$ and the main modification of the proof of \thmref{t:colin stab} concerns extensions of the polyhedra which lead to controlled modifications of the characteristic foliation on the boundary. 

The contact structures $\xi,\widehat{\xi},\xi'$ are transverse to a rank $1$-foliation on the tubular neighbourhood $\widehat{N}$. This can be used to show that the restrictions of $\widehat{\xi},\xi'$ to $\widehat{N}$ are tight. We then want to appeal to classification results for tight contact structures. In the case when a connected component of $\widehat{N}$ is a solid torus and the characteristic foliation on the boundary has exactly two non-degenerate closed leaves the contact structure is uniquely determined up to isotopy. If a connected component of $\widehat{N}$ is the tubular neighbourhood of a closed leaf of $\FF$, then the contact structure on $\widehat{N}\simeq \Sigma\times[-1,1]$ is not uniquely determined by the properties of $\widehat{\xi}, \xi'$  we have mentioned so far.  

If $\Sigma$ is a closed leaf then the Euler class $e(\FF)$  of $\FF$ satisfies the extremal condition 
\begin{equation} \label{e:extr intro}
\langle e(\FF),[\Sigma] \rangle = \pm (2-2g)
\end{equation}
where $g$ is the genus of $\Sigma$ and $g\ge 2$ since in the presence of a spherical leaf there is nothing to prove and the case of torus leaves is excluded. We will assume in the following that \eqref{e:extr intro} holds with the plus sign on the right hand side. (In the opposite case one has to interchange positive/negative singularities and attractive/repulsive closed leaves). 

Following ideas in \cite{gi-bif} we show that tight contact structures on $\Sigma\times[-1,1]$ can be distinguished using sheets of the movie of characteristic foliations on $\Sigma_t:=\Sigma\times\{t\}$ which contain attractive closed leaves of the characteristic foliation on $\Sigma_{\pm 1}$. A sheet $A$ is an embedded submanifold in $M$ such that the characteristic foliation $A(\xi)$ is a non-singular foliation by circles, more details can be found in \secref{s:movies main}. The sheets we will consider are formed by closed leaves of the characteristic foliations $\Sigma_t(\xi)$ and of simple closed curves formed by positive elliptic singularities and stable leaves of positive hyperbolic singularities. (If $\xi$ is sufficiently close to $\FF$, then $\Sigma_t(\xi)$ has no negative singularities.) 

When the genus of $\Sigma$ is larger than $1$, we show in \secref{s:higher} using the pre-Lagrangian extension lemma from \secref{s:prelag extend} that a tight contact structure on $\widehat{N}$ is uniquely determined by its restriction to $\widehat{N}\setminus N$ when it is sufficiently close to $\FF$.  Then it follows that $\widehat{\xi}$ is isotopic to $\xi'$ and therefore $\xi$ is isotopic to $\xi'$. 

If $\Sigma$ is a closed surface of genus $\ge 2$ we rely on classification results for tight contact structures from  \cite{hkm}.  If $\Sigma$ is a torus, then we can use the more complete classification of tight contact structures on $T^2\times [-1,1]$ in the form given in \cite{gi-bif} to obtain \thmref{t:unique2}

One of the most important points in the proof of these theorems is to ensure that there is no sheet connecting the two boundary components of $\widehat{N}$. This is done by choosing the neighbourhood of $\FF$ in the space of plane fields properly. In particular, all plane fields are transverse to the foliation on $\widehat{N}\simeq \Sigma\times[-1,1]$ induced by the second factor.  

Because of the position of the contact plane field with respect to the parts of sheets consisting of attractive closed leaves of the characteristic foliation on level surfaces $\Sigma_t$ restrictions on the $C^0$-distance between the contact structures and  $\FF$ lead to restrictions on the position of sheets.  This is illustrated in \figref{b:un-stable}.

The figures on the left hand side of \figref{b:un-stable} show the intersection of $\FF$ with an annulus which is transverse to the line field $\widehat{N}(\FF)$ when $\Sigma$ is a stable (upper part of \figref{b:un-stable}) or an unstable (lower part) torus leaf. In each case the right hand side shows  the intersection of the same annulus with a sheet of a contact structure which could arise when $\FF$ is approximated by a contact structure $\xi$. The thickened arcs correspond to those parts of sheets where $\Sigma_t\cap A$ is an attractive closed leaf of $\Sigma_t(\xi)$ and the straight segments correspond to the contact planes.  

\begin{figure}[htb]
\begin{center}
\includegraphics[scale=0.8]{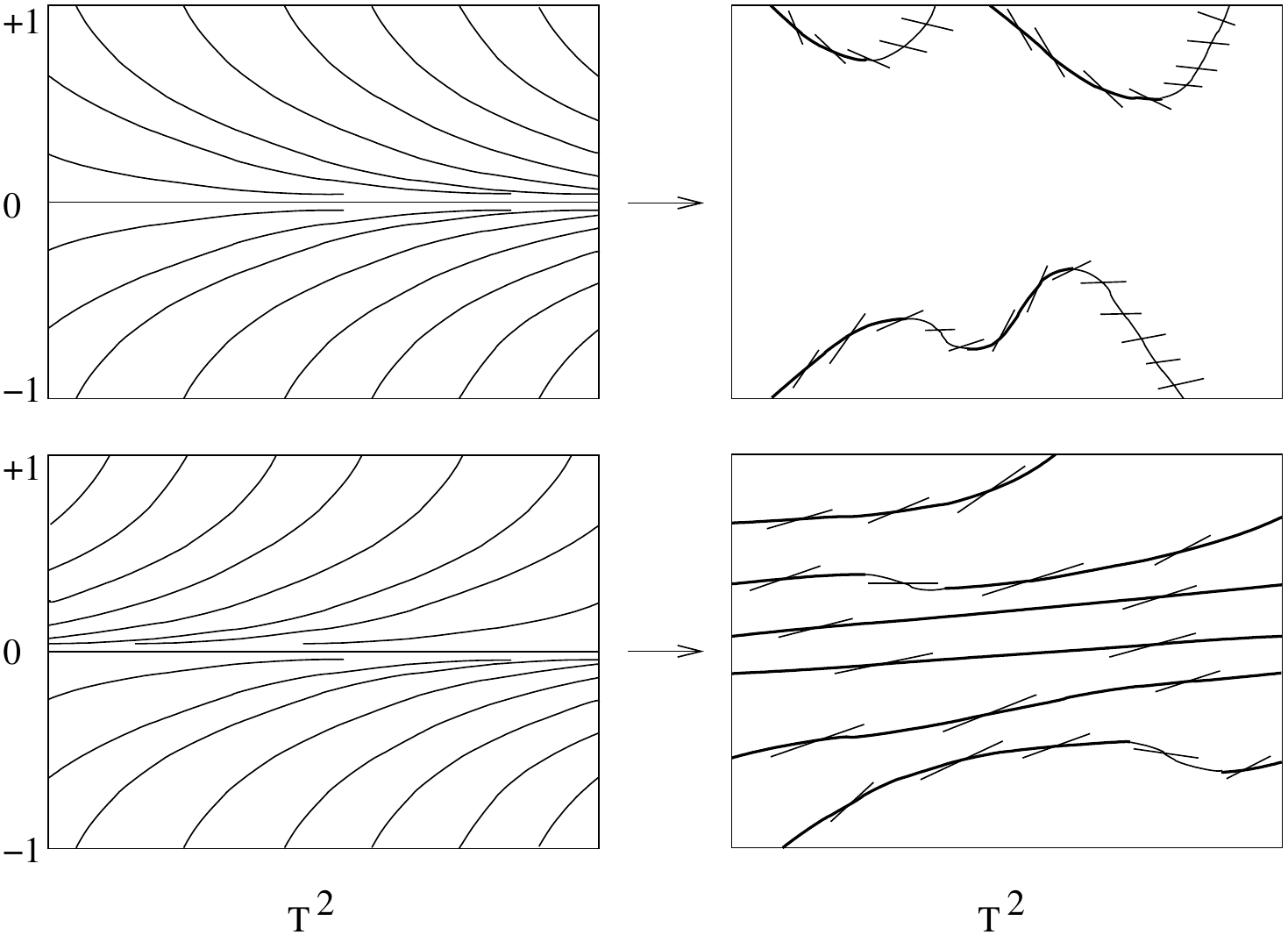}
\end{center}
\caption{Stable/unstable torus leaves and sheets of nearby contact structures.}\label{b:un-stable}
\end{figure}

A difference between the case of torus leaves and the case of surfaces of higher genus is that in the case of surfaces of higher genus  an embedding of an annulus connecting the two boundary components of $\widehat{N}\simeq \Sigma\times[-1,1]$  is determined up to isotopy by the boundary curves while this is not true if the leaf is a torus. If a torus leaf is stable then one can  still choose the neighbourhood $U$ of $\FF$ such that there are no sheets connecting the two boundary components of $\partial\widehat{N}$. If the torus is not stable, then it may happen that no such neighbourhood exists.

As we have already mentioned foliations without holonomy have to be treated in a different fashion. 
Recall from \cite{confol} that foliations without holonomy can be $C^0$-approximat\-ed by fibrations.   The most delicate part of the proof of \thmref{t:unique} for foliations without holonomy is to find a fibration which approximates the foliation well enough so that one can exclude the appearance of sheets which intersect every fiber of the fibration for contact structures close to $\FF$. These approximations are  constructed in \secref{s:facts without} using a theorem of Dirichlet about Diophantine approximations of real numbers.

\subsection{Organization of the paper}

This paper consists of nine sections. The author hopes that the results of this paper are relevant for people interested in contact structures or foliations. In order to make it more accessible we have included most of the relevant definitions and basic theorems in \secref{s:prelim}. However we are very brief and we only prove statements which we did not find in the literature. Also, some results (like Sacksteder's theorem) from the theory of foliations are stated in the section where they are used. In \secref{s:movies main} we review Giroux's theory of movies of contact structures from \cite{gi-bif} (some of this material can be found in \cite{gei}). Again most results we prove are modifications of theorems in \cite{gi-bif} or results which are probably well-known but which we did not find in the literature in the required form. An exception is the pre-Lagrangian extension lemma in \secref{s:prelag extend} which is a new result.

The proofs of \thmref{t:unique} and \thmref{t:unique2} are contained in the Sections 4--8. The author hopes that by dealing with increasingly more difficult situations in separated sections the proofs become more transparent than a proof covering all possible types of minimal sets at once. 
\begin{itemize}
\item \secref{s:transitive} deals with the case of transitive confoliations and its main purpose is to extend the proof of \thmref{t:colin stab} using ribbons. This technique will be used in all subsequent cases, except in the case of foliations without holonomy. 
\item \secref{s:sack} contains a proof of the uniqueness theorem for confoliations which are not foliations without holonomy and have no closed leaves.  In this section we also show how the subsequent proofs for foliations carry over to the confoliated case. 
\item \secref{s:tori} contains a proof of \thmref{t:unique2} when there are no closed leaves of higher genus. 
\item \secref{s:higher} completes the proofs of \thmref{t:unique} and \thmref{t:unique2} for (con-)foliations which are not foliations without holonomy. 
\item \secref{s:without holonomy} contains the proof of \thmref{t:unique} when $\FF$ is a foliation without holonomy. We also  discuss which torus bundles satisfy the conclusion of \thmref{t:unique2}.
\end{itemize}
Finally, \secref{s:appex} contains a discussion of applications of the uniqueness result and examples where the approximating contact structure is not well defined.  In particular, we show that neighbourhoods of  foliations by planes and foliations by cylinders contain many non-isotopic contact structure with vanishing Giroux torsion.

{\thanks Acknowledgments:} This work has benefited from conversations with V.~Colin, Y.~Eliashberg and, in particular, with J.~Bowden. This project was started while the author visited the MSRI, Berkeley, and parts of the paper were written while visiting the Simons Center for Geometry and Physics in Stony Brook. It is a pleasure for me to thank the organizers of the respective programs, and Y.~Eliashberg in particular.

\section{Preliminaries} \mlabel{s:prelim}
 
The sections \ref{s:def}--\ref{s:tight} contain basic definitions from contact topology and the theory of foliations in order to make this text more accessible. In \secref{s:tight class} we review the relevant classification results for tight contact structures.

\subsection{Contact structures, foliations and confoliations} \mlabel{s:def}

In this paper $M$ will always denote a closed connected oriented $3$-manifold. We fix an auxiliary Riemannian metric on $M$. In this section we give some standard definitions and fix some conventions used throughout this paper. We start with the definition of a foliation. Usually,  foliations are defined in terms of a foliated atlas. For our purposes the following definition is more convenient.
\begin{defn} \mlabel{d:fol def}
A $C^k$-smooth, $k\ge 1$, {\em foliation} $\FF$ on $M$ is a plane field such that if $\alpha$ is a $C^k$-smooth $1$-form defined on an open set $U$ given by a $1$-form $\alpha$ with $\ker(\alpha)\eing{U}=\FF$, then $\alpha\ww d\alpha\equiv 0$. 
\end{defn}
By the theorem of Frobenius and Prop. 1.0.2. of \cite{pet} this is equivalent to the standard definition of a foliation of codimension $1$ when $k\ge1$. When $k=0$ it is not even true in general that a foliation defined by an atlas corresponds to a subbundle of the tangent bundle of $M$. But since we will only be interested in $C^2$-foliations, we do not have to discuss this (more information can be found in \cite{cc}). 

Given a foliation $\FF$ there is a collection of immersed hypersurfaces which are everywhere tangent to $\FF$, a maximal connected hypersurface with this property is a {\em leaf} of $\FF$. We will often confuse the collection of leaves with the corresponding plane field.

\begin{defn} \mlabel{d:cont def}
A positive {\em contact structure} on a $3$-manifold is a $C^1$-plane field $\xi$ such that every $1$-form $\alpha$ defined on an open set $V$ with $\ker(\alpha)=\xi\eing{V}$ satisfies $\alpha\ww d\alpha>0$. Negative contact structures are defined by requiring $\alpha\ww d\alpha<0$. A positive {\em confoliation} $\xi$ is a $C^1$-smooth plane field on $M$ such that $\alpha\ww d\alpha\ge 0$ for every $1$-form defining $\xi$ on an open set. 
\end{defn}
Note that if $\alpha\ww d\alpha>0$ holds somewhere, then the same is true for every other $1$-form defining the same distribution. All plane fields in this paper will be oriented subbundles of $TM$, so we can assign an Euler class to each foliation, contact structure or confoliation. 
We consider two plane fields as different when they coincide but have opposite orientations. If $\xi$ is an oriented plane field,  then $\overline{\xi}$ denotes the same plane field with its orientation reversed.

The condition that $\xi$ is a positive confoliation has the following geometric interpretations: 
\begin{enumerate}
\item Fix a vector field $X$ tangent to $\xi$ and a disc $D$ transverse to the flow lines of $X$. The disc is oriented such that its orientation followed by the orientation of $X$ is the orientation of $M$. We denote the flow of $X$ by $\varphi_t$. Then the line field $TD\cap \varphi_{-t*}(\xi)$ rotates clockwise with positive speed as $t$ increases.
\item Let $\xi$ be a positive confoliation on $D^2\times \R$  which is transverse to the second factor and complete as a connection. Then the parallel transport $h: \R\lra\R$ along $\partial D^2$ satisfies
\begin{equation} \label{e:decreasing} 
h(x)\le x \textrm{ for all } x\in\R.
\end{equation}
\end{enumerate}
\thmref{t:elth}  and the second interpretation implies that the closure of the space of positive contact structures in the space of $C^1$-plane fields with respect to the $C^0$-distance is exactly the space of positive  confoliations. 

The following terminology is borrowed from contact topology.
\begin{defn} \mlabel{d:legendrian}
A piecewise smooth curve $\gamma$ in $M$ is {\em Legendrian} if it is tangent to $\xi$ where $\xi$ is any smooth plane field on $M$. 
\end{defn}
 
\begin{defn} \mlabel{d:confol stuff}
If $\xi$ is a confoliation, then the open set
$$
H(\xi)=\{x\in M \,|\, \xi \textrm{ is a pos. contact structure on a neighbourhood of }x\}
$$ 
is the {\em contact region} of $\xi$. We say that $\xi$ is {\em transitive} if for every point of $M$ there is a Legendrian curve which connects $x$ and $H(\xi)$. The {\em fully foliated set} of a confoliation consists of those points which are not connected to $H(\xi)$ by Legendrian curve.  
\end{defn}
The fully foliated set of a confoliation is a closed subset of $M$ containing immersed hypersurfaces everywhere tangent to $\xi$. We will refer to these hypersurfaces as leaves. The theorems from foliation theory which we shall use later carry over to fully foliated sets of confoliations. 

\begin{defn} \mlabel{d:min set}
Let $\FF$ be a foliation on $M$. A subset $X\subset M$ is called {\em minimal} if 
\begin{itemize}
\item[(i)] $X$ is closed,
\item[(ii)] $X$ is a non-empty union of leaves of $\FF$, and
\item[(iii)] $X$ contains no proper subset satisfying (i) and (ii).
\end{itemize}
\end{defn}
If $M$ is compact and carries a foliation, then there are minimal sets and the topological closure of a leaf contains a minimal set. Moreover, every minimal set $X$ of a foliation belongs to one of following three categories: In order to describe them we fix a point $p\in X$ and short interval $I$ transverse to $\FF$ containing $p$. 
\begin{enumerate}
\item $X$ is a closed leaf. Then $X\cap I$ is a discrete set.
\item $X=M$, then every leaf is dense and $X\cap I=I$. 
\item $X$ is an exceptional minimal set, i.e. $X\cap I$ is a Cantor set (so no point of $X\cap I$ is isolated and $X\cap I$ is nowhere dense).  
\end{enumerate}
This is true for foliations of codimension one regardless of the smoothness of the foliation \cite{cc}. It also holds for the fully foliated set of a confoliation. 

If $L$ is an integral surface of a confoliation $\xi$ (i.e. a surface tangent to $\xi$) and $\gamma : S^1 \lra L$ a smooth map, then the holonomy along $\xi$ is defined as follows: Fix an immersed annulus 
$$
\varphi : S^1\times (-\delta,\delta)\lra M
$$ 
transverse to $\xi$ such that $\varphi(z,0)=\gamma(z)$. The characteristic foliation on this annulus has a closed leaf, namely $\gamma(S^1)$ and the Poincar{\'e} return map $h_\varphi$ defined by parallel transport along the oriented curve $\gamma$ is well defined on a neighbourhood of $0$ in $(-\delta,\delta)$. The conjugacy class $h_\gamma$ of the germ of $h_\varphi$ depends only on $\gamma$, it is independent from $\varphi$. In particular, it makes sense to speak of attractive and repulsive holonomy (i.e. $|h_\gamma(x)|<|x|$ etc.) or of fixed points on both sides of $\gamma$ in the annulus. 

\begin{defn} \mlabel{d:holonomy}
The conjugacy class of this germ is the {\em holonomy} of $\xi$ along $\gamma$ and $\xi$ has {\em non-trivial linear holonomy} along $\gamma$ if $h_\gamma'(0)\neq 1$. A foliation $\FF$ is a {\em foliation without holonomy} if  $h_\gamma$ is the germ of the identity for all closed curves $\gamma$ tangent to $\FF$.  
\end{defn}

If $\xi$ is a foliation, then the holonomy depends only on the free homotopy class of $\gamma$ in the integral surface. If $\xi$ is not a foliation, then we have at least the following lemma. 
\begin{lem} \mlabel{l:confol holonomy}
Let $L$ be an integral surface of a confoliation. If $h_\gamma$ has non-trivial linear holonomy, then the same is true for all curves which are freely homotopic to $\gamma$. Also, if  $\Sigma$ is a closed surface and $\gamma$ is a non-separating simple closed curve with attractive holonomy, then every curve isotopic to $\gamma$ also has attractive holonomy. 
\end{lem}
\begin{proof}
The first statement is proved in \cite{confol}. The second statement follows almost immediately from \eqref{e:decreasing}: Let $\gamma,\gamma'$ be isotopic such that $\gamma$ has attractive holonomy. Consider a covering  of $\Sigma$ such that lifts of $\gamma$ remain closed but $\gamma,\gamma'$ have disjoint lifts $\widetilde{\gamma},\widetilde{\gamma}'$ and the annulus $\widetilde{A}$ between the two lifts has $(-\widetilde{\gamma}\cup\widetilde{\gamma}')$ as its oriented boundary. Such a covering exists because $\gamma$ is non-separating in $\Sigma$. 

Pick an embedded arc $\lambda\subset\widetilde{A}$  connecting  $\widetilde{\gamma}(0),\widetilde{\gamma}'(0)$. Then  \eqref{e:decreasing} applied to the disc bounding the concatenation of $\lambda,\widetilde{\gamma}',(-\lambda),(-\widetilde{\gamma})$ and the pulled back confoliation on a tubular neighbourhood $\Sigma\times(-\delta,\delta)$ of $\Sigma$ such that the second factor is transverse to $\xi$ 
implies 
$$
h_{\gamma'}(x)\le h_{\lambda}\circ h_{\gamma}\circ h_\lambda^{-1}(x)
$$
since $h_\gamma=h_{\widetilde{\gamma}}$ and $h_{\gamma'}=h_{\widetilde{\gamma}'}$. Also, we use the obvious definition of the holonomy $h_\lambda$ along arcs. This implies the claim for $x>0$. An analogous argument proves the claim for $x<0$. 
\end{proof}
It is easy to construct confoliations on neighbourhoods of surfaces such that the second conclusion of the lemma does not hold  for a separating curve (when attractive is defined by strict inequalities). 

Recall the Reeb stability theorem. It can by found e.g. in \cite{cc}. The last part in the statement below is a consequence of the usual Reeb stability theorem (cf. Proposition 1.3.7 of \cite{confol}).  
\begin{thm}[Reeb stability] \mlabel{t:Reeb}
Let $\FF$ be a foliation and $\Sigma$ a leaf of $\FF$ diffeomorphic to a sphere. Then $\FF$ is diffeomorphic to the foliation by the first factor on $S^2\times S^1$.  Every confoliation transverse to the fibers of $S^2\times S^1\lra S^2$ is diffeomorphic to a foliation by spheres. 
\end{thm}
This theorem holds with minimal smoothness assumptions on the foliation and is also true for confoliations. Since the product foliation on $S^2\times S^1$ cannot be approximated by contact structures spherical  leaves do not play any role in the uniqueness problem.

\subsection{Gray's theorem, surfaces in contact manifolds, convexity} \mlabel{s:charfol}

In the proof of \thmref{t:unique} will use Gray's theorem:
\begin{thm} \mlabel{t:gray}
Let $\xi_t$ be a smooth family of smooth contact structures on a closed manifold $M$. Then there is an isotopy $\psi_t$ of $M$ so that $\psi_{t*}(\xi_t)=\xi_0$ for all $t$.
\end{thm}
The proof of this theorem is based on Moser's method which is described e.g. in \cite{mcduff}. \thmref{t:gray} holds in the relative case (i.e. if $\xi_t$ is constant on some domain, the resulting isotopy is then the identity on that domain) and it also works with parameters. By Gray's theorem, in order to prove \thmref{t:unique}, it suffices to find a neighbourhood of $\xi$ so that for every pair of contact structures in that neighbourhood there is a family of contact structures interpolating between them. 

The Moser method is omnipresent in all results producing contact isotopies, e.g. the theory of convex surfaces outlined below.

\subsubsection{Characteristic foliations and their singular points} \mlabel{s:sing points}

Let $\Sigma$ be an oriented surface embedded in a contact manifold. If $\Sigma$ has boundary, then the boundary will be assumed to be Legendrian.  

\begin{defn} \mlabel{d:char fol}
The {\em characteristic foliation} $\Sigma(\xi)$ on $\Sigma$ is determined by the singular line field $\xi\cap T\Sigma$ on $\Sigma$, the singularities are points where $\xi(x)=T_x\Sigma$. A singularity is {\em positive} if $\xi_x=T_x\Sigma$ as oriented vector spaces, otherwise the singularity is {\em negative}. If $\Sigma(\xi)$ is one-dimensional at $x\in\Sigma$, then $\Sigma(\xi)(x)$ is oriented so that this orientation followed by the coorientation of $\xi$ coincides with the orientation of the surface.  
An isolated singular point of the characteristic foliation is {\em elliptic} respectively {\em hyperbolic} if its index is $+1$ respectively $-1$. 
\end{defn}

The fact that $\xi$ is a contact structure has strong consequences for the characteristic foliation on a small neighbourhood of the singularities of $\Sigma(\xi)$. Recall from \cite{gi-conv} that the divergence of a singular point of $\Sigma(\xi)$ never vanishes, its sign is well defined and coincides with the sign of the singularity. As the next lemma shows, this is the only property which distinguishes characteristic foliations of contact structures from general singular foliations:

\begin{lem} \mlabel{l:char fol det}
Let $\GG$ be a singular foliation on $\Sigma$ such that there is a defining form $\alpha$ with $d\alpha\neq 0$ at all singular points. Then there is a contact structure $\xi$ on $\Sigma\times(-1,1)$ such that $\Sigma_0(\xi)=\GG$. This contact structure is unique up to an isotopy on a small neighbourhood of $\Sigma_0$ and the isotopy is tangent to the characteristic foliation on $\Sigma_0$. 
\end{lem}

The following lemma shows that the dynamical properties of the characteristic foliation are quite restricted near isolated singular points. A part of this lemma can be found in \cite{gi-bif}, cf. p. 629.

\begin{lem} \mlabel{l:char fol sing}
Let $p\in\Sigma$ be an isolated singular point of the characteristic foliation $\Sigma(\xi)$. Then the index of $p$ equals $-1,0$ respectively $+1$ and the characteristic foliation on a neighbourhood of $p$ is topologically conjugate to neighbourhoods of a hyperbolic, simply degenerate respectively elliptic singularities.
\end{lem}

\begin{proof}
We assume that $p$ is positive. Choose local coordinates $x_1,x_2$ on $\Sigma$ around $p$ with $x_1(p)=x_2(p)=0$ and a $1$-form defining $\alpha$ on a small neighbourhood of $p$ so that there is a vector field $V$ on $\Sigma$ near $p$ such that 
$$
i_V\left(d\alpha\eing{\Sigma}\right)=\alpha\eing{\Sigma}.
$$
Since $\alpha$ is a contact form, $d\alpha$ is an area form on $\Sigma$ near $p$ and the vector field $V$ is well defined near $p$. In terms of  $x_1,x_2$ 
$$
V(x_1,x_2) = \left(\begin{array}{cc} a_{11}(x_1,x_2) & a_{12}(x_1,x_2) \\ a_{21}(x_1,x_2) & a_{22}(x_1,x_2) \end{array}\right)\left(\begin{array}{l}x_1 \\ x_2 \end{array}\right) + o(\|(x_1,x_2)\|)
$$
for smooth functions $a_{ij}(x_1,x_2)$ with $a_{11}(0,0)+a_{22}(0,0)>0$ because the divergence of $V$ is positive at $p$.
Hence the eigenvalues of $A= \left((a_{ij}(p))_{i,j}\right)$ are either both real and at least one of them is positive or both eigenvalues are complex with positive real part. 

Unless $0$ is an eigenvalue of $A$ the singularity is non-degenerate and the index depends only on the sign of $\mathrm{det}(A)$. If $0$ is an eigenvalue of $A$, then $p$ is degenerate and by the center manifold theorem (eg. Theorem 3.2.1 in \cite{guck}) there is a $1$-dimensional unstable manifold (uniquely defined and of class $C^r$) tangent to the eigenspace of the non-vanishing eigenvalue and  a  center manifold $Z$ (not necessarily unique and only of class $C^{r-1}$) tangent to the kernel of $A$. Both submanifolds are invariant under the flow of $V$. 

The index of the singularity is now completely determined by the nature of the isolated zero at $p$ of the restriction of $V$ to $Z$. If $p$ is an attractive respectively repelling singularity of $V\eing{Z}$ then the index of $p$ is $-1$ respectively $1$. If the singularity is attractive on one side while it is repelling on the other, then the index of $p$ is $0$.    

This also shows that an isolated singularity with index $\pm 1$ of the characteristic foliation of a contact structure is topologically equivalent to a non-degenerate singularity with the same index. If the index is $0$, then the unstable manifold of the singularity decomposes a neighbourhood of $p$ into two parts, one half space is filled with integral curves of $V$ whose $\alpha$-limit set is $p$ while the other half looks like the corresponding half space of a hyperbolic singularity and the center manifold is unique  on that side. 
\end{proof}

Let $p$ be a singularity of $\Sigma(\xi)$ and $U\subset\Sigma$ a neighbourhood of $p$ such that $d\alpha\eing{U}$ is an area form. Assume that $p$ is a positive singularity of index $-1$ or $0$. In other words $p$ has a stable leaf. Choose a point $x\in U$ on a stable leaf and $y\in U$ a point on the strong unstable manifold of $p$. We fix half-open intervals $\sigma_x$ respectively $\sigma_y$ containing $x$ respectively $y$ such that there are leaves of the characteristic foliation in $U$ connecting points in the interior of $\sigma_x$ to $\sigma_y$. The positive divergence of $p$ has consequences for map $\varphi$ from $\sigma_x$ to $\sigma_y$ defined by following the leaves of the characteristic foliation:

\begin{lem} \mlabel{l:very rep}
For all $K$ there is a neighbourhood of $x$ in $\sigma_x$ such that 
$$
\varphi'(q)>K
$$
for all $q\neq x$ in that neighbourhood. 
\end{lem}

\begin{proof}
Let $X$ be the vector field satisfying $i_Xd\beta=\beta$ where $\alpha=dt+\beta$. Then the time-$t$-flow $\psi_t$ expands $\beta$ exponentially in $t$, i.e. $\psi_t^*\beta=e^t\beta$. As $q$ approaches $x$ the time the flow takes to move $q$ to $\varphi(q)$ goes to infinity because $p$ is a singularity. This implies the claim.
\end{proof}

\subsubsection{Convexity} \mlabel{s:convex}

In this section, we review the notion of convexity. The material presented here was developed by E.~Giroux in \cite{gi-conv}. Since the notion of convexity is standard in contact topology by now we will be very brief. 

\begin{defn} \mlabel{d:convex}
Let $\Sigma\subset (M,\xi)$ be an oriented surface in a contact manifold such that $\partial \Sigma$ is Legendrian. Then $\Sigma$ is {\em convex} if there is a contact vector field transverse to $\Sigma$. 
\end{defn}
Building on \cite{peix} Giroux showed that convexity can be achieved by  $C^\infty$-small perturbations of $\Sigma$ when this surface is closed. If $\Sigma$ has Legendrian boundary, then according to \cite{honda1} the same statement holds (at least for $C^0$-small perturbations fixing the boundary) if the twisting number of $\xi$ along $\partial \Sigma$ is not positive.   When $\Sigma$ convex, then a lot of information about the contact structure near $\Sigma$ is contained the {\em dividing set} $\Gamma$. In order to define it we fix a contact vector field $X$ transverse to $\Sigma$. Then 
$$
\Gamma = \{x\in \Sigma\,|\,X\in\xi_x\}.
$$
It turns out that $\Gamma$ is always a submanifold transverse to $\Sigma(\xi)$ whose isotopy type does not depend on the choice of $X$. Moreover, whether or not a surface in convex can be determined using only the characteristic foliation on $\Sigma$.

\begin{lem}[\cite{gi-conv}] \mlabel{l:why dividing set}
$\Sigma$ is convex if and only if there is a decomposition of $\Sigma$ into two subsurfaces $\Sigma^+,\Sigma^-$ with boundary  such that the boundary $\partial \Sigma^+=\partial \Sigma^-$  which is not part of $\partial \Sigma$ is transverse to $\Sigma(\xi)$  and there are defining forms for the singular foliation $\alpha_+$ on $\Sigma^+$ and $\alpha_-$ on $\Sigma^-$ such that $d\alpha_+>0$ and $d\alpha_-<0$.  

In this case the dividing set is isotopic to the closure of the parts of $\partial\Sigma_{\pm}$ which are not contained in $\partial\Sigma$. 
\end{lem}
In other words, the dividing set of a convex surface separates the surface into two domains such that the characteristic foliation on each part is tangent to a Liouville vector field associated to an exact area form. 

Given a closed convex surface one can compute the evaluation of the Euler class on that surface as follows
\begin{equation} \label{e:euler}
\langle e(\xi), [\Sigma]\rangle = \chi(\Sigma^+)-\chi(\Sigma^-),
\end{equation}
a fact which is attributed to Kanda \cite{ka}. The following lemma shows that characteristic foliations on convex surfaces can be manipulated effectively. 

\begin{defn}
Let $\Sigma$ be a compact oriented surface and $\Gamma$ a collection of pairwise disjoint simple closed curves and arcs which are transverse to the boundary separating $\Sigma$ into two surfaces with boundary $\Sigma^+,\Sigma^-$. A singular foliation $\GG$ on $\Sigma$ is {\em adapted} to $\Gamma$ (or $\GG$ is {\em divided} by $\Gamma$) if
\begin{itemize}
\item[(i)] $\GG$ is transverse to $\Gamma$, the boundary of $\Sigma$ consists of leaves and singularities of $\GG$ and 
\item[(ii)] there are defining forms $\alpha_\pm$ for $\GG$ on $\Sigma^\pm$ such that $d\alpha_+>0$ and $d\alpha_-<0$. 
\end{itemize}
\end{defn}

\begin{lem}[\cite{gi-conv}] \mlabel{l:giroux flex}
Let $\Sigma\subset(M,\xi)$ be a convex surface, $X$ a transverse contact vector field, $\Gamma$ the associated dividing set and $\GG$ a singular foliation adapted to $\Gamma$. Then there is an isotopy  $\varphi_s: \Sigma\lra M, s\in[0,1],$ such that $\varphi_0$ is the inclusion and $\varphi_{1*}(\GG)$ is the characteristic foliation on $\varphi_1(\Sigma)$. Moreover, $\varphi_s(\Sigma)$ is transverse to $X$ for all $s$ and the characteristic foliation on $\varphi_s(\Sigma)$ is divided by 
$$
x\in \left\{\varphi_s(\Sigma)\cap \left\{x\in M \big| X\in\xi(x)\right\}\right\}
$$
\end{lem}
A somewhat stronger version of this statement is \lemref{l:convex families}. An immediate consequence of this lemma is the Legendrian realization principle.
\begin{lem}[\cite{honda1}] \mlabel{l:LeRP}
Let $\Sigma\subset(M,\xi)$ be a convex surface, $\Gamma$ a dividing set and $C\subset \Sigma$ a simple closed curve transverse to $\Gamma$ such that every connected component of $\overline{\Sigma\setminus C}$ meets $\Gamma$. Then there is an isotopy as in \lemref{l:giroux flex} such that $\varphi_1(C)$ is a Legendrian curve in $\varphi_1(\Sigma)$.
\end{lem} 

A basic tool for controlled modifications of the dividing set on a given convex surface used in particular in the work of K.~Honda and J.~Etnyre is the attachment of bypasses. 
\begin{defn}[\cite{honda1}] \mlabel{d:bypass}
A {\em bypass} for the convex surface $\Sigma\subset(M,\xi)$ is an oriented half-disc $D$ which is embedded (except that the two corners of $D$ may coincide) with the following properties. 
\begin{itemize}
\item[(i)] $\partial D=\gamma_1\cup\gamma_2$ is the union of two smooth Legendrian arcs such that $\gamma_1$ is contained in $\Sigma$ and intersects the dividing set $\Gamma$ of $\Sigma$ transversely in exactly three (or two) points, two of these intersection points are the endpoints $\gamma_1$ (or the two endpoints of $\gamma_1$ coincide). The bypass is called {\em singular} if the two endpoints of $\gamma_1$ coincide. 
\item[(ii)] The interior of $\gamma_2$ is disjoint from $\Sigma$.
\item[(iii)] All singular points of $D(\xi)$ along $\gamma_2$ are positive. Apart from $\gamma_1\cap\gamma_2$   there is exactly one more singularity of $D(\xi)$ on $\gamma_1$. It is negative elliptic. 
\end{itemize}
\end{defn}
One boundary component $\Sigma'$ of a neighbourhood of $D\cup\Sigma$ can be chosen such that $\Sigma'$ is convex, diffeomorphic to $\Sigma$ and the dividing set on $\Sigma'$ is obtained from the dividing set on $\Sigma$ by the operation shown in Figure~\ref{b:switch} on p.~\pageref{b:switch}. In that figure the dividing set is dashed and $\gamma_1$ is the diagonal arc in the left-most figure. For more information about bypasses and their applications we refer the reader to \cite{gs,honda1,honda2,hkm} and the references therein.

\subsubsection{Basins of attractive orbits}

The following terminology is from \cite{el20}. It will be used in the proof of the pre-Lagrangian extension lemma in \secref{s:prelag extend}.  
\begin{defn}\mlabel{d:legpoly}
A {\em Legendrian polygon} $(Q,V,\alpha)$ on an oriented surface in a contact manifold $(M,\xi)$ is a smooth immersion 
$$
\alpha : Q \setminus V \lra \Sigma
$$
such that $Q$ is an oriented surface with piecewise smooth boundary, $V$ is a finite set contained in $\partial Q$, $\alpha$ is an orientation preserving embedding on the interior of $Q$, and each segment of $\partial Q\setminus V$ is mapped to a Legendrian arc. Smooth pieces of $\partial Q$ are mapped to smooth Legendrian curves of $F(\xi)$. 

For $v\in V$ the two segments of $\partial Q$ get mapped to two Legendrian curves with the same $\alpha$-/$\omega$-limit set $\gamma_v$ and $\gamma_v$ is not a singularity of $F(\xi)$. The elements of $V$ are called {\em virtual vertices}.

The preimage of a singularity is a {\em pseudovertex} respectively {\em corner} if the singularity has index $0$ or $-1$ and a neighbourhood (in $\partial Q$) of the preimage is mapped to stable leaves of the singularity respectively to a stable leaf and an unstable leaf. 
\end{defn}

Let $\beta\subset\Sigma$ be a non-degenerate attractive closed orbit of $\Sigma(\xi)$. The following definitions (except the notions upper/lower) and lemmas also apply when $\beta$ is a piecewise smooth closed curve consisting of negative singularities and unstable leaves of negative singularities. Fix a closed curve $\sigma$ transverse to $\Sigma(\xi)$ in a small tubular neighbourhood of $\beta$ such that the $\omega$-limit set of every leaf of $\Sigma(\xi)$ intersecting $\sigma$ is contained in $\beta$. 
\begin{defn}
Let $B_\beta$ be the union of all leaves of $\Sigma_t(\xi)$ which intersect $\sigma$. We will call this set the {\em basin} of $\beta$. If $\sigma$ lies on the side of $\beta$ in $U$ determined by the coorientation of $\xi$ respectively on the opposite side we speak of the {\em upper} respectively {\em lower} basin. 
\end{defn}
The proof of the following lemma is completely analogous to the proof of Lemma 3.2 in \cite{rigflex}. The assumptions made before the corresponding Lemma 3.4 in \cite{rigflex} about the non-degeneracy of singular points are not necessary  but they were made in order to facilitate the presentation (see also \lemref{l:char fol sing}).

\begin{lem} \mlabel{l:legpoly exist} Let $B_\beta$ be a basin of an attractive closed leaf of $\Sigma(\xi)$. 
There is a Legendrian polygon $(Q,V,\alpha)$ on $\Sigma$ with $Q=[0,1]\times S^1$   such that 
\begin{itemize}
\item[(i)] $\alpha(\{0\}\times S^1)=\beta$ and
\item[(ii)] $\alpha(Q\setminus V)\cup\bigcup_{v\in V}\gamma_v=\overline{B}_\beta$.  
\end{itemize}
\end{lem}
We say that the Legendrian polygon {\em covers} the basin.

\subsection{Properties of contact structures and foliations} \mlabel{s:tight}

In this section we summarize definitions concerning geometric properties of foliations, contact structures and confoliations.
\begin{defn}\mlabel{d:taut fol}
A foliation $\FF$ is {\em taut} if for every leaf $L$ of $\FF$ there is a closed curve transverse to $\FF$ which intersects $L$. A {\em Reeb component} is a foliation on $S^1\times D^2$ such that the boundary is a leaf. An oriented foliation of $S^1\times[0,1]$ by lines is a two-dimensional Reeb component if the boundary curves are leaves which are oriented in opposite directions. A foliation is {\em minimal} if every leaf is dense.  \end{defn}
In \figref{b:equiv} in \secref{s:tori} one can see a pair of two-dimensional Reeb components. 

The following definition of tight confoliations is an extension of the usual definition of tightness for contact structures as introduced in \cite{confol}.
\begin{defn} \mlabel{d:tight}
An {\em overtwisted disc} in a confoliated manifold is an embedded closed disc $D$ so that $\partial D$ is Legendrian and all singularities of $D(\xi)$ on $\partial D$ have the same sign. A contact structure is {\em tight} if there is no overtwisted disc, otherwise it is {\em overtwisted}. A contact structure is {\em universally tight} if the pull back of $\xi$ to the universal covering $\widetilde{M}$ of $M$ is tight. 

A confoliation is called {\em tight} if for every overtwisted disc $D$ there is an integral $D'$ of $\xi$ with the following properties
\begin{itemize}
\item[(i)] $D'$ is a disc, $\partial D'=\partial D$, and
\item[(ii)] $\langle e(\xi),[D\cup D']\rangle = 0$ where $e(\xi)\in H^2(M;\Z)$ is the Euler class of $\xi$ and $D, D'$ are oriented in such a way that their union is also oriented. 
\end{itemize} 
\end{defn}

This definition interpolates between tight contact structures (in this situation there are no integral discs $D'$) and foliations without Reeb-components. Contact structures (and foliations without Reeb components) satisfy the Thurston-Bennequin inequalities. In order to state them let $\Sigma$ be an oriented compact embedded surface (whose boundary is positively transverse to the plane field $\xi$). Then
\begin{align} \label{e:tb}
\begin{split}
\langle e(\xi),[\Sigma]\rangle = 0 & \textrm{ if } \Sigma\simeq S^2 \\
|\langle e(\xi),[\Sigma]\rangle| \le -\chi(\Sigma) & \textrm{ if } \Sigma\not\simeq S^2 \textrm{ is closed} \\
-\langle e(\xi),[\Sigma]\rangle \le -\chi(\Sigma) & \textrm{ if }\partial\Sigma\neq\emptyset
\end{split}
\end{align}
where the left  hand side of the least inequality is the obstruction for the extension of the trivialization of $\xi$ along the boundary of $\Sigma$ given by $\Sigma(\xi)$ to the interior. As shown in \cite{rigflex} the Thurston-Bennequin inequalities for tight confoliations do not hold in general while they are always satisfied for s-tight contact structures.

\begin{defn} \mlabel{d:otstar} Let $\xi$ be a confoliation on $M$. 
An {\em overtwisted star}  on a compact surface $\Sigma$ is a Legendrian polygon $\alpha : D^2\setminus V \lra \Sigma\subset M$ where $V\subset\partial D$ is a finite set of points and $F$ is a closed embedded surface so that
\begin{itemize} 
\item $\alpha(\partial D\setminus V)$ is a union of Legendrian arcs so that for all $v\in V$ the image of the two arcs approaching $v$ on $\partial D$ have the same $\omega$-limit set $\gamma_v$ when the arcs are oriented towards $v$, and $\gamma_v\cap H(\xi)=\emptyset$.   
\item All singularities of $\Sigma(\xi)$ on $\alpha(\partial D\setminus V)$ have the same sign, and this sign is opposite to all singularities in $\alpha(\ring{D})$.
\end{itemize}
A tight confoliation without overtwisted stars is {\em s-tight}. 
\end{defn}
Tight contact structures are considered to be much more interesting than overtwisted contact structures because of the following classification result of Y.~Eliashberg \cite{elot}. 
\begin{thm} \mlabel{t:ot class} 
For $\xi_0$ a contact structure on a closed manifold $M$ with an overtwisted disc $D$ let $\mathrm{Cont}(M,D,\xi_0)$ be the space of contact structures which have $D$ as an overtwisted disc and are homotopic as plane fields to $\xi_0$ relative to $D$. 

The space $\mathrm{Cont}(M,D,\xi_0)$ is weakly contractible. In particular, two overtwisted contact structures are isotopic if and only if they are homotopic as plane fields. 
\end{thm}
Moreover, there are interesting  analogies between taut foliations respectively foliations with Reeb components and symplectically fillable contact structures respectively overtwisted contact structures. 

It is easy to show that tightness implies s-tightness in the context of the following theorem. 
\begin{thm}[Eliashberg, Thurston \cite{confol}] \mlabel{t:tight connection}
Let $\xi$ be a confoliation on $\R^3$ which is transverse to the fibers of the projection $\R^3\lra\R^2$ and complete as a connection of this bundle. Then $\xi$ is tight (and s-tight). 
\end{thm}

\begin{ex} \mlabel{ex:tight connection}
The $1$-form $dz+f(x,y,z)dy$ defines a contact structure respectively a confoliation if $\frac{\partial f}{\partial x}>0$ respectively $\frac{\partial f}{\partial x}\ge 0$. A simple case when $\xi$ is a complete connection of the bundle 
\begin{align*}
\R^3\ &\lra \R^2 \\
(x,y,z) & \lmt (x,y)
\end{align*} 
is when $f$ is an affine or bounded function. 
\end{ex}
Usually, tightness of a contact structure is shown by either embedding the contact manifold into a contact structure which is already known to be tight, or one uses symplectic fillings or gluing theorems (eg. from \cite{col-glue}).

\begin{defn}
Let $M$ be a closed oriented manifold and $\xi$ a confoliation. A symplectic manifold $(X,\omega)$ is a {\em weak symplectic filling} of $(M,\xi)$ if
\begin{itemize}
\item[(i)] $M=\partial X$ as oriented manifolds (where $X$ is oriented by $\omega\ww\omega$ and the outward normal first convention is used to orient the boundary) and
\item[(ii)] $\omega\eing{\xi}$ is a symplectic vector bundle. 
\end{itemize}
\end{defn}
The following theorem is due to M.~Gromov (for the case when $\xi$ is a confoliation see \cite{rigflex}). 
\begin{thm} \mlabel{t:symp fill}
If a contact manifold $(M,\xi)$ admits a weak symplectic filling, then it is tight. 
\end{thm}
This criterion is used in \cite{confol} to show the following result about contact structure approximating taut foliations.
\begin{thm} \mlabel{t:taut tight}
Every contact structure which is sufficiently $C^0$-close to a taut foliation is universally tight. 
\end{thm}
In general a symplectically fillable contact structure does not have to be universally tight but at least there is a very efficient criterion to decide whether or not there is a universally tight neighbourhood of a convex surface. 
\begin{lem}[Giroux's criterion] \mlabel{l:Giroux crit}
Let $\Sigma$ be a convex surface in a contact manifold $(M,\xi)$ and $\Gamma$ its dividing set. If $\Sigma\simeq S^2$, then we require that $\Gamma$ is connected, otherwise we ask that no component of $\Gamma$ bounds a  disc in $\Sigma$.

Then $\Sigma$ has a neighbourhood so that the restriction of $\xi$ to that neighbourhood is universally tight. 
\end{lem}
This lemma applies to the case when $\Sigma$ is a sphere in a contact manifold such that $\Sigma(\xi)$ has exactly two singular points and all leaves of $\Sigma(\xi)$ connect the two singular points. Such a sphere is automatically convex. The following corollary of Giroux's criterion can be found in \cite{hkm}.

\begin{cor} \mlabel{c:extremal}
Let $\xi$ be a tight oriented contact structure near an oriented closed surface $\Sigma$ with positive genus such that $\langle e(\xi),[\Sigma] \rangle = \pm \chi(\Sigma)$ and $\Sigma$ is convex. Then $\Sigma^-$ or $\Sigma^+$ is a nonempty union of annuli.
\end{cor}

\begin{proof} 
No component of the dividing set of $\Sigma$ bounds a smooth disc. Hence all components of $\Sigma^+$ and $\Sigma^-$ have non-positive Euler characteristic. The claim is now an easy consequence of $\langle e(\xi),[\Sigma]\rangle=\chi(\Sigma^+)-\chi(\Sigma^-)$ and $\chi(\Sigma)=\chi(\Sigma^+)+\chi(\Sigma^-)$. 
\end{proof}
According to the Thurston-Bennequin inequalities \eqref{e:tb} the situation considered in the corollary corresponds to the maximal possible absolute value of the evaluation of the Euler class on a closed surface in a tight contact manifold or a foliation without Reeb components.  A contact structure $\xi$ on $\Sigma\times[0,1]$ will be called {\em extremal} if $|\langle e(\xi),[\Sigma]\rangle|=-\chi(\Sigma)$ where $\Sigma$ is an oriented closed surface. 

There is another invariant associated to contact structures which can distinguish diffeomorphism classes of tight contact structures.
\begin{defn} \mlabel{d:torsion}
Let $(M,\xi)$ be a contact manifold. For every positive integer $n$ we consider the contact structures 
$$
\xi_n=\ker(\cos(2\pi nt)dx_1-\sin(2\pi nt)dx_2)
$$ 
on $T^2\times[0,1]$. The {\em Giroux torsion} of $(M,\xi)$ is 
$$
\sup\left\{ m\left|\begin{array}{l} m\in\mathbb{N}^+ \textrm{ and there is a contact embedding}   \\ (T^2\times[0,1],\xi_m)\lra (M,\xi)\textrm{ or } m=0\end{array}\right\}\right. .
$$
\end{defn}
In the previous definition one can specify the isotopy class of the embedding of $T=T\times\{0\}$. If such an embedding is specified (eg. by a torus leaf of a foliation) then we will sometimes refer to the Giroux torsion along $T$. The contact structures $\xi_k$ from \exref{ex:T3} have Giroux torsion $k-1$ and are hence distinguished by this invariant.

\subsection{Classification results for tight contact structures} \mlabel{s:tight class}

There are several classification results for tight contact structures up to isotopy relative to the boundary that we shall use. They concern $B^3, S^1\times D^2, T^2\times[0,1]$ and $\Sigma\times[0,1]$ where $\Sigma$ is a surface with genus $g\ge 2$. The following result is fundamental.

\begin{thm}[Eliashberg, \cite{el20, gi-bour}]\mlabel{t:fill ball}
The space of positive tight contact structures on $(B^3,\partial B^3)$ which induce a fixed characteristic foliation $\GG_0$ on $\partial D$ is weakly contractible. It is not empty if and only if $\GG_0$ admits a taming function. 
\end{thm}
For the definition of taming functions and their construction we refer to \cite{el20} and \cite{rigflex} (admittedly, there are no proofs with parameters in these references).  All we will need to know is that a taming function increases along leaves of the characteristic foliation on $\partial B^3$ and it exists when the contact structure is tight. However, note that by \lemref{l:char fol sing} there are taming functions on neighbourhoods of a degenerate isolated singularities of characteristic foliations.

Usually, \thmref{t:fill ball} is stated in a weaker form covering only the connectedness of the space of contact structures. The version given above is implicitly contained in Theorem 2.4.2 of \cite{el20} and stated in \cite{gi-bour}. 

The proof of \thmref{t:fill ball} also shows that for a family of characteristic foliations $\GG_s$ on $\partial D$ there is a family of contact structures $\xi_s$ on $B^3$ such that $(\partial B^3)(\xi_s)=\GG_s$ for all $s$ provided that there is a  family of taming functions for the foliation $\GG_s$. 

\subsubsection{Contact structures on solid and thickened tori} 

A lot of information about the classification of tight contact structures on the solid torus up to isotopy can be found in \cite{honda1, gi-bif}. We will only need the following simple case but we give a parametric version. 
 
\begin{thm} \mlabel{t:solid torus}
Let $\xi$ be a tight contact structure on $N=D^2\times S^1$ with convex boundary such that the dividing set has exactly two connected components and the intersection number each component with a meridional disc is $\pm 1$. 
 
Then the space of positive tight contact structures  on $N$ which coincide with $\xi$ near $\partial N$  is weakly contractible. 
\end{thm}

\begin{proof}
By \lemref{l:giroux flex} we may assume that $\partial N(\xi)$ has the following properties.
\begin{itemize}
\item[(i)] There are two canceling pairs of singularities, one of them is negative the other one is positive. There are no closed orbits and no connections between hyperbolic singularities. 
\item[(ii)] Both unstable leaves of the positive hyperbolic singularity are connected to the negative elliptic singularity and their union bounds a meridional disc $D$ in $N$. The union of both stable leaves of the negative hyperbolic singularities also bound a meridional disc $D'$ and $D'(\xi)$ is convex with respect to $\xi$.  
\end{itemize}
Let $S$ be a compact manifold and let $\xi_s, s\in S$, be a smooth family of tight contact structures on $N$ with $\xi_s=\xi$ near $\partial N$.  We will construct a family of contact structure $\xi'_s$ with $\xi_s'=\xi$ near $\partial N$ such that the characteristic foliation on $D'$ is constant while $D(\xi_s)=D(\xi'_s)$ for all $s\in S$. 

The contact structure $\xi$ naturally extends to a slightly thicker torus $N'$. We choose two smooth embedded spheres $\Sigma_1,\Sigma_2$ such that 
\begin{itemize}
\item $\Sigma_i$ contains a neighbourhood of $\gamma_i$ in $\partial N$ for $i=1,2$,
\item $\Sigma_1\cap \Sigma_2=D\cup D'$, and $\Sigma_i\setminus (D\cup D')$ does not meet the interior of $N$. 
\end{itemize}
The properties of $\xi$ near $\partial N$ imply that there are two curves $\gamma_1,\gamma_2$ on $\partial N$ transverse to $\partial N(\xi)$ separating $\partial D$ from $\partial D'$. Therefore the question whether or not a given singular foliation on $\Sigma_1$ admits a taming function depends only on the characteristic foliations on the discs $\Sigma_1\setminus\gamma_1$. Hence the singular foliations on $\Sigma_i$ given by $\xi$ on $\Sigma_i\setminus D$  and by $D(\xi_s)$ on $D$ admit taming functions and we obtain tight contact structures on the balls bounded by these spheres and they form a family tight contact structure $\xi_s'$ on $N$ such that $D(\xi'_s)=D(\xi)$. 

Since $\partial N\cup D$ bounds a ball \thmref{t:fill ball} implies that the families $\xi_s$ and $\xi'_s$ can be deformed into each other. Then \thmref{t:fill ball} applied to the ball bounded by $\partial N\cup D'$  and the family of contact structures $\xi_s'$ implies the claim of the lemma. 
\end{proof}

The theorem below contains the information from Theorem 4.4 of \cite{gi-bif} about the classification of tight contact structures on the thickened torus $T^2\times[-1,1]$ we are going to use. (The theorem is stated in a way which can be easily translated to the terminology developed in \cite{gi-bif}. This terminology is explained in \secref{s:movies} below.) 

\begin{thm}[Giroux, \cite{gi-bif}]  \mlabel{t:tight on TxI}
Let $\GG_{\pm 1}$ be two foliations on $T^2$ which have exactly $n_{\pm 1}>0$ non-degenerate attractive closed leaves such that there is no Reeb component of dimension $2$ and all closed leaves are non-degenerate. 

For a given integer $k\ge 0$ there is a contact structure $\xi$ with Giroux torsion $k$ on $T^2\times[-1,1]$, unique up to isotopy, such that 
\begin{itemize}
\item[(i)] $T_{\pm 1}(\xi)=\GG_{\pm 1}$,
\item[(ii)] for all $t\in[-1,1]$ all singularities of the characteristic foliation are positive, and 
\item[(iii)] there is an embedded torus $T'\subset T^2\times(-1,1)$ isotopic to $T_0$  such that $T'(\xi)$ is a foliation by closed leaves. 
\end{itemize}
\end{thm}
We will see later in \secref{s:transverse contact} that a contact structure satisfying the assumptions (i)-(iii) of the theorem is automatically universally tight. \thmref{t:tight on TxI} is then obtained from Theorem 4.4 in \cite{gi-bif} using \remref{r:sheet prop}. 

\subsubsection{Contact structures on $\Sigma\times[-1,1]$ with $g(\Sigma)\ge 2$} \mlabel{s:Sigma class}

Before we can state the main classification result about tight contact structures on $N=\Sigma\times[-1,1]$ with $|\langle e(\xi),[\Sigma] \rangle| =-\chi(\Sigma)$ and convex boundary, we need to recall the notion of the {\em relative Euler class } from \cite{hkm}. We will choose the orientation of $\xi$ such that $\langle e(\xi),[\Sigma]\rangle=\chi(\Sigma)$, i.e. $\Sigma_t^-$ is a non-empty union of disjoint annuli whenever $\Sigma_t$ is convex. The dividing set of $\Sigma_{\pm 1}$ will be denoted by $\Gamma_{\pm 1}$. 

Let $\beta\subset \Sigma_i,i=\pm1,$ be a closed curve. We say that $\beta$ is {\em primped} if the following conditions hold.
\begin{itemize}
\item $\beta$ is non-isolating in $\Sigma_i$, i.e. $\beta$ is transverse to $\Gamma_i$, and every boundary component of $\overline{\Sigma_i\setminus (\beta\cup\Gamma_i)}$ meets $\Gamma_i$.    
\item The intersection $\beta\cap\Sigma_i^-$ consists only of arcs each of which does not separate an annulus in $\Sigma_i^-$ into two connected components.
\end{itemize}
Clearly  every curve is isotopic to a primped curve. According to the Legendrian realization principle (\lemref{l:LeRP}) there is a $C^0$-small isotopy of $\Sigma_i$ through convex surfaces so that $\beta$ is a Legendrian curve on the isotoped surface. In order to define the relative Euler class $\widetilde{e}(\xi)$ on $[\beta\times I]\in H_2(N,\partial N,\Z)$ we proceed as follows: Isotope $\beta\times\{\pm 1\}$ to primped Legendrian curves in $\Sigma_{\pm 1}$. Then consider an annulus $A$ bounded by the two Legendrian curves (since $\Sigma$ is not $T^2$ this annulus is uniquely determined up to isotopy relative to the boundary).  After a small perturbation we may assume that the annulus is convex. In analogy to \eqref{e:euler} one defines
\begin{equation} \label{e:rel euler}
\langle\widetilde{e}(\xi),[\beta\times I] \rangle:=\chi(A^+)-\chi(A^-). 
\end{equation}
The following is shown in \cite{hkm}:

\begin{prop} \mlabel{p:rel euler}
The relative Euler class $\widetilde{e}(\xi) \in H^2(N,\partial N;\Z)$ is well defined and extends the Euler class $e(\xi)$ viewed as homomorphism $e(\xi) : H_2(N;\Z) \lra\Z$ to $H_2(N,\partial N;\Z)$.
\end{prop}

Now we can state Theorem 1.1 of \cite{hkm}:
\begin{thm}[Honda, Kazez, Mati{\'c}, \cite{hkm}] \mlabel{t:Sigma class}
Let $\Sigma$ be a closed oriented surface of genus $g\ge 2$ and $\GG_{\pm 1}$ singular foliations on $\Sigma\times \{\pm 1\}$ so that $\GG_{\pm 1}$ is adapted to a dividing set $\Gamma_{\pm 1}$ consisting of exactly two non-separating closed curves bounding an annulus so that
$$
\chi(\Sigma_{-1}^+)-\chi(\Sigma_{-1}^-)=\chi(\Sigma_1^+)-\chi(\Sigma_1^-).
$$

If $\Gamma_{-1}$ and $\Gamma_1$ are not isotopic, then there are exactly four isotopy classes of tight contact structures $\xi$ on $N$ so that $\Sigma_{\pm 1}(\xi)=\GG_{\pm 1}$. They are distinguished by the relative Euler class $\widetilde{e}(\xi)$ which takes the values
$$
\mathrm{PD}(\widetilde{e}(\xi))=\pm\gamma_{-1}\pm\gamma_{+1}\in H_1(N,\Z)
$$
where $\gamma_{\pm 1}$ is a connected component of $\Gamma_{\pm 1}$. 

If $\Gamma_{-1}$ and $\Gamma_1$ are isotopic, then there are exactly five isotopy classes of tight contact structures on $N$ inducing the given characteristic foliation on $\partial N$. Three of these contact structures satisfy $\mathrm{PD}(\widetilde{e}(\xi))=0$ while the two remaining isotopy classes satisfy 
$$
\mathrm{PD}(\widetilde{e}(\xi))= \pm 2\gamma_{-1}=\pm 2\gamma_1 \in H_1(N,\Z).
$$ 
In all of the above cases the tight contact structures are universally tight.
\end{thm}

We still need to explain how to distinguish tight contact structures with $\mathrm{PD}(\widetilde{e}(\xi))=0$. This is done by embedding properties. The following definition of a basic slice is not quite the same as Definition 5.12 in \cite{hkm} but the two definitions are equivalent by \thmref{t:Sigma class} (and the set-up used in \cite{hkm} to analyze what is called the Base-Case on p.~323 of \cite{hkm}). 

\begin{defn} \mlabel{d:basic slice}
A {\em basic slice} is a tight contact structure on $\Sigma\times[-1,1]$ such that
\begin{itemize}
\item[(i)] $\Sigma_{-1}$ and $\Sigma_1$ satisfy all assumptions of \thmref{t:Sigma class},
\item[(ii)] $\gamma_{-1}$ and $\gamma_1$ intersect exactly once, and
\item[(iii)] $\mathrm{PD}(\widetilde{e}(\xi))=\pm(\gamma_{-1}-\gamma_1)$ when $\gamma_{-1},\gamma_1$ are oriented so that $\gamma_1\cdot\gamma_{-1} =1$. 
\end{itemize} 
As in \cite{hkm} we denote a basic slice by $\llbracket \gamma_{-1},\gamma_1;\pm (\gamma_{-1}-\gamma_1) \rrbracket$ depending on the value of the relative Euler class.
\end{defn}
Note the definition of a basic slice is independent from the orientation of $\gamma_{-1},\gamma_1$  satisfying $\gamma_1\cdot\gamma_{-1}=1$. 

\begin{prop}[Honda, Kazez, Mati{\'c}, \cite{hkm}] \mlabel{p:vanishing rel Euler}
Let $\xi$ be a tight contact structure on $N$ such that $\Gamma_{-1}=\Gamma_1=2\gamma$. Then $\xi$ is isotopic to a vertically invariant contact structure if and only if there is no embedding of a basic slice $\llbracket \gamma,\gamma';\pm(\gamma-\gamma')\rrbracket$. There are two tight contact structures $\xi_+,\xi_-$ such that there are contact embeddings
\begin{align*}
\llbracket \gamma,\gamma';+(\gamma-\gamma') \rrbracket  & \lra (N,\xi_+) \\
\llbracket \gamma,\gamma';-(\gamma-\gamma') \rrbracket  & \lra (N,\xi_-)
\end{align*}
mapping the boundary component $\Sigma_0$ of the basic slice to $\Sigma_{-1}$ while there are no contact embeddings 
\begin{align*}
\llbracket \gamma,\gamma';-(\gamma-\gamma') \rrbracket  & \lra (N,\xi_+) \\
\llbracket \gamma,\gamma';+(\gamma-\gamma') \rrbracket  & \lra (N,\xi_-)
\end{align*}
with the same property. 
\end{prop}

The relative Euler class behaves well when $\Sigma\times [-1,+1]$ is decomposed  along $\Sigma_0$ provided that the contact structure on $\Sigma\times[-1,1]$ is tight and $\Sigma_{\pm 1}, \Sigma_0$ are convex such that the dividing set consists of two connected non-separating curves. This is best expressed as follows 
\begin{equation} \label{e:sum}
PD\left(\widetilde{e}\left(\xi\eing{\Sigma\times[-1,0]}\right)\right)+PD\left(\widetilde{e}\left(\xi\eing{\Sigma\times[0,1]}\right)\right) = PD\left(\widetilde{e}\left(\xi\eing{\Sigma\times[-1,1]}\right)\right).
\end{equation}
This is part of Theorem 6.1 of \cite{hkm}.


\section{Movies and their properties} \mlabel{s:movies main}

In this section we explain some of the material in Giroux's work \cite{gi-bif} about families of characteristic foliations of positive contact structures on $\Sigma\times[-1,1]$ where $\Sigma$ is a closed oriented surface. Parts of this material can also be found in \cite{gei}. Our main omission is that we do not discuss the turnaround locus ({\em lieu de retournement} in \cite{gi-bif}).  The results from \secref{s:movies} and \secref{s:prop sheets} were used by E.~Giroux to obtain a classification of tight contact structures on torus bundles and lens spaces. Some of the results proved in \secref{s:movies} are probably folklore (like \lemref{l:compactify by degen}) but we did not find a good reference.  We will apply some of the techniques of Giroux to contact structures on $\Sigma\times[-1,1]$ with convex boundary: The main result from \secref{s:prelag extend} is new and will be used in combination with \thmref{t:Sigma class} in the proof of \thmref{t:unique}.

\subsection{Movies associated to contact structures} \mlabel{s:movies}

Let $\xi$ be a contact structure and $\Sigma\subset M$ an embedded oriented surface. According to the (strong) Thom transversality theorem we may assume that all points of the characteristic foliation $\Sigma(\xi)=T\Sigma\cap\xi$ on $\Sigma$, i.e. points $p\in\Sigma$ where $T_p\Sigma=\xi(p)$, are isolated. This remains true for all surfaces appearing in compact finite-dimensional families. 

$\Sigma\times[-1,1]$ has the product orientation, so the contact orientation and the coorientation of the contact structure induces an orientation of $\Sigma$. 

If $\xi$ is a contact structure on $\Sigma\times[-1,1]$, then we obtain a family of singular foliations on $\Sigma_t=\Sigma\times\{t\}$. This family will be referred to as {\em movie} of $\xi$. The contact condition has implications for the singular foliations appearing in the movie, cf. \lemref{l:gi-birth-death} and \lemref{l:switch} below. Recall also, that the contact condition implies that the divergence at positive respectively negative singular points of  $\Sigma_t(\xi)$ is positive respectively negative. Although it is not clear which families of singular foliations are movies associated to a positive contact structure, there is the following uniqueness result (Lemma 2.1 of  \cite{gi-bif}).

\begin{lem} \mlabel{l:movie unique}
Let $\xi_0$ and $\xi_1$ be two positive contact structures on $\Sigma\times[-1,1]$ such that $\Sigma_t(\xi_0)=\Sigma_t(\xi_1)$ for all $t\in[-1,1]$. Then $\xi_0$ and $\xi_1$ are isotopic.
\end{lem} 
\begin{proof}
A $1$-form $\alpha=\lambda_t+u_tdt$ defines a positive contact structure if and only if 
\begin{equation} \label{e:contact cond}
u_t\,d\lambda_t+\lambda_t\ww(du_t-\dot{\lambda}_t)>0.
\end{equation}
The movie of a positive contact structure determines the family $\lambda_t$ up to multiplication with a nowhere vanishing function when we consider only $1$-forms $\lambda_t$ coming from a defining form of a contact structure. The set of functions $u_t$ satisfying \eqref{e:contact cond} for a given family of $1$-forms $\lambda_t$ is convex. Hence the lemma follows immediately from Gray's theorem. 
\end{proof}

There are a few situations when it is easy to show that a given family of singular foliations is the movie of a positive contact structure.

\begin{lem} \mlabel{l:convex ext}
Let $\GG_t, t\in[-1,1],$ be a family of singular foliations on $\Sigma$ such that there is a continuous family of curves $\Gamma_t$ dividing $\GG_t$, i.e. for every $t$ there is a smooth function $v_t$ on $\Sigma$ such that
\begin{equation} \label{e:dividing}
v_t\,d\lambda_t+\lambda_t\ww dv_t>0.
\end{equation}
Then there is a contact structure $\xi$ on $\Sigma\times[-1,1]$ such that $\Sigma_t(\xi)=\GG_t$. 
\end{lem}
\begin{proof} 
Notice that we made no assumption concerning the dependence of $v_t$ on $t$. But the set of functions $v_t$ satisfying \eqref{e:dividing} for a given family of $1$-forms $\lambda_t$ on $\Sigma$ is convex. By compactness of the interval we can assume that $v_t$ depends smoothly on $t$. Then for $k$ sufficiently large, the $1$-form 
$$
kv_tdt+\lambda_t 
$$
on $\Sigma\times[-1,1]$ is a contact form with the desired properties. 
\end{proof} 

The following lemma is the corresponding uniqueness result. 
\begin{lem}[Lemma 2.7 of \cite{gi-bif}] \mlabel{l:convex  families}
Let $\xi_0,\xi_1$ be two contact structures on $\Sigma\times[-1,1]$ which coincide near the boundary such that the characteristic foliations $\Sigma_t(\xi_0)$ and $\Sigma_t(\xi_1)$ are divided by $\Gamma_t$ where $\Gamma_t$ varies continuously with $t\in[-1,1]$.

Then $\xi_0$ and $\xi_1$ are isotopic relative to the boundary. 
\end{lem}

\begin{proof}
The contact structures $\xi_i,i=0,1$, are defined by $1$-forms $u_t^idt+ \lambda_t^i$. Since $\Gamma_t$ divides the characteristic foliations $\Sigma_t(\xi_i)$ there is a family of functions $v^i_t$ for $i=0,1$ such that 
\begin{itemize}
\item $0$ is a regular value of $v_t$ and $\Gamma_t=v_t^{-1}(0)$, and
\item $v_t^id\lambda_t+\lambda_t^i\ww dv_t^i>0$.
\end{itemize}
Since we can multiply $\lambda_t^1$ with the nowhere vanishing function $h=v_t^0/v_t^1$, we may assume that $v_t^0=v_t^1=v_t$. Then for every positive constant $k$ the family of $1$-forms 
$$
(su_t^i+(1-s)kv_t)dt+\lambda_t^i,\quad s\in[0,1], i=0,1
$$
on $\Sigma\times[-1,1]$ is a family of contact forms. Moreover, for $k$ sufficiently large the $1$-forms
$$
kv_tdt+s\lambda_t^i+(1-s)\lambda_t^i, \quad s\in[0,1]
$$
on $\Sigma\times[-1,1]$  is a family of contact forms. The conclusion follows again from Gray's theorem.
\end{proof}
We will need the following slight modification of \lemref{l:convex families} which gives a relative version of the previous lemmas.

\begin{lem}\mlabel{l:rel convex}
Let $\xi,\xi'$ be contact structures on $\Sigma\times [-1,1]$ so that there is a family of compact subsurfaces $F_t\subset\Sigma_t$ such that
\begin{itemize}
\item[(i)] $\Sigma_t(\xi)=\Sigma_t(\xi')$ outside of $F_t$,
\item[(ii)] $\partial F_t$ is transverse to $\Sigma_t(\xi)$, and
\item[(iii)] there are contact forms $\alpha,\alpha'$ defining $\xi,\xi'$ such that $d\alpha\eing{F_t}>0$ and $d\alpha'\eing{F_t}>0$.
\end{itemize}
Then $\xi$ and $\xi'$ are isotopic. 

Moreover, given a contact structure $\xi$ defined by $\alpha$, a family of domains $F_t$ with properties (ii), (iii) and a smooth family $\lambda_t$ of $1$-forms such that $d\lambda_t>0$ on $F_t$ with $\alpha\eing{F_t}=\lambda_t$ on a neighbourhood of the boundary of 
$$
\bigcup_{-1\le t\le 1}F_t
$$
there is a contact structure $\xi'$ on $\Sigma_t\times[-1,1]$ which coincides with $\xi$ outside of $\cup_tF_t$ whose characteristic foliation is defined by $\lambda_t$ inside of $F_t$. 
\end{lem}

\begin{proof}
We begin with the existence part. We may assume that there is a domain with boundary $F\subset \Sigma$ so that $F_t=F\times\{t\}$ and the characteristic foliations of $\xi$ near $\partial F_t$ are independent of $t$. Fix a collar $C$ of $\partial F$. The part of $\partial C$ in the complement of $F$ will be denoted by $\gamma_{out}$ and $\gamma_{in}=\partial C\setminus\gamma_{out}$. 

Since $\partial F_t$ is transverse to $\Sigma_t(\xi)$ we may choose $C$ so thin that every leaf of $C(\xi)$ connects two boundary components of $C$ and $\alpha\eing{C_t\cap F_t}=\lambda_t$. Without loss of generality  assume that $C_t(\xi)$ is constant. By (ii),(iii) the characteristic foliations are pointing out of $F_t$ along the boundary.

The angle between $\xi$ and $F_t$ along $\gamma_{out}$ is bounded away from $0$ by a constant $\nu$. Since $d\lambda_t>0$ the $1$-form $\lambda_t+kdt$ defines a positive contact structure on $F\times[-1,1]$. If $k$ is sufficiently large, then the angle between $F_t$ and $\xi'=\ker(\lambda_t+kdt)$ is smaller than $\nu$ along $\gamma_{in}$. Since $C(\xi)$ is a product foliation, the contact structure $\xi'$ defined inside of $F\times[-1,1]$ can be extended to a contact structure on $\Sigma\times[-1,1]$ by twisting along the leaves of $C(\xi)$.

The proof that the resulting contact structure $\xi'$ on $\Sigma\times[-1,1]$ is isotopic to $\xi$ is analogous to the proof of \lemref{l:convex families}.   
\end{proof}

Before we proceed with manipulations of movies we state the result of an explicit computation we will use later: Assume that $\xi$ is a contact structure on a family of annuli $A_t, t\in[0,1], $ defined by  $\alpha=dt+\lambda_t$  such that the characteristic foliation is transverse to the boundary of $A_t$ for all $t$ and pointing outwards. The contact condition implies that $d_A\lambda_t+\dot{\lambda}_t\ww\lambda_t$ is an area form (here $d_A$ is the exterior differential on the annuli) and we can choose $\lambda_t$ such that 
\begin{align*}
d_A\lambda_t >  0  & \textrm{ on the repulsive side of a closed leaf, } \\
d_A\lambda_t >  0  & \textrm{ along non-degenerate repulsive leaves, and } \\ 
d_A\lambda_t = 0   & \textrm{ along degenerate closed leaves.} 
\end{align*}
Let $X$ be a vector field  on $A\times[0,1]$ without zeroes tangent to the characteristic foliation but pointing in the opposite direction and let $\varphi_s$ be the flow of $X$ defined for $s\ge 0$. Set $i_Xd\lambda_t = h_X\lambda_t$, $H(s):=\exp\left(\int_0^s\varphi^*_\sigma h_X d\sigma\right)$ and $i_X(d_A\lambda_t+\dot{\lambda}_t\ww\lambda_t)=f_X\lambda_t$.  By the contact condition, $f_X$ is bounded away from $0$ and $f_X<0$. Also $h_X\le 0$ on the repulsive side of closed orbits. Then
\begin{equation} \label{e:evol-alpha}
\varphi_s^*\alpha = H(s)\left(\left(1-\frac{\int_0^s\varphi_\sigma^*f_Xd\sigma}{H(s)}\right) \lambda_t+dt\right). 
\end{equation}  
This shows that the contact structure defined $\varphi_s^*\alpha$  converges to the tangent planes of the annulus along all leaves of the characteristic foliation which are not closed and attractive. This applies in particular to leaves entering $A_t$ through the boundary.  

\subsubsection{Elimination of singularities}

The following lemma is a translation of Lemma 2.15 of \cite{gi-bif}. It is a version of the elimination lemma from \cite{el20} which also controls the characteristic foliation not only on one surface $\Sigma$ but also on nearby surfaces.  

\begin{lem} \mlabel{l:elim}
Let $\xi$ be a contact structure on $\Sigma\times[-1,1]$ such that $\Sigma_0(\xi)$ has two singular points $e_0,h_0$ which have the same sign (we will assume it is negative) and are connected by a leaf of $\Sigma_0(\xi)$. Moreover, for $-1<t<1$, there are singularities $e_t,h_t$ of $\Sigma_t(\xi)$ such that $h_t$ is connected to $e_t$ by a continuous family of leaves $C_t$ of $\Sigma_t(\xi)$.  We fix $0<\delta<1$ and neighbourhood $U$ of $\cup_{|t|\le\delta}\overline{C}_t$, so that $e_t,h_t$ are the only singular points of $\Sigma_t(\xi)$ inside $\Sigma_t\cap U$.

\begin{itemize}
\item[(i)] There is an isotopy of $\Sigma\times[-1,1]$ with support inside $U$ such that the characteristic foliation of the isotoped contact structure has no singular points in $\Sigma_t\cap U$ for $|t|\le\delta$.
\item[(ii)] Given a family of points $a'_t\in \Sigma_t\setminus \overline{C}_t$  connected to $e_t$ by a leaf of $(\Sigma_t\cap U)(\xi)$ and another family $b_t'$ of points on the unstable leaf of $b_t$ which is opposite to $C_t$ the isotopy in (i) can be chosen such that for the isotoped contact structure the leaf of the characteristic foliation on $\Sigma_t$ through $b_t'$ passes arbitrarily close to $a'_t$.
\item[(iii)] Assume that for a given compact set $\Lambda\subset(-1,1)$ the restriction of the characteristic foliation on $\Sigma_t$  to  \begin{align*}
G_t := & \textrm{ union of all segments of leaves of }\Sigma_t(\xi) \\ &  \textrm{ with both endpoints in } U 
\end{align*}
can be defined by a $1$-form with nowhere vanishing exterior derivative. Then the isotopy from (i), (ii) can be chosen such that if  $\Sigma_t(\xi)$ is convex for $t\in\Lambda$ and the isotoped copy of $\Sigma_t$ is also convex.
\end{itemize} 
\end{lem}
Pairs of singular points of $\Sigma(\xi)$ like $e_0,h_0$ in the above lemma will be referred to as a {\em canceling pair}. 
 
We will use a partial converse of this result which allows us to create canceling pairs of singular points. On a single surface it is possible to introduce a canceling pair of singularities without any restriction. But by \lemref{l:switch} we cannot arbitrarily prescribe the limit set of the (un-)stable leaf of the hyperbolic singularity which does not connect the elliptic singularity of the pair.   

\subsubsection{Closed leaves in movies} \mlabel{s:closed leaves}

In this section we discuss a result from E.~Giroux's paper \cite{gi-bif} about closed leaves of characteristic foliations. 

First, let $\gamma\subset \Sigma_t$ be a non-degenerate closed leaf of the characteristic foliation. Let $V\subset\Sigma_t$ be a tubular neighbourhood of $\gamma$ so that $\partial V$ is transverse to $\Sigma_t(\xi)$ and $\gamma$ is the unique closed leaf of $\Sigma_t(\xi)$ in $V$. Since $\gamma$ is non-degenerate, the characteristic foliation is pointing out of $V$ or into $V$ along both components of $\partial V$. Because the characteristic foliation depends smoothly on $t$ and there is a closed leaf in $\Sigma_{t'}$ contained in $V$ now viewed as a subset of $\Sigma_{t'}$ when $t'$ is sufficiently close to $t$. Since $\gamma$ is non-degenerate the neighbouring closed leaves are uniquely determined by the surface containing them and the union of these closed leaves is a smooth submanifold transverse to $\Sigma_t$. 

Now let $\gamma\subset\Sigma$ be a degenerate closed leaf and $\varphi$ the germ of the holonomy along $\gamma$ with respect to a fixed segment $\sigma$ through $\gamma$ transverse to the characteristic foliation. We assume that the degeneracy of $\gamma$ is finite, i.e. $\varphi^{(k)}(0)\neq 0$ for some $k\in\{1,2,\ldots\}$ and $\varphi^{(j)}(0)=0$ for $j=0,\ldots,k-1$. This is a $C^\infty$-generic property. 

We first discuss the case when $k$ is even. Depending on the sign of $\varphi^{(k)}(0)$ there are two possibilities.
\begin{defn}
We say that $\gamma$ is {\em positive} respectively {\em negative} if the holonomy of $\gamma$ is repulsive respectively attractive on the side of $\gamma$ given by the coorientation of $\xi$ while the behavior on the other side is opposite.  
\end{defn}
The following lemma can be found in \cite{gi-bif} combining Lemma 2.12 and the remark following it. 
\begin{lem} \mlabel{l:gi-birth-death}
Let $\gamma$ be a positive degenerate closed orbit of $\Sigma_\tau(\xi)$. Then for $t>\tau$ and $t$ close to $\tau$ there is a pair of non-degenerate closed leaves of $\Sigma_t(\xi)$ close to $\gamma$. One of these orbits is repulsive while the other one attractive. For $t<\tau$ there is no closed leaf of $\Sigma_t$ near $\gamma$. 

For negative degenerate orbits, the situation is opposite. 
\end{lem}
Thus positive respectively negative degenerate closed leaves indicate the birth respectively death of a pair of closed leaves of the characteristic foliation. 

The case when $k$ is odd is simpler. Following the proof of Lemma 2.12 in \cite{gi-bif} we obtain:

\begin{lem}\mlabel{l:closed-leaves-surface}
If $\gamma\subset\Sigma_\tau$ has odd degeneracy, then the characteristic foliation on a surface $\Sigma_t$ sufficiently close to $\Sigma_\tau$ has a single closed leaf $\gamma_t$ near $\gamma$ and $\gamma_t$ is attractive or repulsive if and only if the same is true for $\gamma$ and $\gamma_t$ is non-degenerate for $t\neq\tau$. The union of the closed leaves $\gamma_t$ is a smooth embedded surface in $M$.  
\end{lem}

So, to summarize this section, if $\gamma$ is a closed leaf $\Sigma_t(\xi)$, then the union of nearby closed leaves on nearby surfaces $\Sigma$ is a smooth submanifold.  

\subsubsection{Retrograde saddle-saddle connections} \mlabel{s:retrograde}

Let $\xi$ be a contact structure on $\Sigma\times[-1,1]$. 
\begin{defn} \mlabel{d:retrograde}
A stable leaf $\eta$ of a positive hyperbolic singularity of the characteristic foliation on $\Sigma_0$ which coincides with the unstable leaf of a negative hyperbolic singularity  is a {\em retrograde saddle-saddle connection}. 
\end{defn}
The fact that $\xi$ is a positive contact structure has consequences for the characteristic foliation on $\Sigma_t$ when $\Sigma_0(\xi)$ contains a retrograde saddle-saddle connection.  In the following lemma the words {\em over} and {\em under} refer to the coorientation of the leaves of the characteristic foliation.

\begin{lem}[Giroux \cite{gi-bif}] \mlabel{l:switch}
A retrograde saddle-saddle connection $\eta$  on $\Sigma_0$ implies that for $t<0$ respectively $t>0$ sufficiently close to $0$, the stable leaf $\eta$ of the positive hyperbolic singularity  passes under respectively over the unstable leaf of the negative hyperbolic singularity. 
\end{lem}

A retrograde saddle-saddle connection is depicted in Figure~\ref{b:switch}. If the surface $\Sigma_t$ is convex for $t\neq 0$ then a bypass attachment along the thickened arc in the leftmost figure has the same effect on the dividing set (represented by dashed curves) as the retrograde saddle-saddle connection.  

\begin{figure}[htb]
\begin{center}
\includegraphics[scale=0.7]{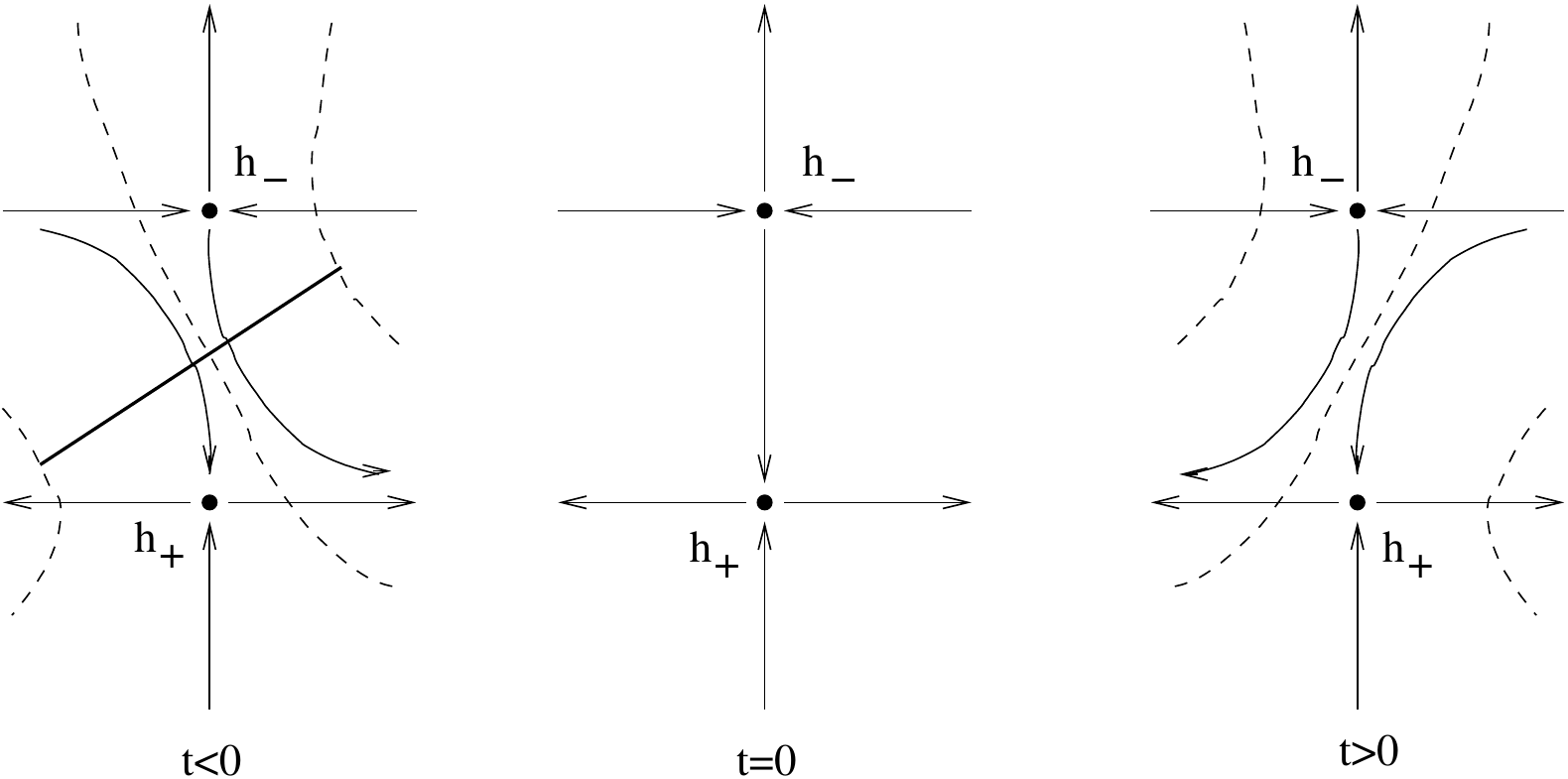}
\end{center}
\caption{Retrograde saddle-saddle connection}\label{b:switch}
\end{figure}

Later we will want to reduce the number of retrograde saddle-saddle connections. In some situations, the classification of tight contact structures on the ball can be used for this. 

\begin{lem} \mlabel{l:elim retrograde}
Let $\xi$ be a tight contact structure on $\Sigma\times[-1,1]$ so that the characteristic foliation on $\Sigma_0$ has a single  retrograde saddle-saddle-connection $\eta$ between a positive hyperbolic singularity $h_+$ and a negative hyperbolic singularity $h_-$.

Assume that both unstable leaves of $h_+$ connect to the same negative elliptic singularity $e_-$  and that $\Sigma_t$ is convex for all $t\neq 0$.  Then there is a contact structure $\xi'$ isotopic to $\xi$ with the following properties. 
\begin{itemize}
\item $\Sigma_t(\xi')$ is convex for all $t\in[-1,1]$. In particular, there are no retrograde saddle-saddle connections.
\item $\xi'$  coincides with $\xi$ outside of a tubular neighbourhood of the union of the unstable leaves of $\Sigma_0$.
\end{itemize}
\end{lem}

This lemma is a consequence of the discussion of trivial bypasses in \cite{honda-glue}. Let us describe how the movie of the contact structure with a retrograde saddle-saddle-connection as in \lemref{l:elim retrograde} can be replaced by a movie which is convex at all levels.

Let $V\subset \Sigma_0$ be a tubular neighbourhood of the two unstable leaves of $h_+$ whose boundary is transverse to $\Sigma_0(\xi)$. After eventually adding pairs of canceling positive singularities and a small perturbation on the complement of a neighbourhood of the retrograde saddle-saddle connection we may assume that the basin of $V$ is compact. Then $\Sigma_t(\xi)$ and  $V_t$ have the same properties as $V$ for $t\in[-\delta,\delta]$ close to $0$.  Now replace the movie $\Sigma_t(\xi)$ by a family of singular foliations obtained by rotating the characteristic foliation on $\Sigma_t$ starting with $\Sigma_{-\delta}(\xi)$ so that as $t$ increases no stable leaf of  $h_+$ meets the point where the unstable leaf of $h_-$ enters $V_t$ and $V_\delta\subset\Sigma_{\delta}(\xi)$ makes a full twist. This is possible because there is exactly one unstable leaf of a negative singularity  entering $V$.  

The singular foliations obtained in this way all admit dividing sets. By \lemref{l:convex ext} the movie of singular foliations is the movie of a contact structure $\xi'$ on $\Sigma\times[-\delta,\delta]$ which coincides with $\xi$ on the boundary. 

The assumption that $e_-$ is a negative elliptic singularity which is not connected to any negative hyperbolic singularity can be replaced by the following assumption: Both unstable leaves of $h_+$ end on the same connected component $e_-$ of the graph formed by negative elliptic singularities, attractive closed leaves and negative hyperbolic singularities together with their unstable leaves and $e_-$ is a closed tree. Then the annulus $V$ above  has to contain the entire tree. 

\subsubsection{Sheets of movies} \mlabel{s:sheets}

Consider a contact structure on $\Sigma\times [-1,1]$  and the movie of characteristic foliations $\Sigma_t(\xi)$ with $t$ varying in $[-1,1]$.  
\begin{defn} \mlabel{d:sheet}
A {\em sheet} of the movie is a smooth embedded surface $A\subset\Sigma\times[-1,1]$ so that  every connected component of $A\cap\Sigma_t$ is a smooth Legendrian curve and all singularities of $\Sigma_t(\xi)$ on the curve have the same sign. 
\end{defn}
These surfaces play a very important role in \cite{gi-bif}. In that paper, a sheet is referred to as {\em feuille} while the collection of all sheets is called {\em feuillage}. By definition, $A$ is foliated by circles. Therefore $A$ is either a torus, an annulus, a Klein bottle, or a M{\"o}bius band. In this paper $A$ will always by orientable, so $A$ will be either a torus or an annulus. The following lemma is part of Lemma 3.17 in \cite{gi-bif}. It implies that a sheet $A$ is really foliated by closed Legendrian curves:

\begin{lem}  \mlabel{l:no sing on sheet}
Every sheet $A$  is a pre-Lagrangian surface, i.e. $A(\xi)$ is a non-singular foliation by closed Legendrian curves.
\end{lem}    

Averaging the contact form using a flow which is tangent to $A(\xi)$ and periodic we obtain a contact form $\alpha$ whose restriction to $A$ is closed. Then $A\subset M$  is a Lagrangian submanifold of the symplectisation $(M\times\R,d(e^t\alpha))$ of $(M,\alpha)$
Moreover, $A$ is either tangent to $\Sigma_t$ along a given closed Legendrian curve $\beta\subset \Sigma_t\cap A$ or $A$ is everywhere transverse to $\Sigma_t$ along $\beta$. More precisely, $A$ is tangent to $\Sigma_t$ along a Legendrian curve $\beta\subset A$ if and only if $\beta$ is a degenerate closed leaf of $\Sigma_t(\xi)$. 

We consider the situation when $\xi$ is transverse to $\II$, the foliation given by the second factor of $\Sigma\times[-1,1]$. In this case the behavior of a sheet relative to the product decomposition $\Sigma\times[-1,1]$ is subject to restrictions which we now describe.

Let $\alpha$ be a contact form which is closed on $A$ and $\beta\subset\Sigma_t\cap A$ a non-degenerate attractive closed leaf. Because $\alpha$ is a contact form the $1$-form $(d\alpha)(\dot{\beta},\cdot)$ is non-vanishing, only its restriction to $A$ vanishes.  Since $\beta$ is attractive the $2$-form $d\alpha$ is a negative area form on $T\Sigma\eing{\beta}$. Using \lemref{l:no sing on sheet} it follows that the lines
\begin{equation} \label{e:order}
\II_p,\xi_p/\dot{\beta},T_pA/\dot{\beta},T_p\Sigma_t/\dot{\beta}
\end{equation}
appear in this order in the projective line $\mathbb{P}(T_pM/\dot{\beta})$.  Hence the slope of $\xi$ is steeper at $p\in\beta$ than the slope of $A$ in that point (we interpret the second factor in $\Sigma\times[-1,1]$ as height). This is shown in \figref{b:un-stable} where the parts of sheets consisting of attractive closed leaves of $\Sigma_t(\xi)$ are thickened. If $\beta$ is repulsive, then $\xi(p)$ is closer to $T_p\Sigma_t$ than $T_pA$. Moreover, if $A$ is tangent to $\II$ at a point $p\in\Sigma_t$, then the closed leaf $\gamma_p$ of $T_pA$ is a Legendrian curve on $\Sigma_t$ which is either closed and repelling or which contains some positive singularities and after elimination of these singularities we again obtain a closed repelling leaf. 

\begin{prop} \mlabel{p:sheets}
Let $\xi$ be a contact structure on $\Sigma\times[-1,1]$ which is transverse to and cooriented by the second factor and $\beta$ a closed attractive leaf of $\Sigma_t(\xi)$. 

Then $\beta$ is contained in a sheet $A(\beta)$ of the movie $\Sigma_t(\xi)$ consisting of non-degenerate attractive closed leaves of $\Sigma_t(\xi)$. The restriction of the projection $\mathrm{pr}: \Sigma\times[-1,1]\lra\Sigma$ to this sheet is a submersion. As $t$ increases the image of $\Sigma_t\cap A$ moves in the direction opposite to the direction determined by the coorientation of $\xi$. 
\end{prop}

\begin{proof}
Let $\beta$ be an attractive closed leaf of $\Sigma_t(\xi)$. Then this curve is part of a sheet $A(\beta)$ by the implicit function theorem and this also implies that $A(\beta)$ is transverse to the surfaces $\Sigma_t$. By \eqref{e:order} the tangent space of this submanifold lies  between the tangent space of $\Sigma_t$ and $\xi\eing{\beta}$ which does not contain $T\II\oplus T\beta$. Since $\mathrm{pr}_*$ is an isomorphism on $T\Sigma_t$ and $\xi\eing{\beta}$  it is an isomorphism on $TA$, too.  
\end{proof}

Hence sheets are transverse to $\Sigma_t$ away from degenerate closed leaves and they are transverse to $\II$ along attractive pieces. According to \lemref{l:gi-birth-death} the locus where sheets are not transverse to $\Sigma_t$ corresponds to degenerate closed leaves of $\Sigma_t(\xi)$. 

\begin{rem} \mlabel{r:sheet prop}
\propref{p:sheets} has important consequences for how sheets are embedded into $\Sigma\times[-1,1]$ when the contact structure is transverse to the foliation $\II$ given by the second factor. 

Let $\beta$ be a non-degenerate attractive closed curve in $\Sigma_{-1}$ (the following discussion for $\beta\subset\Sigma_{+1}$ is completely analogous). If one moves on the sheet $A(\beta)$ containing $\beta$, then the $t$-coordinate increases until either a degenerate closed orbit or $\Sigma_{+1}$  is reached. We will consider only the first case. Moreover, we assume that this degenerate leaf is of birth-death type since otherwise the sheet simply continues.  

By \lemref{l:gi-birth-death} the degenerate closed orbit is negative and after we cross the degenerate closed leaf the sheet consists of repulsive closed orbits of $\Sigma_t$ and the $t$-coordinate decreases as we move on $A$ away from $\beta$. 

Along the part of $A(\beta)$ which consists of either closed repulsive leaves of $\Sigma_t(\xi)$ or of a graph consisting of unstable leaves of positive hyperbolic singularities connected to positive elliptic singularities such that the graph is diffeomorphic to a circle the $t$-coordinate decreases until a degenerate closed orbit of birth-death type or $\Sigma_{-1}$ is reached unless  the sheet simply ends in a level surface $\Sigma_t$ (this happens for example if an elliptic singularity of the movie lying on $A$ forms a canceling pair with a hyperbolic singularity which does not lie on $A$ such that these two singularities merge on $\Sigma_t$). Now assume that a degenerate orbit in $\Sigma_{t'}$ is reached on $A(\beta)$ and the orientation of this orbit is opposite to the orientation of $\beta$. Then the $t$-coordinate increases as we move on the sheet $A(\beta)$ away from $\beta$, but the part of the sheet consisting of attractive closed leaves of $\Sigma_t(\xi)$ is now trapped inside a solid torus bounded by the sheet and an annulus in $\Sigma_{t'}$. Therefore the sheet $A(\beta)$ reaches the highest level  
$$
t(A(\beta))=\sup\{t\in[-1,1]\,|\,A(\beta)\cap \Sigma_t\neq\emptyset\}
$$
along an attractive closed leaf of $\Sigma_{t(A(\beta))}(\xi)$ which is attractive or degenerate and the orientation of this leaf coincides with the orientation of $\beta$ provided that the supremum above is actually attained. The next thing we show is that this is always the case.

The situation described here is depicted in \figref{b:sheet} (which contains some notations and a dashed line that will be explained later). The parts of $A(\beta)$ which consist of attractive leaves of $\Sigma_t(\xi)$ are thickened.  
\end{rem}

\begin{figure}[htb]
\begin{center}
\includegraphics[scale=0.8]{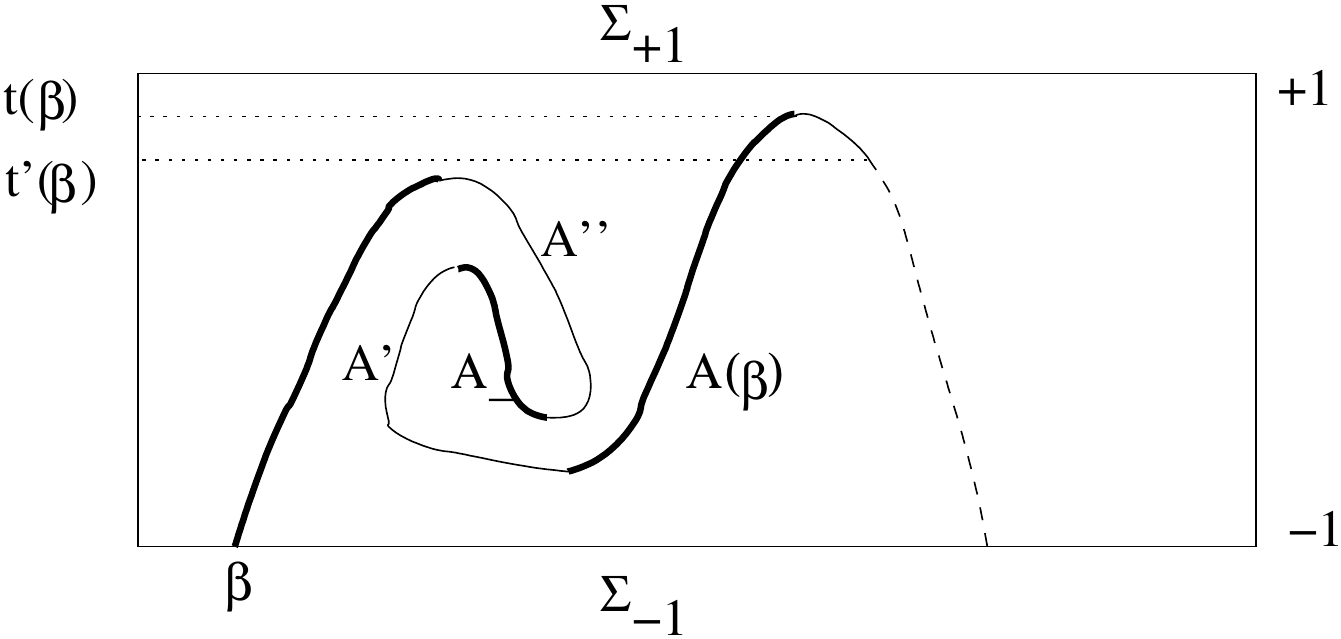}
\end{center}
\caption{Sheet in $\Sigma\times[-1,1]$ containing $\beta\subset\Sigma_{-1}$}\label{b:sheet}
\end{figure}

The following lemma shows that degenerate closed orbits are the {\em only} way in which attractive closed orbits of a movie in  a contact manifold appear or disappear (the lemma is wrong when $\xi$ is just a plane field) when $\xi$ is transverse to the second factor of $\Sigma\times[-1,1]$. This also provides a natural way to compactify sheets consisting of attractive closed leaves of $\Sigma_s(\xi)$. 

\begin{lem} \mlabel{l:compactify by degen}
Let $A$ be a sheet consisting of closed attractive leaves of $\Sigma_t(\xi)$ such that $\Sigma_t\cap A$ is not empty for $t\in[-1,b)$.

The annulus $A\subset\Sigma\times[-1,b)$ can be compactified by adding a closed leaf $\gamma_b$ of the characteristic foliation of $\Sigma_b$. The holonomy of $\gamma_b$ is attractive on the side determined by the coorientation of $\xi$. If $A\subset\Sigma\times(b,1]$, then the holonomy of $\gamma_b$ is attractive on the side opposite to the coorientation of $\xi$. 
\end{lem}

\begin{proof}
We consider the case $A\subset\Sigma\times[-1,b)$. The set $\LL=\overline{A}\cap \Sigma_b$ is non-empty, closed and saturated (i.e. a union of leaves of the characteristic foliation).

If $\LL$ contains a non-trivial recurrent leaf $\rho$, then we can find a closed curve $\tau$ on $\Sigma_b$ transverse to $\Sigma_b(\xi)$ through a given point of $\rho$. Because $\rho$ is recurrent this leaf intersects $\tau$ infinitely many times and all intersection points are transverse and have the same sign when $\tau$ is oriented. Then the intersection number of $\beta_t=A\cap \Sigma_t$ with $\tau$ is unbounded as $t$ approaches $b$. But this is absurd since the homology class of $\beta_t\subset \Sigma_t\simeq \Sigma$ is constant. Therefore $\LL$ does not contain a non-trivial recurrent leaf.

Now assume that $\LL$ contains a degenerate closed leaf as proper subset (there could be a chain of degenerate closed leaves  connected by leaves of the characteristic foliation). Then every degenerate closed leaf is positive because otherwise two closed leaves of $\Sigma_t(\xi)$ would intersect $A$ for $t<b$ by \lemref{l:gi-birth-death}.  \figref{b:posneg} depicts a configuration with one positive and one negative degenerate orbit. 

\begin{figure}[htb]
\begin{center}
\includegraphics[scale=0.8]{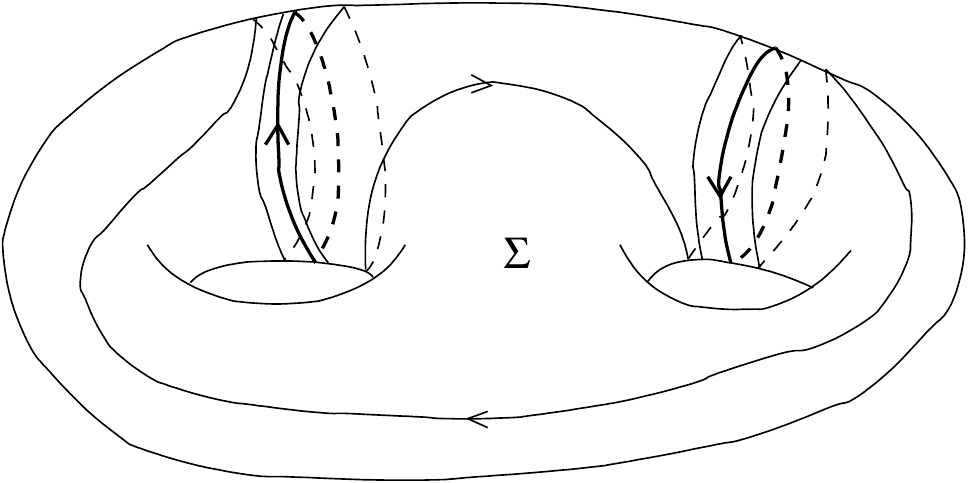}
\end{center}
\caption{Impossible limit configuration}\label{b:posneg}
\end{figure}

But when all degenerate closed leaves are positive then using the leaves of $\Sigma_b(\xi)$ which connect the degenerate closed orbits in $\LL$ one can construct a closed curve $\tau$  transverse to $\Sigma_b(\xi)$ which intersects the degenerate closed leaves. This leads to the same contradiction as above. Thus either there $\LL$ is closed attractive leaf, or it is degenerate, or $\LL$ contains a cycle. If $\LL$ is an attractive leaf  we compactify $A$ by adding it. Also, if $\LL$ is degenerate, then it has to be negative by \lemref{l:gi-birth-death} and serves as a natural compactification of $A$. 

In order to finish the proof we have to exclude the possibility that $\LL$ contains a cycle consisting of stable/unstable leaves of singularities of index $0$ or $-1$. For this recall that all singularities of $\Sigma_t(\xi)$ are positive. Therefore the holonomy of the characteristic foliation is strongly repelling (this is the property described in \lemref{l:very rep}) when one passes from a stable leaf to an unstable leaf of the characteristic foliation. We choose a finite collection short transversals $\tau_i,i\in\Z_m,$ of $\Sigma_b(\xi)$, one for each stable leaf in the cycle in cyclic order and intersecting $\LL$ exactly once. Let $x_i=\tau_i\cap\rho$. The holonomy of $\Sigma_b(\xi)$ determines diffeomorphisms of open sets in $\tau_i$ to open sets in $\tau_{i+1}$. These diffeomorphisms extend to homeomorphisms $\varphi_i$  when we add $x_i$ to the domain. By \lemref{l:very rep} $\varphi_i'(x_i)=\infty$ for all $i\in\Z_m$. 

Consider the attractive leaf $\beta_t=A\cap\Sigma_t$ for $t<b$ close to $b$ and the holonomy diffeomorphisms $\psi_t$ it induced on open sets if $\tau_i$ to open sets of $\tau_{i+1}$ and let $y_{t,i}:=\tau_i\cap\beta_t$. Since $\beta_t$ is attractive
$$
\psi'_{t,1}(y_1)\cdot\ldots\cdot\psi'_{t,m}(y_m)<1.
$$   
Therefore we may (after choosing sequences and subsequences) assume without loss of generality that $\psi_{t,1}'(y_1)<1$. But on the side of $\rho$ where the holonomy of $\Sigma_b(\xi)$ is defined $\psi_{t,1}$ converges uniformly to $\varphi_{1}$. This contradicts the fact that $\varphi_{1}$ is very repelling.   

So $\LL$ is either an attractive closed leaf or a negative degenerate orbit of $\Sigma_b(\xi)$. 
\end{proof}

According to \lemref{l:gi-birth-death} one can extend the sheet beyond $\overline{A}\cap\Sigma_b$. Whenever there are conditions which ensure that there is a closed repulsive leaf or union of stable leaves of singular points of the characteristic foliation such that this union is the boundary of the basin of  $A\cap\Sigma_t$, then we will assume that this circle is smooth and we extend the sheet we are considering as far as possible. A condition which often ensures that sheets can be extended easily is
\begin{itemize}
\item the contact structure is tight and 
\item the basin of $A\cap\Sigma_t$ is contained in an annulus bounded by attractive closed leaves of $\Sigma_t(\xi)$.
\end{itemize}

Finally we fix some terminology: We could say that a connected sheet is maximal, if it is not a proper subset of a connected sheet. The problem with this definition is that  leaves of characteristic foliations in a smooth sheet $A$ can contain singularities (all of the same sign). Therefore, a smooth Legendrian curve  in $A\cap\Sigma_t$ can be the limit of a family of non-smooth Legendrian curves (the non-smooth points are elliptic singularities of the characteristic foliation) which would naturally extend the sheet if they were smooth. However, the non-smoothness of the curves can be easily corrected using for example \lemref{l:rel convex}.
\begin{defn} \mlabel{d:maximal sheet}
A connected sheet $A$ is {\em maximal} if it is not a proper subset of a smooth connected sheet and no component of $\partial A$ is the limit of non-smooth Legendrian curves in $\Sigma_t$ such that all singularities have the same sign.   
\end{defn}

\subsubsection{Simplifying the dynamics of characteristic foliations in movies} \mlabel{ss:pb}

Let $\Sigma$ be a closed surface with positive genus $g\ge 1$. The purpose of this section is to describe how contact structures  on $\Sigma\times[-1,1]$ can be isotoped so that the characteristic foliations on $\Sigma_t$ have relatively simple dynamical properties when $\Sigma_t$ is not convex with respect to the isotoped contact structure.
\begin{defn} \mlabel{d:pb}
A surface $\Sigma$ in a contact manifold has the {\em Poincar{\'e}-Bendix\-on property} if $\Sigma(\xi)$ has no non-trivial recurrent orbits.
\end{defn}
If $\Sigma$ has the Poincar{\'e}-Bendixon property and $\Sigma(\xi)$ has only finitely many singularities, then according to \cite{nik} all limit sets of leaves of $\Sigma(\xi)$ are either
\begin{itemize}
\item closed leaves,
\item singular points, or
\item cycles formed by singularities and leaves connecting them. 
\end{itemize}
An embedded closed surface in a contact manifold has this property after a $C^\infty$-generic perturbation. The point of Lemma 2.10 of \cite{gi-bif} is to ensure this property for all those surfaces $\Sigma_t\subset\Sigma\times[-1,1]$ which are not convex. We are going to use the following simple refinement of that lemma.

\begin{lem} \mlabel{l:pb}
Let $\xi$ be a contact structure on $N=\Sigma\times[-1,1]$ such that the boundary surfaces are convex. Then there is an isotopy of $\xi$ relative to the boundary such that after the isotopy $\Sigma_t$ has the Poincar{\'e}-Bendixon property for all $t\in[-1,1]$ for which $\Sigma_t$ is not convex. 

If there is a sheet $A(\beta)$ such that one boundary component $\beta_+$ of $A(\beta)$ is contained in $\Sigma_1$ while the other boundary component $\beta_-$ is contained in $\Sigma_{-1}$ and $\beta_+,\beta_-$ are non-degenerate and both attractive or both repelling, then the isotopy can be chosen to preserve the sheet $A(\beta)$. 
\end{lem}

\begin{proof}
The proof follows Giroux's proof of Lemma 2.10 in \cite{gi-bif} closely. We summarize the required changes and the main idea. Let us first recall that the Poincar{\'e}-Bendixon theorem (cf. p. 154 of \cite{hartman}) states that a singular foliation on the plane or the sphere has no non-trivial recurrent orbits.

For concreteness we assume that $\beta_+$ and $\beta_-$ are both attractive. Then there is a family of annuli $P_t\subset \Sigma_t$ containing $\Sigma_t\cap A(\beta)$ such that  $\partial P_t$ is transverse to $\Sigma_t(\xi)$. We chose the identification $N\simeq \Sigma\times [-1,1]$ such that $P_t=P\subset\Sigma$ is constant. Now fix a graph $F$ so that 
\begin{itemize}
\item[(i)] $F\cup P$ is planar,
\item[(ii)] the complement $\Sigma\setminus F\cup P$ is also planar, and
\item[(iii)] $F$ is non-isolating in $\Sigma_{+1}$ and in $\Sigma_{-1}$. \end{itemize}

Then $F$ can be realized as Legendrian graph consisting Legendrian curves, negative elliptic and positive hyperbolic singularities. There is a positive number $\delta$ such that all surfaces $\Sigma_t$ with $t\in[-1,-1+3\delta]\cup[1-3\delta,1]$ are convex. Using \lemref{l:convex families} and the usual proof of the Legendrian realization principle (\lemref{l:LeRP}) we can now isotope $\xi$ near the boundary of $N$ so that 
\begin{itemize}
\item the isotopy is supported in $|t-1|\le 3\delta$, 
\item the characteristic foliation on $\Sigma_t$ is constant for $t\in[-1+\delta,-1+2\delta]$ and for $t\in[1-2\delta,1-\delta]$, and
\item $F$ is a Legendrian curve of the characteristic foliation on $\Sigma_t$ for $t\in[-1+\delta,-1+2\delta]$ and for $t\in[1-2\delta,1-\delta]$.\end{itemize}
where all surfaces $\Sigma_t$ are convex while keeping $\xi$ constant on $\partial N$ such that for suitable small real numbers $\delta_{\pm}>0$, the graph $F$ is realized as a Legendrian graph in $\Sigma_{1-\delta_+}$ and $\Sigma_{-1+\delta_-}$. Clearly, this can be done without changing anything near $A(\beta)$. 

A thickening of $F$ combined with $P$ is a planar subsurface $F_{in}$ of $\Sigma$ whose complement is also planar. In addition, choosing the thickening appropriately, we may assume that the characteristic foliation on $\Sigma_{1-\delta_+}$ and $\Sigma_{-1+\delta_-}$ is transverse to $\partial F^{in}$. Let $F^{out}$ be the complement of $F^{in}$ with a collar of the boundary removed. The collar is chosen such that following leaves of the characteristic foliation on the collar one gets a retraction of the collar onto $\partial F^{out}$. The characteristic foliations point out of $F^{out}$ and into $F^{in}$ for $t\in[-1+\delta,-1+2\delta]\cup[1-2\delta,1-\delta]$.

Now choose a strictly monotone function $g: [0,1] \lra [1-2\delta,1]$ so that $g=\id$ on $[1-\delta,1]$. Pick an isotopy $\phi_\tau$ of $N$ which translates along leaves of $\II$ such that
\begin{align*}
\phi_1(F^{in}_t) & = F^{in}_{g(t)} & \textrm{for } & t\ge 0 \\
\phi_1(F^{out}_t) & = F^{out}_{-g(-t)} & \textrm{for } & t\le 0.
\end{align*} 
The contact structure $\widehat{\xi}=\phi^{-1}_{1*}(\xi)$ has the desired properties: For $t\in[-1+\delta,1-\delta]$ there are no non-trivial recurrent orbits in $\Sigma_t(\widehat{\xi})$ by the Poincar{\'e}-Bendixon theorem and $\Sigma_t$ is convex with respect to $\widehat{\xi}=\xi$ when $t\in[-1,-1+\delta]$ or $t\in[1-\delta,1]$.
\end{proof}
Of course, \lemref{l:pb} also holds in the presence of several sheets with the same properties as $A(\beta)$. The proof above implies that the resulting contact structure can be assumed to be $C^1$-generic with respect to the surfaces $\Sigma_t$, i.e. we can make genericity assumptions concerning for example the nature of connections between hyperbolic singularities.

\subsection{Manipulations and properties of sheets} \mlabel{s:prop sheets}

In this section we explain how to manipulate sheets and circumstances under which it is possible to find overtwisted discs from certain configurations of sheets. 

\subsubsection{Simplifying sheets}

The next lemma  is part of the proof of Proposition 3.22 of \cite{gi-bif}. It allows us to isotope $\xi$ so that the sheet contains fewer degenerate closed curves after the isotopy. For example, the part $A'\cup A_-\cup A'' $ of $A(\beta)$ in \figref{b:sheet} can  be replaced by collection of attractive closed leaves of $\Sigma_t(\xi)$.

\begin{lem} \mlabel{l:isotope sheets}
Let $A\subset(\Sigma\times[-1,1],\xi)$ be a sheet such that $A$ is the union of three sheets $A',A_-,A''$ with the following properties:
\begin{itemize}
\item[(i)] $\gamma'=A'\cap A_-$ and $\gamma''=A''\cap A_-$  are degenerate closed orbits with parallel orientations. 
\item[(ii)] $A_-\cap\Sigma_t$ is a smooth attractive Legendrian curve unless $t=t_{min}$ or $t=t_{max}$ with 
\begin{align*}
t_{min} & = \min\{t\in[-1,1]\,|\,A_-\cap\Sigma_t\neq\emptyset\} \\
t_{max} & = \max\{t\in[-1,1]\,|\,A_-\cap\Sigma_t\neq\emptyset\}. 
\end{align*}
\item[(iii)] For all $t\in (t_{min},t_{max})$ there is a compact annulus $S_t\subset\Sigma_t$ whose unoriented boundary consists of $A_-\cap\Sigma_t$ and $A'\cap\Sigma_t$ such that $S_t(\xi)$ intersects no other sheets of $\xi$. 
\end{itemize} 
Then there is a family of contact structures $\xi_s,s\in[0,1],$ with $\xi_0=\xi$ which is constant near $\partial A$ such that after the deformation, there is a new sheet $A_1$  which coincides with $A_0$ near the boundary such that $A_1\cap\Sigma_t$ is either empty or an attractive closed leaf. 
\end{lem}
The next lemma shows that a given degenerate closed leaf of birth-death type can be replaced by a retrograde saddle-saddle connection. However, without additional assumptions it is not possible to exclude the formation of degenerate closed leaves passing through a given neighbourhood of the original degenerate leaf.

\begin{lem} \mlabel{l:replace degen}
Let $\xi$ be a contact structure on $\Sigma\times[-1,1]$ and $\gamma\subset \Sigma_t$ with $t\in (-1,1)$ a degenerate closed orbit which is attractive on one side while it is repelling on the other side. Then there is a contact structure $\xi'$ which is isotopic to $\xi$, coincides with $\xi$ outside of an arbitrarily small neighbourhood of $\Sigma_t$ and $\gamma$ is replaced by a retrograde saddle-saddle connection such that there is no degenerate closed leaf of $\Sigma_t(\xi')$ in a neighbourhood of $\gamma$. 
\end{lem}

\begin{proof}
We consider the case when the degenerate closed orbit is negative. We use the following model contact structure on $A\times[-\delta,\delta], \delta>0$, where $A=S^1\times[-1,1]$ is an annulus.  On $A_0$ we fix the following singular foliation.
\begin{itemize}
\item[(i)] Both boundary components are parallel non-singular Legendrian curves, one of them is repelling and the other one is attractive. They are both non-degenerate.
\item[(ii)] There are four non-degenerate singular points $e_\pm,h_\pm$.  Here $e_+$ is a positive elliptic point (i.e. with positive divergence) etc.
\item[(iii)] There is a retrograde saddle-saddle connection starting at $h_-$ and ending at $h_+$ such that for $t>0$ the stable leaf which participates in the retrograde saddle-saddle connection comes from the boundary of the annulus. 
\item[(iv)] The remaining stable leaf of $h_+$ comes from $e_+$, the remaining unstable leaf of $h_-$ ends at $e_-$.
\item[(v)] One stable leaf of $h_-$ comes from the repulsive boundary component, the other from $e_+$. The unstable leaves of $h_+$ connect $h_+$ to $e_-$ and to the attractive boundary component. 
\end{itemize}
By \lemref{l:char fol det} this singular foliation is the characteristic foliation on $A_0$ of a contact structure on $A\times[-\delta,\delta]$. By \lemref{l:switch} the stable leaf of $h_+$ in $A_t(\xi)$ which participates in the retrograde connection comes from $e_+$ when $t<0$ and from the repulsive boundary for $t\neq 0$. The only non-convex level is $A_0$. 

By \lemref{l:elim} we can eliminate $h_+,e_+$ and $h_-,e_-$ in $A_0$ and since there is a unique leaf connecting the hyperbolic singularity $h_\pm$ to $e_\pm$ there is a unique way to eliminate the singularities. Outside of a neighbourhood of $0$ we obtain either a pair of parallel closed leaves in the interior of the annulus or all leaves of the characteristic foliation (except the boundary leaves) start at one boundary component and go to the other.  (Part (iii) of that lemma can be used to arrange that away from a neighbourhood of $A_0$, there are exactly two or zero closed leaves in the interior of the annulus.)

After the elimination, the contact structure is transverse to a rank $1$-foliation transverse to $A_s,s\in(-1,1)$. According to \thmref{t:tight connection} the contact structure is tight and there are only $3$ sheets: Two at the boundary and one in the interior of the annuli. After a deformation of the interior sheet there is exactly one negative degenerate orbit at exactly one level.

This proves the claim in the model case. In order to deal with the general case note that the degenerate leaf is part of a sheet (as explained in \secref{s:closed leaves}). After folding the sheets outside of a small neighbourhood of the degenerate closed leaf we want to eliminate, we apply the construction of the model case. After we undo the folding using  \lemref{l:isotope sheets}  we get the desired result. 
\end{proof} 

The above lemmas allows us to arrange that the set of instances when $\Sigma_t$ is not convex is discrete and the only cause of  non-convexity is the presence of retrograde saddle-saddle connections. 

\begin{lem} \mlabel{l:finite}
Let $\xi$ be a contact structure on $\Sigma\times[-1,1]$ with convex boundary and $A_1,\ldots,A_n$ a collection of sheets consisting of attractive closed leaves of $\Sigma_t$  connecting the two boundary components of $\Sigma\times[-1,1]$. 

There is a contact structure $\xi_{PB}$ isotopic to $\xi$ relative to the boundary and $A_1\cup\ldots\cup A_n$ such that for all $t\in[-1,1]$ when $\Sigma_t(\xi_{PB})$ is not convex there is a single retrograde saddle-saddle connection. The number of $t$ with $\Sigma_t$ not convex with respect to $\xi_{PB}$ is finite.  Moreover, the upper basin of $A_i\cap\Sigma_t$ is compact for all $t$ where $\Sigma_t(\xi_{PB})$ is not convex. 
\end{lem}

\begin{proof}
We assume that $n=1$ and abbreviate $A_1=A$. 
First, we arrange that the characteristic foliation on $\Sigma_t$ has no non-trivial recurrent leaves at non-convex levels. 

This can be done using \lemref{l:pb} to $\Sigma\times[-1,1]$ with respect to the sheet $A$. Recall that in the proof of \lemref{l:pb} we arranged that for non-convex levels $\Sigma_t$ is decomposed into two planar regions such that $\Sigma_t(\xi_{PB}')$ is transverse to the boundary of the regions. Since the latter condition is open, we may impose that the movie $\Sigma_t(\xi_{PB}')$ is generic.

The planarity of the regions also implies that if $\eta$ is a degenerate closed orbit of $\Sigma_t(\xi_{PB}')$ for $t\in[-1,1]$ 
then there is no sequence of 
closed orbits $\eta_i$ of $\Sigma_{t_i}(\xi_{PB}')$ whose limit contains $\eta$: If such a sequence would exist, then there would be a closed orbit $\eta_i$ whose intersection number with $\eta$ is positive. But $\eta$ and $\eta'$ have to be contained in the same region from the proof of \lemref{l:pb} since they intersect and they are transverse to the boundary of the regions from the proof of \lemref{l:pb}. But two closed curves contained in a planar region have vanishing intersection number. Hence the set of levels $\Sigma_t$ with $t\in[-1,1]$  containing a degenerate closed orbit is discrete and hence finite. 

Using \lemref{l:replace degen} we isotope the contact structure $\xi_{PB}'$ on the union of neighbourhoods of degenerate closed orbits  such that after the isotopy we obtain a contact structure $\xi_{PB}$  there are no degenerate closed leaves in the movie $\Sigma_t(\xi_{PB})$ for $t\in[-1,1]$.  This amounts to introducing canceling pairs of singularities inside the regions from the proof of \lemref{l:pb}. Hence this operation does not affect the Poincar{\'e}-Bendixon property. 

So at all non-convex levels $\Sigma_t(\xi_{PB}), t\in [-1,1],$ has a retrograde saddle-saddle connection and by genericity we can arrange that each level contains at most one connection between saddle points.  Moreover, each retrograde saddle-saddle connection is isolated because of \lemref{l:switch} because the stable respectively unstable leaves of $h_-$ respectively $h_+$  of the singularities participating in such a connection are rigidly attached some stable limit set of the characteristic foliation for levels close to the level where the retrograde saddle-saddle connection occurs. Thus $\xi_{PB}$ has the desired properties except maybe the compactness of the basin. 

Let $t\in[-1,1]$ be such that $\Sigma_t$ is not convex with respect to $\xi_{PB}$. If the upper basin of $A\cap\Sigma_t$ is not compact, then we can introduce canceling pairs of singularities along all closed leaves and cycles of $\Sigma_{t}(\xi_{PB})$ (the signs of the singularities have to be chosen in such a way that we do not introduce retrograde saddle-saddle-connections when we place the singularities on cycles). Then the upper basin of $A\cap\Sigma_t$ is compact (also on nearby levels). 
\end{proof}

\subsubsection{Overtwisted discs from compressible sheets} 

By Giroux's criterion \lemref{l:Giroux crit} a convex oriented surface $\Sigma$ has a tight neighbourhood if and only if no component of the dividing set bounds a disc unless $\Sigma$ is a sphere. In this section we give a criterion for finding overtwisted discs from sheets with particular properties. The following lemma is essentially Lemma 3.34 from \cite{gi-bif}.

\begin{lem} \mlabel{l:overtwisted movie}
Let $\xi$ be a contact structure and $\Sigma\times[-1,1]$ such that there is a sheet $A$ with the following properties. 
\begin{itemize}
\item $A$ bounds a solid torus $S^1\times D^2$ in the interior of $\Sigma\times[-1,1]$. Let  
\begin{align*}
t_{min} & = \min\{t\in[-1,1]\,|\,A\cap\Sigma_t\neq\emptyset\} \\
t_{max} & = \max\{t\in[-1,1]\,|\,A\cap\Sigma_t\neq\emptyset\}. 
\end{align*}
\item For one level surface $\Sigma_t$ with $t_{min}<t<t_{max}$ there is a repulsive closed leaf $\alpha_t$ of $\Sigma_t(\xi)$ which is disjoint from the solid torus and isotopic to one of the curves $A\cap\Sigma_t$ so that the annulus bounded by $\alpha_t$ and the attractive leaf of $\Sigma_t(\xi)$ in $A\cap\Sigma_t$ contains no other closed leaf of $\Sigma_t(\xi)$. 
\end{itemize}
Then $\xi$ is overtwisted. 
\end{lem}

\begin{proof}
Without loss of generality we assume $\alpha_t$ is non-degenerate and $t=0$. Then $\alpha_0$ is part of a sheet $A(\alpha_0)$ which intersects nearby surfaces in repulsive curves close to $\alpha_0$. Using \lemref{l:elim} we can arrange that the annulus bounded by $\alpha_0$ and a connected component of $A(\alpha_0)\cap \Sigma_0$ does not contain a singular point of the characteristic foliation. Furthermore, we create a canceling pair of negative singularities $e,h$ on $\Sigma_t$ such all leaves which end at $e$ come from the repulsive closed curve $A\cap\Sigma_t$ except one unstable leaf of $h$. Just like the closed repulsive leaf on the other side of $A\cap\Sigma_t$ this configuration persists in nearby levels and after eventually using \lemref{l:isotope sheets} to $A$ we may assume that for all $t\in[t_{min},t_{max}]$ there is a canceling pair of negative singularities and a non-degenerate repulsive closed leaf parallel to $A$ and by \lemref{l:rel convex} we can replace these repulsive leaves by circles in $\Sigma_t$ consisting of positive singularities of the isotoped contact structure.  As in \cite{gi-bif} we achieve the following conditions. 
\begin{itemize}
\item $A\cap \Sigma_\tau$ is either empty, connected or has two connected components, and $A\cap\Sigma_\tau$ contains no singularities of the characteristic foliation except when $\tau=0$ and $A\cap\Sigma_0$ consists of two circles of singularities (one negative, the other positive). 
\item When $A\cap \Sigma_\tau$ consists of two connected components they have parallel orientations, so the two components do not bound a Reeb component.  
\end{itemize}

We will find on overtwisted disc in a surface consisting of one arc in $\Sigma_\tau$ with $\tau\in J$ where $J$ is a closed interval containing $[t_{min},t_{max}]$ in its interior. 
\begin{enumerate}
\item The first piece $\sigma_1$ of the boundary of a surface containing $D$ consists of the family of negative elliptic singular points $e_{\tau}$ of $\Sigma_\tau(\xi)$ which contains $e$ and $\tau$ is contained in a closed interval $J$ which is slightly larger that $[t_{min},t_{max}]$. When $\sigma_1$ is oriented from top to bottom it is positively transverse to $\xi$. 
\item The upper endpoint of $\sigma_1$ is connected to $A(\alpha_0)$ by a Legendrian arc in a surface sightly above $\Sigma_{t_{max}}$. Let $\lambda_4$ be one of the Legendrian curves connecting $e_{\tau}$ to $A'$. 

Similarly, the lower endpoint of $\sigma_1$ is connected to $A(\alpha_0)$ by a Legendrian curve $\lambda_2$ and we now use the orientation opposite to the orientation of $\lambda_2$ viewed as a leaf of the characteristic foliation.
\item For each point $e_\tau$ between the two endpoints of $\sigma_1$ we choose an arc in $\Sigma_\tau$ which connects $e_\tau$ to $A(\alpha_0)\cap\Sigma_\tau$ such that the part of the transverse to the characteristic foliation is connected (this part may be empty). If we orient all arcs such that the point to the negative elliptic singularity, then the arc is never tangent and anti-parallel to the characteristic foliation on $\Sigma_\tau$ except in $\Sigma_0$ where the arc is Legendrian.  
\item $\sigma_3$ is an arc in $A(\alpha_0)$ consisting of the endpoints of the arcs we have just picked. We orient $\sigma_3$ form bottom to top. Then $\sigma_3$ is positively transverse to $\xi$.  
\end{enumerate}
The concatenation  $\sigma$ of $\sigma_1,\lambda_2,\sigma_3,\lambda_4$ is a piecewise smooth curve whose smooth segments are positively transverse to $\xi$ or Legendrian and $\sigma$ bounds $D$. We orient $D$ so that $\sigma=\partial D$. Moreover, using the properties of $A$ and the characteristic foliations on $\Sigma_\tau$ it follows that all singularities of the characteristic foliation on $D$ are positive except one point in the interior of $D\cap A$. Moreover, $D\cap A$ is a circle in $D$ such that the characteristic foliation on $D$ points inwards. Since all singular points of the characteristic foliation on $D$ which do not lie in the disc bounded by $D\cap A$ are positive, the basin formed by all flow lines whose $\omega$-limit set is inside the disc bounded by $A\cap D$ is well defined and it yields an overtwisted disc.  
\end{proof}
The second condition of \lemref{l:overtwisted movie} can be achieved using the Legendrian realization principle (\lemref{l:LeRP}) if the curves $A\cap\Sigma_t$ are non-separating and the genus is at least $2$. If $\Sigma\simeq T^2$, then the results in \cite{gi-bif} show that the presence of a sheet bounding a solid torus without any further assumptions does not suffice to produce an overtwisted disc (the corresponding contact structures on $T^2\times[0,1]$ are tight, but virtually overtwisted).

\subsubsection{Transverse contact structures on $\Sigma\times[-1,1]$} \mlabel{s:transverse contact}

The purpose of this section is to prove that contact structures on $\Sigma\times[-1,1]$ which are transverse to the second factor are tight when the boundary does not have a neighbourhood with an obvious overtwisted disc. For this and for other purposes we give an efficient construction of contact structures transverse to the foliation $\II$ given by the second factor of $\Sigma\times[-1,1]$.  We shall assume that the genus of the underlying surface is at least two, the case of tori is simpler.  

The construction of contact structures transverse to $\II$ is explained in the following example which yields a contact structure on $\Sigma\times\R$  which is a complete connection of the $\R$-bundle because it is periodic with respect to a translation of the second factor.  
\begin{ex} \mlabel{ex:periodic cont}
Let $\Sigma$ be an oriented surface of genus $g\ge 2$. We fix two non-separating oriented disjoint closed curves  $\gamma',\gamma''$ and we choose four singular foliations $\FF_1,\ldots,\FF_4$ on $\Sigma$ such that all singularities have positive divergence as follows:
\begin{enumerate}
\item $\gamma'$ is a closed attractive leaf of $\FF_1$ with non-degenerate holonomy, $\gamma''$ is a curve with attractive holonomy on one side and repulsive holonomy on the other side such that the degenerate closed leaf marks the birth of a pair of parallel closed leaves on surfaces $\Sigma_t$ in $\Sigma\times(-\eps,\eps)$ which lie above $\Sigma_0$ (cf. \lemref{l:gi-birth-death}).  All leaves of the characteristic foliation whose $\alpha$-limit set  is $\gamma''$ accumulate on $\gamma'$ (except $\gamma''$ itself, of course) and except for the degenerate closed orbit $\FF_1$ is of Morse-Smale type. By \lemref{l:char fol det}, there is a contact structure $\xi_1$ on $\Sigma\times[-1,1]$ such that $\FF_1=\Sigma_0(\xi_1)$ and the only non-convex level is $\Sigma_0$. 
\item $\gamma''$ is a closed attractive leaf of $\FF_2$ with non-degenerate holonomy and $\gamma'$ is a closed degenerate leaf such that this marks the disappearance of a pair of closed leaves on $\Sigma_s$ of the contact structure $\xi_2$ on $\Sigma\times[-1,1]$ determined by $\FF_2=\Sigma_0(\xi_2)$. All leaves of $\FF_2$ which come from $\gamma'$ accumulate on $\gamma''$ (except $\gamma'$). So for $s<0$ there are two attractive closed leaves on $\Sigma_s$ while there is only one for $s>0$. 
\item  $\gamma''$ is an attractive closed leaf of $\FF_3$ while $\gamma''$ is a degenerate closed leaf, but in contrast to $\FF_2$ it now marks the birth of a pair of non-degenerate closed leaves.  Again all leaves whose $\alpha$-limit set is $\gamma'$ have $\gamma
''$ as their $\omega$-limit set (again except $\gamma'$). The corresponding contact structure on $\Sigma\times[-1,1]$ is called $\xi_3$, for $s<0$ there is only one attractive closed leaf on $\Sigma_s$  isotopic to $\gamma''$ while for $s>0$ there are two such leaves, one of them is isotopic to $\gamma'$ while the other is isotopic to $\gamma''$. 
\item  $\gamma'$ is a closed attractive leaf of $\FF_4$  and $\gamma''$ is a degenerate closed leaf which marks the cancellation of a pair of non-degenerate closed leaves of $\xi_4$. Again we require that all leaves of $\FF_4$ coming from $\gamma''$ to have $\gamma'$ as their $\omega$-limit set \end{enumerate}
For each contact structure $\xi_i, i=1,\ldots,4$, on $\Sigma\times I_i\simeq\Sigma\times[-1,1]$ the surface $\Sigma_t$ is convex except when $t=0$. In order to glue the two pieces such the resulting contact structure has no negative singularity a little bit of care is needed since the condition \eqref{e:order} concerning the the position of tangent space of sheets consisting of attractive closed leaves, the contact planes along these sheets, the tangent space of the surfaces and the vertical direction has to be satisfied. Isotoping the foliation $\FF_2$ such that the attractive closed leaves lie on the side of $\gamma'$ and $\gamma''$ which is opposite to the side determined by the coorientation of the leaves in the surface we can glue the two contact structures $\xi'_1$ and $\xi_2'$ which are restrictions of $\xi_1$ respectively the (isotoped) contact structure $\xi_2$ to $\Sigma\times I_1$ respectively $\Sigma\times I_2$ using \lemref{l:convex families}  such that the resulting contact structure $\xi_{12}$ on $\Sigma\times (I_{1}\cup I_{2})$ is transverse to the second factor. 

Similarly, one can now combine isotoped versions of $\xi_{12},\xi_3$ and $\xi_4$ to obtain a contact structure on $\Sigma\times\left(\cup_iI_i\right)\simeq\Sigma\times[0,1]$ transverse to the second factor such that the contact structure $\xi$ near $\Sigma_0$ coincides with the contact structure on $\Sigma_1$ when we use the second factor to identify these levels. 

In order to obtain a contact structure on $\Sigma\times\R$ which is transverse to the second factor and complete when viewed as a connection it suffices to glue infinitely many copies together. 
\end{ex}

\begin{lem} \mlabel{l:tight extension} 
Let $\Sigma$ be a closed surface of genus $g\ge 1$ and $\xi$ a contact structure on $\Sigma\times[-1,1]$ transverse to the fibers of the projection $\Sigma\times[-1,1]\lra\Sigma$ such that $\Sigma_{\pm 1}$ is convex and no component of the dividing set bounds a disc. Then $\xi$ is universally tight.
\end{lem}

\begin{proof}
Since $\xi$ is transverse to the foliation $\II$ defined by the second factor in $\Sigma\times[-1,1]$, it is automatically extremal, i.e. $|\langle e(\xi),[\Sigma]\rangle|=2g-2$,  and we coorient $\xi$ using the second factor. In particular, $\langle e(\xi),[\Sigma]\rangle = 2-2g$.  By assumption, no component of the dividing set  on $\Sigma_{\pm 1}$ bounds a disc. Hence the dividing curves on each of the surfaces $\Sigma_{\pm 1}$ come in pairs, each pair bounding an annulus containing a closed attractive leaf.

The idea of the proof is to embed $(\Sigma\times[-1,1],\xi)$ into $(\Sigma\times\R,\widehat{\xi})$ such that $\widehat{\xi}$ is transverse to the foliation $\widehat{\II}$ corresponding to the $\R$-factor and $\widehat{\xi}$ is a complete connection on $\Sigma\times\R\lra\Sigma$. \thmref{t:tight connection} then implies that $\widehat{\xi}$ is universally tight, and hence the same is true for $\xi$ (the embedding of $\Sigma\times[-1,1]$ maps $\Sigma_0$ to $\Sigma\times\{0\}\subset\Sigma\times\R$). 

We attach layers of contact structures obtained as in \exref{ex:periodic cont} in order to successively reduce the number of connected components of the dividing set and to arrange that in the end the only attractive closed curve is non-separating in 
$\Sigma$.  Using \lemref{l:convex families} we can modify the characteristic foliations such that at each step of the elimination no new attractive closed curves appear. Some care is needed when we want to eliminate a component which separates the surface into two pieces. In this situation one first introduces a non-separating closed repulsive curve using \lemref{l:LeRP}. Using the folding-procedure we obtain a contact structure with an attractive closed leaf isotopic to the repulsive curve.  

We end up with a contact structure on $\Sigma\times[-2,2]$ which is transverse to the second factor, has convex boundary and the characteristic foliation on the boundary has exactly one non-separating attractive curve. We then attach infinitely many layers  obtained in \exref{ex:periodic cont}.
\end{proof}

\begin{rem} \mlabel{r:disc bounding tight}
The condition that no component of the dividing set of $\Sigma_{\pm 1}$ bounds a disc clearly cannot be omitted. However, if there is one component $\gamma$ of the dividing set which bounds a disc, then we consider the case that $D_\gamma$ contains no other component of the dividing set. 

Then there is an attractive closed leaf $\beta$ bounding a larger disc $D_\beta$ containing $D_\gamma$ in its interior since the interior of $D_\gamma$ necessarily contains a singular point which is positive by transversality. Now consider the basin of all leaves of $\Sigma_{\pm 1}(\xi)$ which leave $D_\gamma$ through $\gamma$. The closure of the basin may not contain any singularities at all (since they would have the opposite sign as the singularities inside the disc). 

Therefore the basin has Legendrian boundary and is again a disc. The boundary is an attractive closed orbit. 
\end{rem} 

\subsection{Boundary elementary contact structures}

Let $\Sigma$ be  a closed oriented surface of positive genus $g$ and $\xi$ a contact structure on $N=\Sigma\times[0,1]$. For our purposes it suffices to consider only the case when $\partial N$ is convex. We require that the contact structure is extremal in the sense that 
\begin{equation} \label{e:max Euler}
\chi(\Sigma)=2-2g=\langle e(\xi),[\Sigma] \rangle 
\end{equation}
where $e(\xi)$ is the Euler class of $\xi$ viewed as an oriented vector bundle. The Thurston-Bennequin inequalities \eqref{e:tb} imply that the left hand side of \eqref{e:max Euler} cannot be bigger than the right hand side provided that $\xi$ is tight. 

From now on we assume that $\xi$ is tight. By \corref{c:extremal} the surface $\Sigma^-$ is then the union of annuli whenever $\Sigma$ is convex. Each such annulus contains a Legendrian curve which is the $\omega$-limit set of all leaves entering the annulus. Let $\beta$ denote such a curve, we will sometimes refer to such curves as {\em sinks}. 

The following definition is an adaptation of Definition 3.14 of \cite{gi-bif} for our situation. 
\begin{defn} \mlabel{d:elementary}
A contact structure is {\em boundary elementary} with respect to the product decomposition $\Sigma\times[0,1]$ of $N$ if for each annulus of $\Sigma_i^-, i=0,1,$ containing the sink $\beta$ there is an annulus $A _\beta$ which is foliated by Legendrian curves in $A(\beta)\cap \Sigma_t$ so that $\beta\subset \partial A(\beta)\subset \partial N$. 
\end{defn}

Compared to Giroux's definition in \cite{gi-bif} of elementary contact structures there are two differences:
\begin{itemize}
\item[(i)] If $\xi$ is elementary in the sense of \cite{gi-bif}, then this has consequences for all closed leaves of characteristic foliation on $(\partial N)(\xi)$. \defref{d:elementary} requires only the existence of some repulsive closed leaves of the characteristic foliation on  $(\partial N)$.
\item[(ii)] \defref{d:elementary} does not put restrictions on the characteristic foliation of all surfaces $\Sigma_t$ in the interior of $N$. 
\end{itemize}

Given a contact structure on $N$ we will need to isotope $\xi$ so that it becomes boundary elementary. This is relatively easy to achieve when $\Sigma=T$ is a torus because if $T_t(\xi)$ intersects a sheet $A$ in a homotopically non-trivial curve, then by the theorem of Poincar{\'e}-Bendixon $T_t(\xi)$ has no non-trivial recurrent leaf since the complement of $A\cap T$ is planar.  

\subsubsection{The pre-Lagrangian extension lemma}  \mlabel{s:prelag extend}

The following lemma will be the main tool for the extension of pre-Lagrangian surfaces.  
\begin{lem} \mlabel{l:parallel}
Let $\xi$ be a contact structure on $N=\Sigma\times[-1,1]$ such that $\Sigma_{\pm 1}$ is convex and $\partial \Sigma_t$ is an attractive Legendrian curve for all $t$. Assume that  $A$ is a sheet and the following conditions are satisfied. 
\begin{itemize}
\item[(i)] $\xi$ satisfies the extremal condition \eqref{e:max Euler}.
\item[(ii)] $A$ is transverse to $\Sigma_t$ for all $t\in[-1,1]$ and $A\cap\Sigma_t$ is a non-separating curve. All sheets which meet $\Sigma_{\pm 1}^-$ connect $\Sigma_{-1}$ and $\Sigma_{1}$. 
\item[(iii)] $\beta_t=A\cap\Sigma_t$ is either a closed attractive leaf or contains only negative singularities. 
\item[(iv)] The characteristic foliation on $\Sigma_{\pm 1}$ has two repulsive closed leaves $\beta_{\pm 1}'$ parallel to $A\cap\Sigma_{\pm1 }$  lying on the same side of $\beta_{\pm1}$. 
\item[(v)] The maximal sheet $A'_{\pm 1}$ containing $\beta_{\pm 1}'$ does not connect the two boundary components of $N$.
\end{itemize}
Then $\xi$ is isotopic to a contact structure $\widehat{\xi}$ such that the isotopy is the identity near the boundary and $A$ and there is a sheet $\widehat{A}$ connecting $\beta_{-1}'$ and $\beta_1'$. 
\end{lem}

The proof of this lemma is rather lengthy and will be given in \secref{s:proof extend}. Our main application of \lemref{l:parallel} is the following result which we will refer to as pre-Lagrangian extension lemma. 

\begin{lem}  \mlabel{l:completing sheets} 
Let $\xi$ be an extremal contact structure on $N=\Sigma\times[-1,1]$ such that the boundary is convex, $\partial \Sigma_t$ is an attractive Legendrian curve, $\xi$ is transverse to the foliation $\II$ corresponding to the second factor and there is a pair of isotopic closed leaves $\beta,
\beta'$ of $\Sigma_{-1}(\xi)$ such that $\beta$ is attractive and $\beta'$ is repulsive and the following conditions are satisfied:
\begin{itemize}
\item[(i)] The maximal sheet containing $\beta$ does not connect the two boundary components of $\Sigma$. 
\item[(ii)] $\beta'$ is not part of a properly embedded sheet in $N$. 
\item[(iii)] $\beta'$ lies on the side of $\beta$ opposite to the coorientation of $\xi$. 
\item[(iv)] For all other attractive closed leaves $\alpha$ of $\Sigma_{-1}(\xi)$ or $\Sigma_{+1}(\xi)$ the maximal sheet containing $\alpha$ connects the two boundary components of $N$. 
\item[(v)] No sheet meets the interior of the annulus bounded by $-\beta\cup\beta'$. 
\end{itemize}
Then there is a contact structure $\widehat{\xi}$  isotopic to $\xi$ relative to the boundary such that the sheet $\widehat{A}(\beta)$ containing $\beta$ is properly embedded and $\partial\widehat{A}(\beta)=-\beta\cup\beta'$. 
\end{lem}
There is an analogous lemma when $\beta\subset\Sigma_1$. In that case $\beta'$ is supposed to lie on the side determined by the coorientation of $\xi$ (in requirement (iii) above).  The dashed line in \figref{b:sheet} on p.~\pageref{b:sheet} corresponds to an extension of the sheet $A(\beta)$ where $\beta\subset\Sigma_{-1}$. 
\begin{proof}
Since $A(\beta)$ does not connect the two boundary components of $N$ we have
$$
t(\beta)= \max\{ t\in[-1,1]\,|\,A(\beta)\cap\Sigma_t\neq\emptyset\} <1.
$$
By \lemref{l:compactify by degen} and \lemref{l:gi-birth-death} $A(\beta)\cap\Sigma_{t(\beta)}$ is a closed degenerate leaf and there is a level $t'(\beta)$ such that for $t'(\beta)\le t<t(\beta)$ the characteristic foliation on $\Sigma_t$ contains a closed attractive leaf $\beta_t=A(\beta)\cap\Sigma_t$ and a closed repulsive leaf $\beta''$ parallel to $\beta$ which lies on the side of $\beta_t$ opposite to the coorientation of the contact structure. (So the pairs $\beta,\beta'$ and $\beta_t,\beta''_t$ are isotopic.) Note that $\beta_t$ is isotopic to $\beta$ as oriented curve since there are no negative singularities. 

After applying \lemref{l:isotope sheets} to all sheets, we may assume that the restriction of $\xi$ to $\Sigma\times[-1,t'(\beta)]$ satisfies the hypothesis of \lemref{l:parallel}. Then we obtain the desired contact structure $\widehat{\xi}$.  
\end{proof}

\subsubsection{The proof of \lemref{l:parallel}} \mlabel{s:proof extend}

Before we start with the proof we fix some notation: For fixed $t$ let $\Gamma_+$ be the graph on $\Sigma_t$ formed by positive singularities and stable leaves of positive singularities of $\Sigma_t(\xi)$  together with repulsive closed leaves, $\Gamma_-$ is the graph formed by negative singularities and unstable leaves of negative  singularities and attractive closed leaves. Thus $\Gamma_+,\Gamma_-$ is {\em not} the dividing set of any surface (we want to use the same notation as in \cite{gi-bif}). 

By the extremal condition \eqref{e:max Euler} the connected components of $\Gamma_-$ are either non-closed trees or homeomorphic to one circle with finitely many (eventually non-closed) trees attached to the circle. In order to simplify the presentation we assume that every repulsive closed leaf appearing in a path in $\Gamma_+$ is replaced by a pair of positive canceling singularities such that both unstable leaves of the new positive hyperbolic singularity come from the new elliptic singularity and the path passes through the new elliptic singularity. 

In the proof below we assume that $\xi$ is tight. If that proof does not work for an overtwisted given contact structure, then this is because there is an overtwisted disc in the complement of the sheets $A(\beta)$ and $A(\alpha)$ for $\alpha$ an attractive closed leaf of $\Sigma_{-1}(\xi)$. Then the classification of overtwisted contact structures (\thmref{t:ot class}) implies all claims of the lemma.

\begin{proof}[Proof of \lemref{l:parallel}]
Before we begin with the construction note that by \lemref{l:isotope sheets} we may assume that every sheet connecting a component of $\Sigma_{-1}^-$ to a component of $\Sigma_1^-$ is transverse to $\Sigma_t$ for all $t$. Furthermore, when we refer to the basin of $\beta_t$ we mean the basin lying on the same side of $\beta_t$ as $\beta'_{\pm 1}$.

If $\Sigma_t(\xi)$ is convex for {\em all} $t\in[-1,1]$, then by \lemref{l:rel convex} or \lemref{l:convex families}  we can isotope $\xi$ without changing $A$ so that the characteristic foliation of the isotoped contact structure on $\Sigma_t$ has a  closed repulsive leaf parallel to $\beta_t=A \cap \Sigma_t$ for all $t\in[-1,1]$. The collection of these repulsive leaves then provides the desired sheet $A'$. However, asking that $\Sigma_t$ is convex for all $t\in[-1,1]$ is far too restrictive and it suffices to arrange that the basin of $A(\beta)\cap \Sigma_t$ is compact and does not contain a negative singularity. This will be achieved in two steps: In Step 1  we arrange that there are only finitely many non-convex levels and after that, in Step 2, we deal with each non-convex level individually.  In Step 3 we apply \lemref{l:rel convex} to construct the desired extension of $A(\beta)$. 

{\bf Step 1:} 
According to \lemref{l:finite} $\xi$ is isotopic to a contact structure  $\xi_{PB}$  (the subscript $PB$ refers to the Poincar{\'e}-Bendixon property) with the following properties.
\begin{itemize}
\item There are only finitely many levels $t\in [-1,1]$ where $\Sigma_t(\xi_{PB})$ is not convex. Moreover, at all these levels a single retrograde saddle-saddle connection is responsible for the non-convexity and $\Sigma_t(\xi_{PB})$ has no non-trivial recurrent leaves.
\item The   basin of $\beta_t$ is compact for all non-convex levels.
\end{itemize}
If $\Sigma_t(\xi_{PB})$ contains a retrograde saddle-saddle connection, then we denote the negative respectively the positive singularity participating in the retrograde connection by $h_-$ respectively $h_+$. These singularities persist on nearby surfaces and we will denote these singularities also by $h_\pm$. 

{\bf Step 2:} Let  $t$ be a non-convex level such that one of the singular points $h_+$ or $h_-$ of $\Sigma_t(\xi_{PB})$ is contained in the closure of the basin of $\beta_t$.  We will isotope $\xi_{PB}$ without creating new non-convex levels such that $\Sigma_t$ becomes convex  or the retrograde saddle-saddle connection does not interact with the basin of $\beta_t$. 

In the following we assume that all negative hyperbolic singularities of $\Sigma_{\tau}(\xi_{PB})$ except $h_-$  have been eliminated for $\tau$ close to $t$. Hence $\Sigma_{t}(\xi)$ has exactly two negative singularities (one of them is $h_-$, the other one is elliptic). 

{\bf Case A:} Both stable leaves of $h_-$ are contained in the boundary of the basin, i.e. $h_-$ is a pseudovertex of the basin of $\beta_\tau$ for $\tau\neq t$ sufficiently close to $t$ and isotopy type of dividing set of $\Sigma_\tau$ does not change when $\tau$ crosses $t$. 

If $h_-$ is a pseudovertex of the basin of $\beta_t$, then we apply \lemref{l:elim retrograde}: Let  $\nu',\nu''$ denote the unstable leaves of $h_+$. Let $\Gamma'_-$ respectively $\Gamma''_-$ be the connected component of $\Gamma_-$ which contains the $\omega$-limit set of $\nu'$ respectively $\nu''$. This configuration is shown in \figref{b:B}. 

\begin{figure}[htb]
\begin{center}
\includegraphics[scale=0.75]{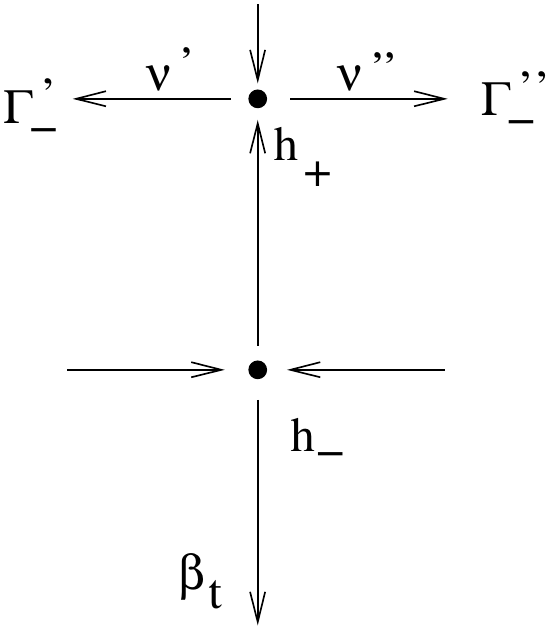}
\end{center}
\caption{Configuration in Case A}\label{b:B}
\end{figure}

Because $\xi_{PB}$ is extremal and tight, both $\Gamma'_-$ and $\Gamma''_-$ are trees since otherwise $\Sigma_\tau^-(\xi_{PB})$ would have components which are not diffeomorphic to annuli for $\tau$ close to $t$ by \lemref{l:switch}. For the same reason $\Gamma_-'=\Gamma_-''$ because otherwise the dividing set of $\Sigma_\tau^-(\xi_{PB})$ would contain a component bounding a disc for $\tau$ close to $t$. Hence we can indeed apply \lemref{l:elim retrograde} so that the number of non-convex levels is reduced by one. 

Now consider the case when the hyperbolic singularity $h_-$ is not a pseudovertex of the basin of $\beta_t$. Then the unstable leaf connecting $h_-$ to $\beta_\tau$ for $\tau\neq 0$  is the unstable leaf of $h_-$ which participates in the retrograde saddle-saddle connection when $\tau=t$. So $h_+$ is part of the basin of $\beta_t$ and both unstable leaves of $h_+$ accumulate on $\beta_t$ (from the side containing the basin under consideration). 

In this situation, both unstable leaves of $h_+$ accumulate on $\beta_t$ from the same side and the region bounded by the two unstable leaves of $h_+$ contains a positive elliptic singularity $e_+$ because $\xi_{PB}$ is tight.  Thus we can eliminate the singular points $e_+,h_+$ of $\Sigma_t(\xi_{PB})$ and on nearby surfaces. 

In this way we have reduced the number of non-convex levels by one. In particular, we did not loose the properties of the movie (like the Poincar{\'e}-Bendixon property, compactness of basins at non-convex levels or the discreteness of non-convex levels all of which are non-convex due to the presence of retrograde saddle-saddle connections). 

{\bf Case B:} Only one unstable leaf of $h_-$ is contained in the closure of the basin while the other one is not. Then $h_-$ is a corner of the basin and one unstable leaf of $h_-$ is connected to a positive hyperbolic singularity. Since generically there is at most one saddle-saddle connection $h_+$ is a pseudovertex of the basin of $\beta_t$. 

Let $\Gamma_+'$ be the connected component of $\Gamma_+$ containing $h_+$ and $\Gamma_-'$ the connected component of $\Gamma_-$ containing $h_-$, the unstable leaf of $h_-$ which does not participate in the retrograde saddle-saddle connection will be denoted by $\eta$. The unstable leaves of $h_+$ are denoted by $\nu,\nu'$ and $\nu$ accumulates on $\beta_t$ (cf. the left part of \figref{b:C1}). In order to obtain closed graphs we remove the retrograde saddle-saddle connection between $h_-$ and  $h_+$ from $\Gamma_\pm'$.  As before, $\Gamma_-'$ has to be a tree.

Assume that  $\nu$ accumulates on $\beta_t$ while the other unstable leaf $\nu'$ of $h_+$ does not accumulate on $\beta_t$ from the same side as $\nu$ (if that happens we are in the situation of the second part of Case A). There are two subcases depending on whether the $\omega$-limit set of $\nu'$ is contained in $\Gamma_-'$ or not.

{\bf Case B1:}  The easier case is when the $\omega$-limit set of $\nu'$ is not contained in $\Gamma_-'$. As in Case A the isotopy type of the dividing set of $\Sigma_\tau(\xi_{PB})$ does not change when $\tau$ passes $t$.  Therefore we can again apply \lemref{l:elim} to both singularities in $\Gamma'_-$ without losing the properties mentioned at the end of Case A. Thus we simply eliminated one non-convex level by removing the negative hyperbolic singularity participating in the retrograde connection. 

{\bf Case B2:} The much more intricate case is when the $\omega$-limit sets of $\eta$ and $\nu'$ are both contained in $\Gamma_-'$. In this situation, a pair of dividing curves appears respectively  disappears as $\tau$ crosses $t$ and the corresponding component of $\Sigma_\tau^-$ splits off respectively merges with the component of $\Sigma_\tau^-$ containing  $A\cap\Sigma_\tau$. 

{\bf Case B2.1:} There is a simple path $c$ in $\Gamma_+'$ with the following properties:
\begin{itemize}
\item[(i)] The unstable leaves of the positive hyperbolic singularities on $c$ lying on the same side of $c$ as the stable leaf of $h_+$ which is connected to $\Gamma_-'$ are also connected to $\Gamma_-'$. 
\item[(ii)] The $\omega$-limit set $\widehat{\beta}_t$ of the other unstable of $\widehat{h}_+$ is either different from $\beta_t$ or this unstable leaf accumulates on $\beta_t$ from the other side than the unstable leaf $\nu$ of $h_+$.  
\item[(iii)] $\widehat{h}_+$ is the only hyperbolic singularity on $c$ with this property. 
\end{itemize}
This configuration is schematically depicted in Figure~\ref{b:C1} where $c$ is the thickened curve. 
\begin{figure} 
\begin{center}
\includegraphics[scale=0.75]{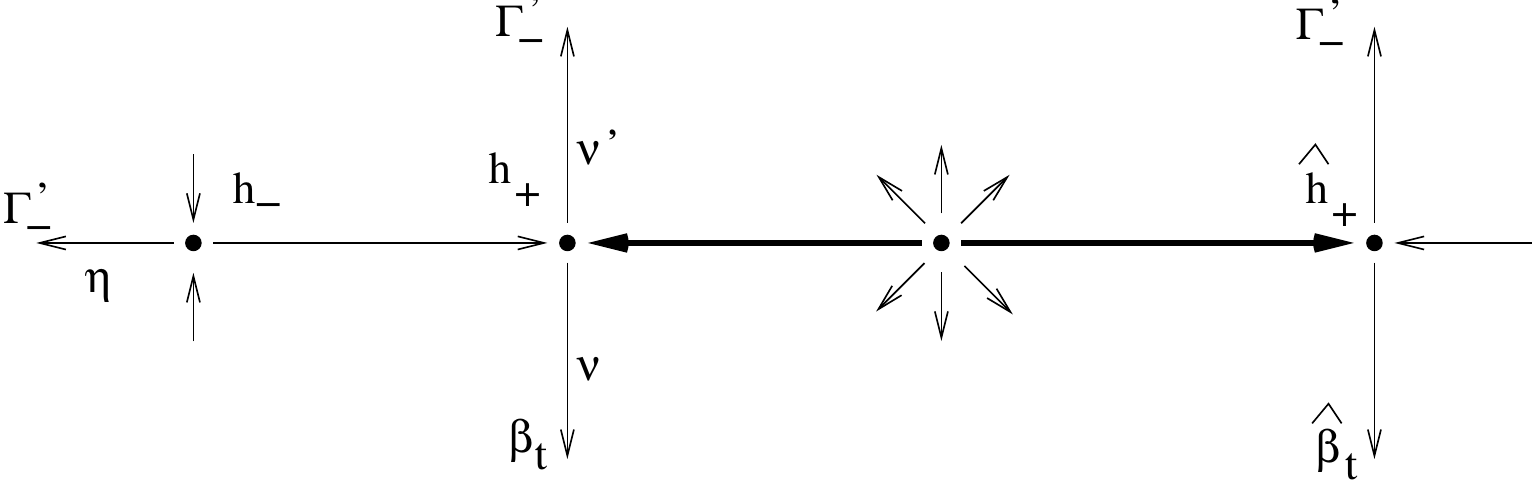}
\end{center}
\caption{Configuration in Case B2.1}\label{b:C1}
\end{figure}

Using \lemref{l:elim} (in particular part (ii) of that lemma) starting at the elliptic singularity closest to $\widehat{h}_+$ we can eliminate the hyperbolic/elliptic singularities along $c$ so that the retrograde connection between $h_-$ and the no-longer present $h_+$ 
is replaced by a retrograde saddle-saddle connection of $h_-$ with $\widehat{h}_+$.

The benefit for us is that the appearance/disappearance of the component of $\Sigma_\tau^-$ containing $h_-$ at  the non-convex level $\tau=t$  has no longer anything to do with the basin of $\beta_t$. Now the component of $\Sigma_\tau^-$ containing $h_-$ splits off from or merges with a component of $\Sigma_\tau^-$ different from the one containing $\Sigma_\tau^-\cap A$ or  this happens on the side of $\beta_\tau$ opposite to the side under consideration.  The construction does not affect the properties mentioned at the end of Case A.

{\bf Case B2.2:} If there is no path with the properties of $c$ in Case B2.1, then the basin of $\Gamma_+'$, i.e. the closure of leaves of the characteristic foliation whose $\alpha$-limit set is contained in $\Gamma_+'$, is a subsurface $S$ with two boundary components and corners. One boundary component is $\beta_t$ while the other boundary component contains $\Gamma_-'$ and $h_-$ is a corner. (Recall that by genericity we may assume that for all $\tau$, the characteristic foliation on $\Sigma_\tau$ has at most one saddle-saddle connection.) Also, if $\Gamma_-'$ would be contained in the interior of the closure of the basin of $\Gamma_+'$ then the current assumption (non-existence of a path like $c$ from Case B2.1) would imply that $\beta_t$ bounds a subsurface of $\Sigma_t$. But $\beta_t$ is non-separating. 

Let $\Sigma_\tau^-(h_-)$ be the connected component of $\Sigma_\tau^-$ which contains $h_-$ for $\tau$ close to $t$ for the relatively open subinterval $I(h_-)\subset (-1,1)$ where this region does not contain $\beta_\tau$. 

The boundary points of $I(h_-)$ correspond to non-convex levels and we eliminate {\em all} negative singularities from the characteristic foliation of $\xi_{PB}$ on these levels. According to \lemref{l:compactify by degen} the closure of $A'$ in $\Sigma\times[-1,1]$ is obtained by adding degenerate orbits to $A'$. (In order to obtain a smooth sheet we have to make sure that the graph formed by unstable leaves of negative singularities and elliptic singularities in $\Sigma_\tau^-(h_-)$ with $\tau\in I(h_-)$ is smooth but this can be achieved by modifications in neighbourhoods of elliptic singularities in $\Sigma_\tau^-(h_-)$.) 

While in the situation of Case B2.1 it was not possible in general to prevent the formation of non-trivial recurrent orbits or an infinite number of degenerate closed leaves now the fact that no leaf of the characteristic foliation can enter the surface $S$ through $\beta_t$ implies that we do not lose the Poincar{\'e}-Bendixon property at non-convex levels and only one degenerate closed leaf appeared. (At most one annulus in $\Sigma^-$ degenerates at a given non-convex level.) 

Our present goal is to isotope the contact structure such that $A'$  disappears completely.
 There are two possibilities:
\begin{enumerate}
\item Both boundary components of $A'$ meet one of the surfaces $\Sigma_\tau$ with $-1<\tau<1$ and the degenerate closed leaves of the characteristic foliations at these levels are parallel
\item Like (1), except that the degenerate closed leaves are anti-parallel.
\end{enumerate}
It will turn out that the second case contradicts the tightness of $\xi$. But first, we deal with the first two cases (which can be treated simultaneously) using an inductive procedure: 

There are two extreme situations, namely $\chi(S)=0$ (i.e. $S$ is an annulus) and $\chi(S)=\chi(\Sigma)$ (then $\Sigma\setminus S$ is an annulus) and the intermediate cases $-2\ge\chi(S)\ge\chi(S)+2$. 

The case $\chi(S)=0$ is straightforward: $\beta_\tau$ and $A'\cap\Sigma_\tau$ with $\tau\in I(h_-)$  bound an annulus $S_\tau\subset \Sigma_t$ such that both boundary curves are attractive closed curves (except in the boundary levels of $A'$).  After eliminating all superfluous negative/positive singularities in $S_t$ we obtain a family of repulsive closed curves separating the two boundary components. This allows us to eliminate $A'$ completely using \lemref{l:isotope sheets}. 
By this procedure we have reduced the number of non-convex levels. 

The case $\chi(S)=\chi(\Sigma)$ can be treated in the same fashion when one considers $S'=\Sigma_t\setminus\ring{S}$ or, in a more indirect fashion, inductively as the non-extremal cases. 

In order to treat the case when $\chi(S)$ is not $\chi(\Sigma)$ or $0$, we note that we may assume by induction (the induction starts with $\chi(S)=0$) that the lemma was already proved for surfaces with attractive Legendrian boundary of lower genus. This is possible since neither the lower or the upper basin of $A'\cap\Sigma_t$ can contain $\beta_t$. Thus we may cut $\Sigma\times[-1,1]$ along $A(\beta)\cap (\Sigma\times[-1,1])$. Then the pre-Lagrangian extension lemma can be applied to $A'$. Using \lemref{l:isotope sheets} $A'$ can be eliminated completely or moved out of $\Sigma\times[-1,1]$. We have thus reduced the number of non-convex levels in $[-1,1]$ and at each such level there is only a retrograde saddle-saddle-connection. 

As we shall see, we encounter only  case (1) from above when the contact structure is tight. Hence after finitely many steps we have eliminated all negative singularities in the closure of the basin of $\beta_t$ at non-convex levels without losing the Poincar{\'e}-Bendixon property at  non-convex levels and the non-convexity is again due only to retrograde saddle-saddle connections. The compactness of the basin at these levels is now easily arranged.  

We now show that the case (2) from above does not occur when $\xi$ is tight. Again this is by induction on the genus. We can cut $\Sigma\times[-1,1]$ along $A$ and we eliminate the negative singularities near both ends of $A'$ thus replacing the retrograde  saddle-saddle connections by a degenerate closed leaves of characteristic foliations. Then we can extend the pre-Lagrangian surface $A'$ to a pre-Lagrangian torus bounding a solid torus. Now we apply the pre-Lagrangian extension lemma to the basin of the attractive curves in $A'$ on the side opposite to $S$. Again this surface has lower genus and therefore we can find a pre-Lagrangian surface parallel to $A'$ which consists of repulsive closed curves in $\Sigma_t$.  But then the contact structure is overtwisted according to \lemref{l:overtwisted movie}. 

This shows that case (2) does not occur and after an isotopy of $\xi$ we may assume that the lower of $\beta_t$ does not contain a negative singular point at non-convex levels. Moreover, there are still only finitely many non-convex levels in $[-1,1]$ all of which correspond to retrograde saddle-saddle connections. The contact structure obtained after these isotopies is still denoted by $\xi_{PB}$.

{\bf Step 3:} We now construct the desired pre-Lagrangian extension of $A(\beta)$. For all $t$ so that $\Sigma_t(\xi_{PB})$ is not convex the boundary of the closure of the basin of $\beta_t$ does not contain a negative singularity. Let $V_t\subset\Sigma_t$ be a collar of $\beta_t$ lying in the basin (covered by $(Q_t=S^1\times[0,1],V_t=\emptyset,\alpha_t)$) such that $\Sigma_t(\xi_{PB})$ is transverse to $\partial V_t\setminus\beta_t$ and all leaves of the characteristic foliation entering $V_t$ accumulate on $\beta_t$. 

Fix a domain $F_t\subset\Sigma_t$ containing a neighbourhood of the basin of $\beta_t$ with $V_t$ removed such that \lemref{l:rel convex} can be applied to $F_t$. Such a neighbourhood exists because $\alpha(S^1\times\{1\}$) contains only positive singularities and no stable leaf of a positive hyperbolic singularity in $\alpha(S^1\times\{1\})$ comes from a negative singularity by construction.  We modify the contact structure on neighbourhoods of $F_t$ to obtain the attracting closed curves parallel to $\beta_t$ near levels where $\xi_{PB}$ is not convex. Once we have dealt with non-convex levels we apply \lemref{l:convex families} to obtain the desired contact structure $\widehat{\xi}$. 

In this last step we have used again that $\beta$ and hence also $\beta_t$ is non-separating (like in  \lemref{l:LeRP}). 
\end{proof}

\section{Transitive confoliations} \mlabel{s:transitive}

In this section we prove \thmref{t:unique} for transitive confoliations. Since a parametric version is not much more difficult we present that version. 

According to Gray's theorem every contact structure $\xi$ on a closed manifold $M$ has a $C^1$-neighbourhood so that every contact structure in that neighbourhood is isotopic to $\xi$. This follows from the fact that the contact condition is open in the $C^1$-topology. In particular, the contact structures interpolating between some contact structure $\xi'$ in the $C^1$-neighbourhood and $\xi$ can be chosen inside that neighbourhood.  V.~Colin has shown the following stability theorem for $C^0$-neighbourhoods. 
\begin{thm}[Colin, \cite{col}] \mlabel{t:colin stab}
Let $\xi$ be a contact structure on the closed $3$-manifold $M$. Then there is a $C^0$-neighbourhood $U$ of $\xi$ in the space of smooth plane fields so that every contact structure in $U$ is isotopic to $\xi$. 
\end{thm}
The family of contact structures constructed in the proof of this theorem does not necessarily stay in $U$. We now extend this theorem further to the case when $\xi$ is a transitive confoliation. 

\begin{thm} \mlabel{t:transitive}
Let $\xi$ be a transitive confoliation on a closed manifold $M$. Then there is a $C^0$-neighbourhood $U$ of $\xi$ such that the space of positive contact structures in $U$ is weakly contractible in the space of all contact structures on $M$. 
\end{thm}
The following two sections contain preliminaries for the proof of \thmref{t:transitive} which can be found in \secref{ss:transitive proof}. The structure of the proof is similar to Colin's proof of \thmref{t:colin stab}. Since this technique will be used later, we present it in a way that makes it amenable to further adaptation. We shall also use the following theorem of Varela \cite{varela}. 

\begin{thm} \mlabel{t:varela}
Let $\xi$ be a positive confoliation on $M$ which is somewhere non-integrable. Then there is a $C^0$-neighbourhood of $\xi$ so that every confoliation in that neighbourhood is somewhere non-integrable and positive.  
\end{thm}
Originally, this theorem was stated in \cite{varela} only for contact structures (and therefore there is no further reference to non-integrability). However, the proof uses only properties of characteristic foliations on the boundaries of a family of tubular neighbourhoods of a single knot transverse to $\xi$. It thus carries over immediately to yield \thmref{t:varela} since every open set of a contact domain, like $H(\xi)$, contains a transverse knot. 

\subsection{Adapted polyhedral decompositions} \mlabel{s:adapted poilydecomp}
Colin's proof of \thmref{t:colin stab} uses polyhedral decompositions which are adapted to $\xi$. We will make use of similar decompositions which we explain in this section.

\subsubsection{Darboux domains}

The original Darboux theorem for contact structures states that every positive contact structure is locally diffeomorphic to a domain in $(\R^3,\ker(dz+x\,dy))$. This is of course not true for confoliations, but we can still make the following definition.

\begin{defn} \mlabel{d:dd}
A pair $(P,V)$, where $P$ is a relatively compact set in $M$ and $V$ an open neighbourhood of $P$, is a {\em Darboux domain} if there is a bounded smooth function $f:\R^3\lra\R$ such that $\frac{\partial f}{\partial x}\ge 0$ and a confoliated embedding 
$$
\varphi : \left(V,\xi\eing{V}\right) \lra  (\R^3,\ker(dz+f(x,y,z)dy))
$$ 
so that the intersection of every flow line of the Legendrian vector field $\partial_x$ with the image of $V$ is connected. 


We say that a relatively compact set $P\subset M$ is a Darboux domain if there is a neighbourhood $V$ so that $(P,V)$ is a Darboux domain and $(x,y,z)$ are {\em Darboux coordinates}. 
\end{defn}
In particular, every subset of a Darboux domain is again a Darboux domain. The importance of this notion comes from the following stability property. 

\begin{lem} \mlabel{l:dd stable}
Let $(M,\xi)$ be a confoliated manifold and $P\subset V$ a Darboux domain. Then there is a $C^0$-neighbourhood $U$ of $\xi$ in the space of smooth plane fields so that $P$ is a Darboux domain for every  positive confoliation $\xi'$ in $U$.  
\end{lem}
\begin{proof}
Let $(P,V)$ satisfy the requirements of \defref{d:dd}. We fix a confoliated embedding $\varphi$ and a surface $D$ which meets every integral curve of $\partial_x$ in $V$ exactly once. In order to determine $U$ we fix an open neighbourhood of $V'\subset V$ of $P$ such that $\overline{V'}$ is compact. Then  $U$ is determined by the following requirements:
\begin{itemize}
\item[(i)] Every plane field in $U$ is transverse to $\varphi_*^{-1}(\partial_z)$ on $\varphi^{-1}(V')$.
\item[(ii)] For a smooth plane field $\xi'$ on $M$ let $X'$ be the projection of $\varphi_*^{-1}(\partial_x)$ along $\varphi_*^{-1}(\partial_z)$ to $\xi'$. We require that $X'$ is transverse to $V'\cap D$ and the flow lines of $X'$ starting at $V'\cap D$  are well defined as long as they stay in $V'$ and the image of $V'\cap D$ under the flow of $X'$ covers $P$.   
\end{itemize} 
The vector field $X'$ is smooth and therefore flow lines are uniquely defined and $U$ is open in the $C^0$-topology by standard theorems about the continuous dependence of solutions of ordinary differential equations on parameters (see chapter V in \cite{hartman}, for example). 

We still have to show that $P$ is a Darboux domain with respect to every positive  confoliation in $U$. We define $\varphi'=\varphi$ on $D\cap V'$. Using the flow of $X'$ we extend $\varphi'$ to a neighbourhood of $P$ so that oriented flow lines of $X'$ get mapped to oriented flow lines of $\partial_x$ in $\R^3$. The function $f'$ associated to $\xi'$ is then determined by $\varphi_*'(\xi')$. It satisfies  $\frac{\partial f'}{\partial x}\ge 0$ since $\xi'$ is a positive confoliation. 

We have defined $\varphi'$ on connected segments of flow lines of $X'$ starting at points of $V'\cap D$. Hence all requirements in \defref{d:dd} are satisfied except that $f$ is not yet defined on all of $\R^3$ but only on the image of $\varphi'$. But the construction allows one to choose an extension of $f'$ to $\R^3$ so that $dz+f'(x,y,z)dy$ defines a positive confoliation and $f'$ is bounded.  
\end{proof}

The neighbourhood $V'$ of $P$ so that $(P,V')$ is a Darboux domain depends on $\xi'\in U$. However, if we replace $P$ by a compact set $P'\subset V$ so that $P$ is contained in the interior then the above lemma shows that $(P,\ring{P'})$ is a Darboux domain for all $\xi'$ in a $C^0$-neighbourhood of $\xi$. 

According to \thmref{t:tight connection} every contact structure in $U$ is tight when restricted to $V$ where $(P,V)$ is a Darboux domain: The boundedness of the function $f$ implies that $\ker(dz+ f(x,y,z)dy)$ is a complete connection of the fibration $\R^3\lra\R^2$ given by $(x,y,z)\lmt (x,y)$.   
 
\begin{rem} \mlabel{r:Darboux for plane fields}
The notion of a Darboux domain can be extended to general smooth plane fields by omitting the requirement that $f(\cdot,y,z)$ is weakly monotone for all $(y,z)$. Then one can formulate the stability property for all smooth plane fields. 
\end{rem}

\subsubsection{Polyhedral decompositions adapted to $\xi$} \mlabel{s:adapted poly}

In this section we describe polyhedral decompositions adapted to a confoliation. Such decompositions are also used in \cite{col}.  Let $\xi$ be a positive cooriented confoliation on $M$ and $\II$ a line field transverse to $\xi$.

\begin{defn} \mlabel{d:transverse}
Given a polyhedron $P$ in $M$ we say that $\xi(x)$ is {\em transverse to} $\partial P$ in $x\in\partial P$ if
\begin{itemize}
\item $\xi(x)$ is transverse to a face of $P$ if $x$ is contained in the interior of that face,
\item $\xi(x)$ is transverse to all edges of $P$ whose closure contains $x$, and
\item if $x$ is a vertex, then for a germ of a surface $\Sigma_x$ tangent to $\xi(x)$ the set $(\Sigma_x\cap P)\setminus\{x\}$ is connected. If it is empty, then $x$ will be called {\em elliptic}. 
\end{itemize}
\end{defn}
Note that if $\xi(x)$ is not transverse to $P$ at a vertex $x$ of $P$, then $(\Sigma_x\cap P)\setminus\{x\}$ is not connected in general. For example, if this set has two connected components with non-empty interior, then $x$ should be thought of as a hyperbolic singularity. 
However, only the cases mentioned in \defref{d:transverse} will play a role in the following.


Let $P\subset (M,\xi)$ be a polyhedron in a confoliated manifold. If the boundary of a polyhedron $P$ is transverse to $\xi$, then one can define the characteristic foliation on $\partial P$ as follows. Since faces are assumed to be smooth they have a characteristic foliation. The characteristic foliation is oriented using the same conventions used for smooth surfaces. Where two faces meet along an edge we concatenate the corresponding oriented leaves. 
The leaves we obtain are piecewise smooth curves.  
At edges and non-supporting vertices the characteristic foliation is tangent to a pair of vectors, one of them tangent to one face adjacent to the edge or vertex while the other vector is tangent to the other face. 

The following notion (introduced by Thurston \cite{thurston}) will be applied to line fields and plane fields. We therefore formulate it in complete generality. 
\begin{defn} \mlabel{d:gen pos}
Let $\tau$ be a distribution of codimension $k$ on a $n$-manifold and $\R^n\supset P\hookrightarrow M$ an embedded polyhedron. Then $P$ is {\em in general position} with respect to $\tau$ if for all $x\in P$ and all $k$-subsimplices of $P$ the map 
$$
\R^n \lra \R^n/\tau_x = \tau^{\perp}_x
$$
restricted to the $k$-simplex is a diffeomorphism onto its image. 
\end{defn}
General position implies that $P$ is transverse to $\tau$ and it is a $C^0$-open condition. We will now state Thurston's jiggling lemma from \cite{thurston}. Note that this lemma uses triangulations (and not just polyhedral decompositions). Later $M$ will of course be a $3$-manifold, and $\tau=\xi$ a confoliation ($k=1$) or $\tau=\II$ a foliation of rank $1$ transverse to $\xi$ ($k=2$).   

\begin{lem}[Jiggling Lemma \cite{thurston}] \mlabel{l:jiggling}
Let $M$ be a compact manifold, $\mathcal{T}$ a triangulation and $\tau^{n-k}$ a continuous distribution of codimension $k$. Then there is a subdivision $\mathcal{T}'$ of $ \mathcal{T}$ such that after a small perturbation of the vertices one obtains a triangulation $\mathcal{T}''$ in general position with respect to $\tau$. 
\end{lem}

The way simplices are subdivided is essential. A possible subdivision is due to Thurston \cite{thurston}. Another method was used by Whitney and is described in \cite{whitney} on p.~358 where each $n$-simplex is decomposed into $2^n$ simplices. In both cases the simplices obtained by subdivision depend on the ordering of the vertices of a simplex $P$ at least if $n\ge 3$. Whitney's method in dimension $3$ goes as follows:

Let $P\subset \R^3$ be a simplex and $p_0,p_1,p_2,p_3$ its vertices. Let $p_{ij}$ be the midpoint between $p_i$ and $p_j$ with $p_{ii}=p_i$. The first Whitney subdivision of $P$ consists of the following simplices (with an ordering of the vertices):
$$
\begin{array}{llll}
p_0 p_{01} p_{02} p_{03}  & p_1p_{01}p_{02}p_{03} &   p_1p_{12}p_{02}p_{03}  &   p_2p_{12}p_{02}p_{03} \\
p_1p_{12}p_{13}p_{03}     & p_2p_{12}p_{13}p_{03}  &  p_2p_{23}p_{13}p_{03}  &   p_3p_{23}p_{13}p_{03}.
\end{array}
$$

Both subdivision schemes have the property that consecutive subdivisions of a simplex yield only finitely many subsimplices of $\R^n$ up to rescaling and translation. An important consequence is the following: If $\TT''$ is obtained as in \lemref{l:jiggling} and is in general position, then all simplices obtained by further Whitney subdivisions of $\TT''$ are still in general position with respect to $\tau$. It is therefore possible to apply \lemref{l:jiggling} to a finite collection of distributions (with varying codimensions).

According to \lemref{l:jiggling} there is a triangulation such that every simplex is in general position with respect to a confoliation $\xi$ and a foliation $\II$ of rank $1$ transverse to it. If we start with a triangulation such that every simplex is contained in a Darboux domain, then we can ensure in addition that all simplices are contained in a Darboux domain and in general position with respect to the corresponding vector field $\partial_x$ (cf. \defref{d:dd}). 

Thus we have established the existence of a triangulation satisfying the requirements of the following definition. It refers to polyhedra rather than simplices because there is one condition we want to impose later on the decomposition of $M$ and which will require the consideration of polyhedra, not only simplices. 
\begin{defn}   \mlabel{d:adapted poly}
A decomposition of $M$ into polyhedra is {\em weakly adapted to} $\xi$ {\em and} $\II$ if each polyhedron $P$ of the decomposition has the following properties. 
\begin{itemize}
\item[(i)] $P$ is in general position with respect to $\xi$ and $\II$ and there are exactly two singular vertices $x_1^P$ and $x_2^P$ which are both elliptic.
\item[(ii)] $P$ is a Darboux domain in general position with respect to the vector field $\partial_x$ from \defref{d:dd} and $P$ is homeomorphic to a ball. 
\item[(iii)] For $i=1,2$ a neighbourhood of $x_i^P$ in $P$ is contained in the half space determined by the plane $\xi(x_i^P)$ in some coordinate chart near $x_i^P$. (This property is independent of the choice of a chart).
\item[(iv)] The faces where $\II$ enters $P$ form a disc in $\partial P$. Moreover the intersection of every leaf of $\II$ with $P$ is connected.
\end{itemize}
We say that $\xi(x_i^P), i\in\{1,2\}$,  {\em supports} $P$ and $x_i^P$ is a {\em supporting vertex} of $P$. If $P$ lies on the side of $\xi(x_i^P)$ determined by the coorientation of $\xi$, then $x_i^P$ is {\em negative}, otherwise this supporting vertex is {\em positive}. 
\end{defn} 
Later we will often say that a polyhedron/polyhedral decomposition is adapted to $\xi$ and implicitly require that it is also adapted to a fixed line field $\II$ transverse to $\xi$.

By requirement (i) of \defref{d:adapted poly} the characteristic foliation $\partial P(\xi)$ has exactly two singular points corresponding to the supporting vertices $x_1^P, x_2^P$, and both are elliptic. On a neighbourhood of $P$ the confoliation can be viewed as connection on a fiber bundle $(x,y,z)\lmt(x,y)$ determined by Darboux coordinates. Since $\xi$ is a positive confoliation $\partial P(\xi)$ is spiraling away from $x_i^P$ (again in the weak sense if $x_i^P\not\in H(\xi)$) if this vertex is positive and towards $x_i^P$ if it is negative. Thus positive vertices are sources of $\partial P(\xi)$ while negative supporting vertices are sinks. 

By the Poincar{\'e}-Bendixon theorem all limit sets of leaves of $\partial P(\xi)$ are singularities, closed cycles (passing through singularities) or closed leaves, because $P$ is adapted to $\xi$ there are no closed cycles.  By \thmref{t:tight connection} the restriction of $\xi$ to a neighbourhood of each polyhedron is tight. Hence $\partial P(\xi)$ has no closed orbits if $\xi$ is a contact structure. If $\xi$ is a positive confoliation and $\partial P(\xi)$ has a closed leaf, then this closed leaf bounds a disc $D$ tangent to $\xi$ inside $P$ and the holonomy of $\partial P(\xi)$ near $\partial D$ is weakly attractive on the side of the positive supporting vertex while it is weakly repelling on the other side.

\begin{defn} \mlabel{d:graphical}
Let $\xi$ be a plane field on a $3$-manifold $M$ and $\II$ a line field transverse to $\xi$. Let $P$ be a polyhedron weakly adapted to both $\xi$ and $\II$. A {\em spine} of $P$ is a path consisting of edges of $P$  connecting the supporting vertices of $P$ such that the orientations induced by $\xi$ match and the image of this path under $P\lra P/\II_x, x\in P$ lies in the interior of the image of $P$. 

$P$ is {\em graphical} if the projection of the leaves of the characteristic foliation on $\partial P$ which start and end on the spine and determine the holonomy of the characteristic foliation with respect to the spine project under each map $P\lra P/\II_x, x\in P$, to a curve without self intersections (expect maybe at the two endpoints on the projection of the spine when the holonomy has a fixed point).
\end{defn}

A spine exists when $P$ is weakly adapted to $\xi$ and $\II$. \figref{b:spine} on p.~\pageref{b:spine} shows the projection of a neighbourhood of a supporting vertex to $\II(x_i^P)^\perp$ with a spine and a few images of leaves of the characteristic foliation.  

\begin{defn} \mlabel{d:adapted}
A polyhedral decomposition of $M$ which is weakly adapted to $\xi$ and $\II$ is {\em adapted to} $\xi$ {\em and} $\II$ if it satisfies the following conditions.
\begin{itemize}
\item[(i)] All polyhedra are graphical. 
\item[(ii)] For all vertices $x$ of the polyhedral decomposition there is at most one polyhedron supported by $\xi(x)$. 
\end{itemize}
\end{defn}

Assume that $P$ is a simplex of the decomposition obtained so far and consider a sequence of simplices $P^i\subset P^{i-1}$ obtained by successively applying Whitney's subdivision to $P$. 
The characteristic foliation on $P_i$ converges (after rescaling $\left(P\subset\R^3,\xi\eing{P},\II\eing{P}\right)$ such that the diameter of the rescaled copy of $P_i$ is $1$) uniformly to the foliation induced by the constant plane field $\xi_x$ with $\{x\}=\bigcap_iP_i$. Since the condition that a simplex is graphical is a $C^0$-open condition we can achieve that all polyhedra of a triangulation are graphical if we choose the subdivision sufficiently fine. 

In order to ensure (ii) of \defref{d:adapted} we modify the triangulation as in \cite{col}. In this step, the triangulation is modified in neighbourhoods of supporting vertices. In an inductive process one adds/removes tetrahedra from polyhedra of the decomposition (here the triangulation is replaced by polyhedral decomposition). This is described in detail in  \cite{col}, we therefore only indicate the main idea in the following figure (cf. Figure 1 in \cite{col}). In \figref{b:suppvert}, $x$ is a supporting vertex of $P_0$ but not of $P_1$. Then a piece is removed from $P_0$ and added to $P_1$ and we obtain $P_0'$ and $P_1'$. Now  $x'$ supports $P_0'$ and no other polyhedron of the modified decomposition. 

\begin{figure}[h]
\begin{center}
\includegraphics[scale=0.8]{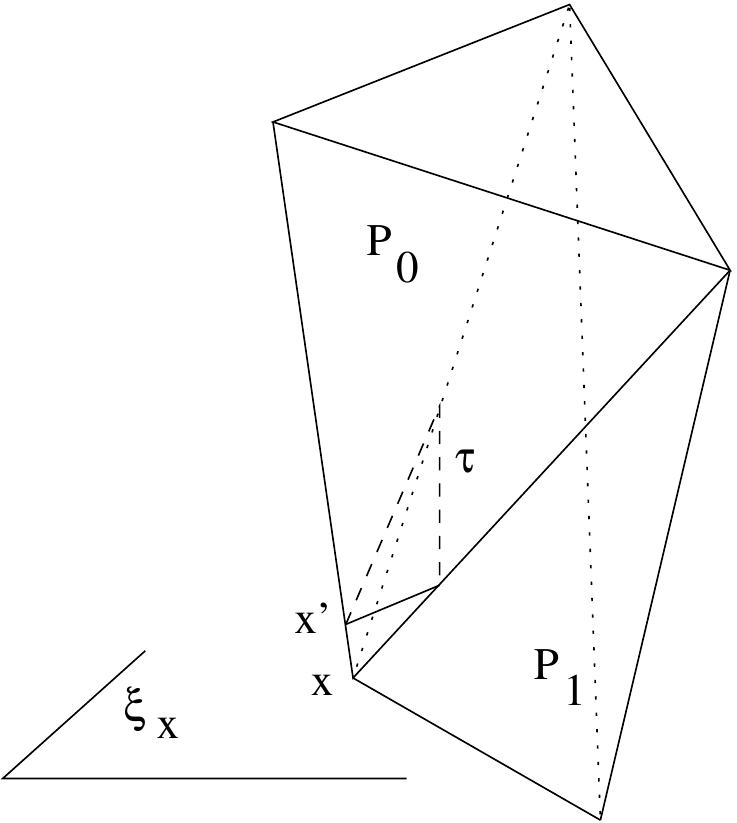}
\end{center}
\caption{Modification of polyhedra}\label{b:suppvert}
\end{figure}

Considering the various cases (one half space of $T_x\II$ is contained in $P$ or not) and choosing the segment $\tau$ close enough to other edges of $P$ one sees that the properties (i),\ldots,(iv) of \defref{d:adapted poly} can be preserved by this construction. Thus we have proved the following lemma (essentially due to Colin). 

\begin{lem} \mlabel{l:adapted triang} 
Let $\xi$ be a confoliation on a $3$-manifold and $\II$ a foliation of rank $1$ transverse to it.
Then there is a polyhedral decomposition of $M$ adapted to $\xi$ and $\II$. 
\end{lem}

Now consider plane fields $\zeta$ which are sufficiently close to $\xi$ to ensure that $P$ is still adapted to $\zeta$ and $\II$. Moreover, we require that $\zeta$ is transverse to $\II$ and the characteristic foliation on $\partial P$ has decreasing holonomy (this is automatic if $\zeta$ is a contact structure). 

Under these circumstances, \thmref{t:fill ball} implies that there is a tight contact structure $\xi'$ on $P$, unique up to isotopy, such that $\partial P(\zeta)=\partial P(\xi')$. In our situation we want to keep $\xi'$ transverse to $\II$ and the purpose of the condition that $P$ is graphical is to ensure that such an extension $\xi'$ can be constructed quite easily. 

\begin{lem} \mlabel{l:hor fill ball} 
Let $\zeta$ be a plane field, $\II$ a line field transverse to $\zeta$ and $P$ a polyhedron adapted to $\zeta$ and $\II$ such that every leaf of $\partial P(\zeta)$ spirals from the positive supporting to the negative supporting vertex.
   
Then there is a contact structure $\xi'$ on $P$ transverse to $\II$ such that $\partial P(\zeta)=\partial P(\xi')$. It is unique up to homotopy through contact structures in that class. 
\end{lem}

\begin{proof}
Given $\zeta$, first construct a foliation on $\partial P$ by circles transverse to $\partial P(\zeta)$ (of course, $x_1^P,x_2^P$ are singular points of this foliation). This is possible since we assume that the characteristic foliation of $\zeta$ on $\partial P$ does not have closed leaves. In order to ensure that these circles bound discs which are also transverse to $\II$ we proceed as follows: Pick a spine of $P$ and starting at the spine choose the transverse curves very close to $\partial P(\zeta)$ until one returns to the spine. Then connect the endpoints of the arc constructed so far using an arc close to the spine (and therefore also transverse to $\zeta$). 

The resulting circles project to simple closed curves in $\zeta_x^\perp$ since we assumed that $P$ is graphical. Therefore the circles bound discs transverse to $\II$. In this way we obtain a foliation of $P$ by discs transverse to $\II$ and the boundary of each disc is transverse to the characteristic foliation on the boundary.  Now pick a curve transverse to the discs connecting $x_1^P$ and $x_2^P$. The intersection points of this arc with the discs serve as midpoints of the discs. Then a contact structure which is tangent to a radial line field on the discs is obtained by twisting the tangent plane around the radial line field starting at the center of each disc. (This can be done in such a way that a given  contact structure near the supporting vertices is extended.) 
\end{proof}


\subsection{Ribbons}

The proof of Colin's stability result \thmref{t:colin stab} does not carry over immediately to \thmref{t:transitive} because $\partial P(\xi)$ can have closed leaves for some polyhedron $P$ of the decomposition if $\xi$ is a confoliation while all leaves of the characteristic foliation pass from the source to the sink if $\xi$ is a contact structure. Characteristic foliations with the latter property are rather stable under $C^0$-perturbations among contact structures and this stability is used in Colin's proof of his stability theorem (\thmref{t:colin stab}). The goal of the construction presented in this section is to modify (depending on the choice of a nearby contact structure $\xi_s$) a given polyhedral decomposition which is adapted to $\xi$ and $\II$ in order to ensure that the characteristic foliation of $\xi_s$ on the boundary of modified polyhedra does not have closed orbits. 

Since the confoliation is transitive we could isotope the polyhedral decomposition so that 
\begin{itemize}
\item the interior of no isotoped polyhedron contains an integral disc with boundary on the polyhedron, and
\item all supporting vertices lie in the interior of $H(\xi)$. 
\end{itemize}
After this isotopy, Colin's proof of \thmref{t:colin stab} yields a proof of \thmref{t:transitive}.   

In order to have a proof of \thmref{t:transitive} which applies to more general - and more interesting - situations we formalize similar isotopies in terms of ribbons attached to $\partial P$ in the context of transitive confoliations. This will be adapted to the case of non-transitive confoliations later.

\subsubsection{Definitions}

Let $P$ be a polyhedron of the polyhedral decomposition adapted to $\xi$ and a fixed line field $\II$ transverse to $\xi$. 

\begin{defn} \mlabel{d:ribbon}
A {\em ribbon attached to} $P$ is a smooth embedding of a rectangle $\sigma\times[0,1]$ into $M$, where $\sigma=\sigma\times\{0\}$ is a compact interval in $\partial P$. We require that the embedding has the following properties:  

\begin{itemize}
\item[(i)] $\sigma\times\{0\}=\sigma$ is transverse to $\xi$ and $\sigma\times\{1\}\subset H(\xi)$. Moreover,
\begin{equation*} 
P\cap(\sigma\times[0,1])=\sigma. 
\end{equation*}
\item[(ii)]  The curves $\{z\}\times[0,1],z\in\sigma,$ are tangent to $X_\varphi:=\varphi^{-1}_*(\partial_x)$ where $\varphi$ denotes the embedding associated to the Darboux domain $(P,V)$ and they are transverse to $\partial P$. The ribbon $\sigma\times[0,1]$ and the curve $\sigma\times\{1\}$ are tangent to $\II$. 
\item[(iii)] There is an open neighbourhood $P\subset V$ so that $(P,V)$ is a Darboux domain and there is a surface $D\subset V$ intersecting each flow line of the vector field $\partial_x$ from \defref{d:dd} exactly once decomposing $\partial P$ into surfaces with boundary so that $\sigma$ does not meet $D\cap\partial P$.  
\item[(iv)] The ribbon is disjoint from the $1$-skeleton of the polyhedral decomposition except that $\sigma\times\{0\}$ may contain one supporting vertex of $P$. The intersection of a ribbon with the faces of polyhedra consists of arcs connecting the two  Legendrian curves which correspond to the endpoints of $\sigma$. 
\item[(v)] We further require that all ribbons are pairwise disjoint. 
\end{itemize}
\end{defn}
If several ribbons are attached to a polyhedron, then we use the same surface $D$ for all ribbons. 

On its way from the polyhedron $P$  to the contact zone $H(\xi)$ the ribbon  $\sigma\times[0,1]$ meets other parts of the $2$-skeleton of the polyhedral decomposition. Let $P^1$ be the first polyhedron the ribbon leaves. We view a copy of the remaining part of the ribbon which lies between $P$ and $\sigma\times\{1\}$ as a ribbon $\sigma^1\times[0,1]$ attached to $P^1$. (At this point we use (iv) from above which ensures that cutting a ribbon along an arc in the intersection of the faces with the ribbon decomposes the ribbon into two ribbons.) 

For later constructions it is useful to extend $\sigma^1\times[0,1] $ in the transverse direction, so we replace $\sigma^1\times[0,1]$ by a slightly larger ribbon which is still attached to $P^1$ and whose opposite end is still contained in $H(\xi)$. This extension is again denoted by $\sigma^1\times[0,1]$ (see \figref{b:ribbon}). We continue until we enter the polyhedron $P^C$ containing $\sigma\times\{1\}$ in its interior. At each step the ribbon that is attached to the polyhedra gets a little bit broader. However, in order to avoid problems with the $1$-skeleton, we do not allow that any of the ribbons $\sigma^1\times[0,1],\sigma^2\times[0,1],\ldots,\sigma^k\times[0,1]$ meets the $1$-skeleton of the polyhedral decomposition. We will sometimes denote the ribbon $\sigma\times[0,1]$ by $\sigma^0\times[0,1]$.

\begin{figure}[htb]
\begin{center}
\includegraphics[scale=0.8]{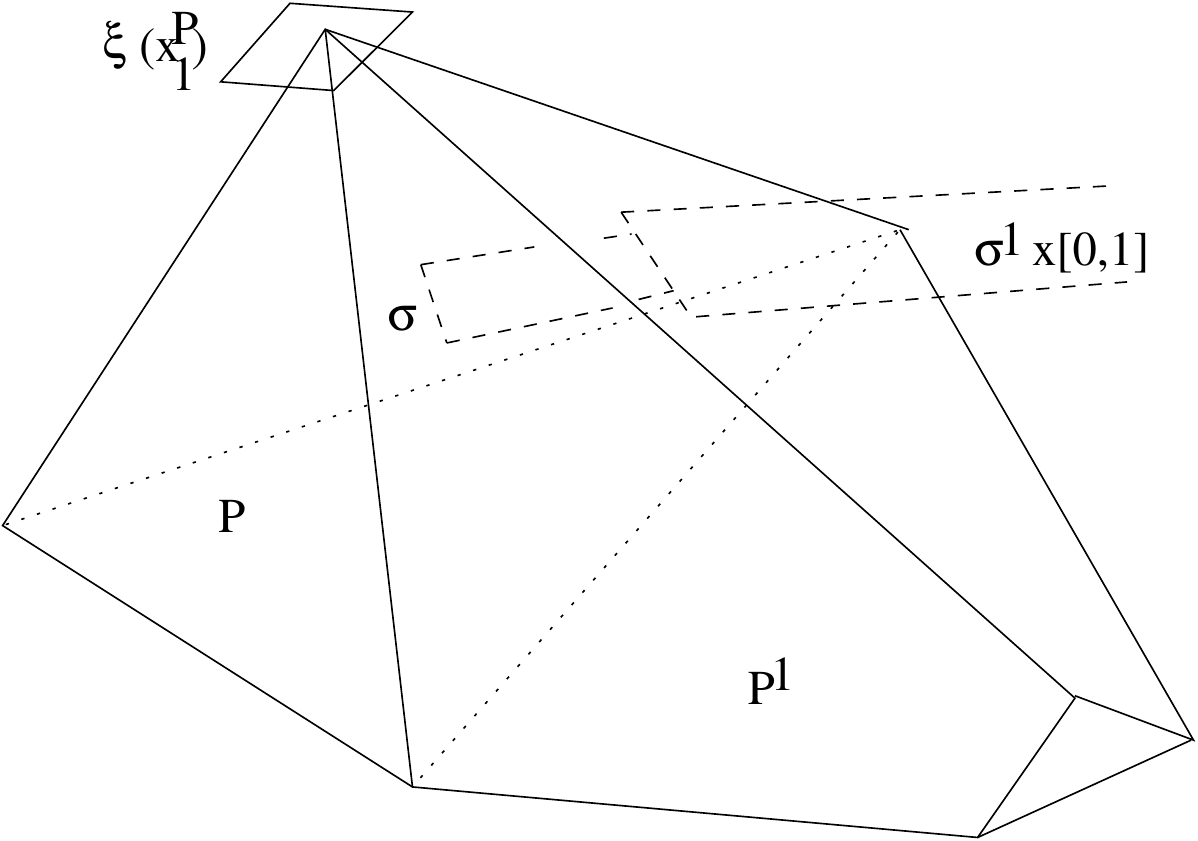}
\end{center}
\caption{The ribbons $\sigma\times[0,1]$ and $\sigma^1\times[0,1]$}\label{b:ribbon}
\end{figure}

Requirement (iii) in \defref{d:ribbon} is used in the proof of the following lemma.

\begin{lem} \mlabel{l:slit darboux}
Let $P$ be a polyhedron adapted to $\xi$ and $\II$.  Let $\sigma_j\times[0,1], j=1,\ldots,l$, are disjoint ribbons attached to $\partial P$, then 
$$
P\cup\bigcup_j(\sigma_j\times[0,1])
$$ 
is a Darboux domain in $M$. 
\end{lem}

\begin{proof}
Since only finitely many ribbons are attached it suffices to consider the case $l=1$.  By (iii) of \defref{d:ribbon} we can extend the embedding $\varphi\eing{P}$ to an embedding of $P^\sigma=P\cup_\sigma\sigma\times[0,1]$. We then extend the embedding $\varphi$ to an open neighbourhood of $P^\sigma$. Since $\sigma\times[0,1]$ is transverse to $\xi$ and compact, the function $f$ remains well defined and bounded. 
\end{proof}

Let $\sigma\times[0,1]$ be a ribbon attached to $P$ and let $\sigma^1\times[0,1],\ldots,\sigma^k\times[0,1]$ be the ribbons induced by $\sigma\times[0,1]$. 
At the end of the ribbon $\sigma^k\times[0,1]$ opposite to $P$ we fix a small cylinder $C$ contained in $H(\xi)$. In the following we specify some structures associated to such cylinders which will be useful in the sequel.  

Let $X$ be the $\xi$-Legendrian vector field from the proof of \lemref{l:slit darboux}. We fix a compact domain 
$C\simeq  I\times I \times I_{\sigma^{k}}$.  Here $I=[-1,1]$ and $I_{\sigma^k}$ are compact intervals tangent to $\II$. We view $C$ 
as  the total space of a fibration over the base $I\times I$, the bundle map being the projection onto the first two factors. We assume that the following conditions are satisfied. 
\begin{enumerate}
\item 
$C$ is contained in the support of $X$ and $\sigma^k\times\{1\}\subset C \subset H(\xi)$. 
\item The first factor in 
$C\simeq I\times I\times I_{\sigma^{k}}$ is tangent to $X$, the last factor is transverse to $\xi$ and tangent to ${\sigma^{k}}\times[0,1]$.  Each fiber of 
$C$ is either disjoint from the ribbon $\sigma^k\times[0,1]$ or the intersection of the fiber with the ribbon is connected and contained in the interior of the fiber.
\item The final assumption on 
$C$ is that the holonomy along the boundary of the base (with respect to both orientations) is defined for all points of $(\sigma^k\times[0,1])\cap\partial C$.  
The projection of $(\sigma^k\times[0,1])\cap\partial C$ 
to $[0,1]\times I$ is the base point for the definition of the holonomy.
\end{enumerate} 


We assume that $C$ 
is contained in a longer cylinder $\widehat{C}\subset H(\xi)$ so that the holonomy along the boundary of the base (with respect to both orientations) is defined for every point $x\in(\sigma^k\times[0,1])\cap \partial \widehat{C}$. Because that interval is compact, there is a number $\delta>0$ so that the difference between $x$ and both its preimage and image under the holonomy are separated by an interval whose length is at least $2\delta$. 

\begin{lem} \mlabel{l:domain}
Let $C=I\times I \times \R\lra I\times I$ be the projection and assume that the fibers are transverse to a positive contact structure $\xi$ such that the foliation corresponding to the first factor is Legendrian and the holonomy 
$h: \R\lra\R$ along $\partial (I\times I)$ is well defined with respect to a fixed base point in $\{-1\}\times I$. Then for every function $g:\R\lra\R$ with 
\begin{equation} \label{e:domain}
h\le g\le \id
\end{equation}
there is a domain $C(g)\subset C$ containing the fiber over the base point  so that the holonomy of the characteristic foliation on $C(g)$ is $g$. Moreover, if $h\le g_1\le g_2\le\id$, then $C(g_2)\subset C(g_1)$. 
\end{lem}
\begin{proof}
Let $X$ denote the vector field tangent to the first factor of $C=I\times I\times \R$. We use the flow of $X$ to identify the front $\partial_+C=\{1\}\times I\times \R$ of $C$ with the back $\partial_- C=\{-1\}\times I \times\R$. Because $\xi$ is a contact structure the image $\LL_+$  of the characteristic foliation  on $\partial_+ C$ is transverse to the characteristic foliation $\LL_-$ on $\partial_-C$. We use the characteristic foliation on $\partial_-C$ and the flow of $X$ to identify all fibers with the fiber over the base point.

For $x$ in the fiber over the base point we move along the leaf of $\LL_+$ starting at the point in $\{(+1,+1)\}\times \R$ which corresponds to $x$ until this leaf intersects the leaf of $\LL_-$ coming from $\{(g(x),-1)\}\times\R\}$. Such an intersection point exists because of \eqref{e:domain} and is unique by transversality of $\LL_+$ and $\LL_-$. 

We now vary $x$ to obtain a domain in $\partial_+ C$ as the union of segments of $\LL_+$. Translating this domain inside $C$ using the flow of $X$ we obtain a domain in $C$. In order to make sure that points of this domain are connected to the fiber over the base point, we add $\partial_-C$ to the domain to obtain $C(g)$. This domain has the desired properties. The last  part of the statement is immediate from the construction. 
\end{proof}

\subsubsection{Attaching a full collection of ribbons}

\begin{defn} \mlabel{d:comp slice}
Given a polyhedral decomposition of $M$ adapted to $\xi$ we say that a pairwise disjoint collection $(\sigma_i\times[0,1])_{i=1,\ldots,l}$, of ribbons attached to polyhedra of the decomposition is {\em full} if for every polyhedron $P$ of the decomposition the following holds: After a spine $S_P$ of $P$ is fixed, for each segment of the characteristic foliation on $\partial P\setminus S_P$ respectively for each supporting vertex of $P$ there is $i\in\{1,\ldots,l\}$ so that $\sigma_i\times[0,1]$ is attached to $P$ and $\sigma_i$  intersects the segment of the characteristic foliation respectively $\sigma_i$ contains the supporting vertex. 
\end{defn}

\begin{lem} \mlabel{l:comp slice}
Every polyhedral decomposition adapted to a transitive confoliation $\xi$ and a line field $\II$ transverse to $\xi$ on a closed manifold admits a full collection of ribbons $\sigma_i\times[0,1], i=1,\ldots,l$, so that for every $i\in\{1,\ldots,l\}$ the arc $\sigma_i\times\{1\}$ is contained in $H(\xi)$ and disjoint from the $2$-skeleton of the polyhedral decomposition.
\end{lem}

\begin{proof}
For each $P$ we choose a spine $S_P$ and a pair of Legendrian curves connecting the supporting vertices to $H(\xi)$ and we use a flow to extend these curves to obtain ribbons $\sigma_i\times[0,1],i=1,2$, tangent to $\II$ (it may be necessary to extend the ribbon a little to ensure that it meets $P$ in a segment.) We require that the Legendrian curves avoid the $1$-skeleton (except $x_i^P$) and are transverse to all faces of polyhedra in order to satisfy (iv) of \defref{d:ribbon}.

In the same way one obtains a collection of ribbons $\sigma_i,i=3,\ldots,l$, such that $\cup_i\left(\sigma_i\times\{0\}\right)$ intersects every segment of the characteristic foliation on $\partial P\setminus S_P$ and the ribbons satisfy the conditions in \defref{d:ribbon} individually.


However, the ribbons may still intersect each other and we can assume that they do so transversely. We remove the intersections inductively starting with one ribbon $\sigma_1\times[0,1]$.  Then if the second ribbon $\sigma_2\times[0,1]$ meets $\sigma_1\times[0,1]$ as we move along $\sigma_2\times[0,1]$, we split the second ribbon into two or three ribbons. One of these ribbons is replaced by a ribbon running parallel to $\sigma_1\times[0,1]$  while the other  parts follow the original ribbon $\sigma_2\times[0,1]$. In this way we obtain a full collection of ribbons which are pairwise disjoint (of course the ribbons $\sigma^j_i\times[0,1]$ induced by the attachment of one ribbon $\sigma_i\times[0,1]$ are not pairwise disjoint). 
\end{proof}

Let $\sigma\times[0,1]$ be a ribbon attached to a polyhedron $P$ and $g : \sigma\lra\sigma$ a decreasing map with compact support in the interior which corresponds to the monodromy of the boundary of a domain $C(g)$ in the cylinder containing $\sigma\times\{1\}$. Fix a tubular neighbourhood $N(\sigma)\subset \partial P$ whose fibers are segments of leaves of $P(\xi)$ such that $N(\sigma)$  is disjoint from the $1$-skeleton of the polyhedral decomposition. (If the ribbon contains a supporting vertex this tubular neighbourhood has to be chosen as indicated by the dotted lines in \figref{b:spine}). 

\begin{figure}[htb]
\begin{center}
\includegraphics[scale=0.7]{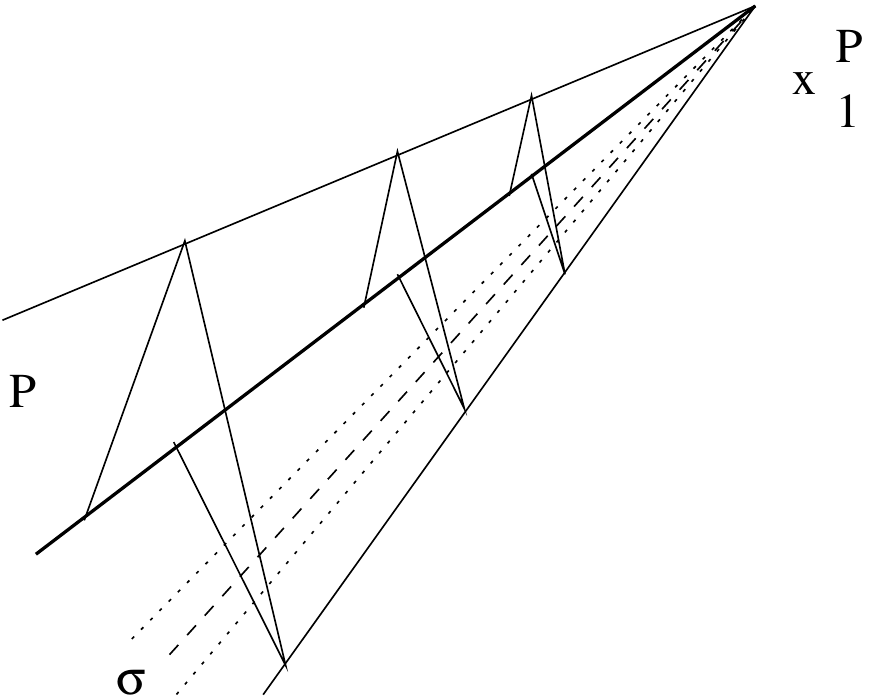}
\end{center}
\caption{Ribbon ending at a supporting vertex $x_1^P$ \newline projected to $\II(x_1^P)^\perp$}\label{b:spine}
\end{figure}

The union $P\cup\sigma\times[0,1]\cup C(g)$ can be thickened slightly to a domain with smooth boundary such that 
\begin{itemize}
\item the thickened domain is diffeomorphic to a ball, 
\item the characteristic foliation on the boundary has no singular points, and
\item the new boundary contains $\partial P\setminus N(\sigma)$.
\end{itemize}
We will not introduce any new notation for the thickening of $P\cup\sigma\times[0,1]\cup C(g)$.  
In order to identify $\partial P$ with $\partial (P\cup\sigma\times[0,1]\cup C(g))$ we extend the original ribbon $\sigma\times[0,1]$ to a family of ribbons which covers the thickening of  $\sigma\times[0,1]\cup C(g)$ such that all ribbons are tangent to $\II$. We then use a flow tangent to the characteristic foliation on the ribbons to push the new piece of the boundary to $N(\sigma)$ (this is shown in \figref{b:pushforw-xi}, the dashed curves correspond to the family of ribbons).

\begin{figure}[htb]
\begin{center}
\includegraphics[scale=0.7]{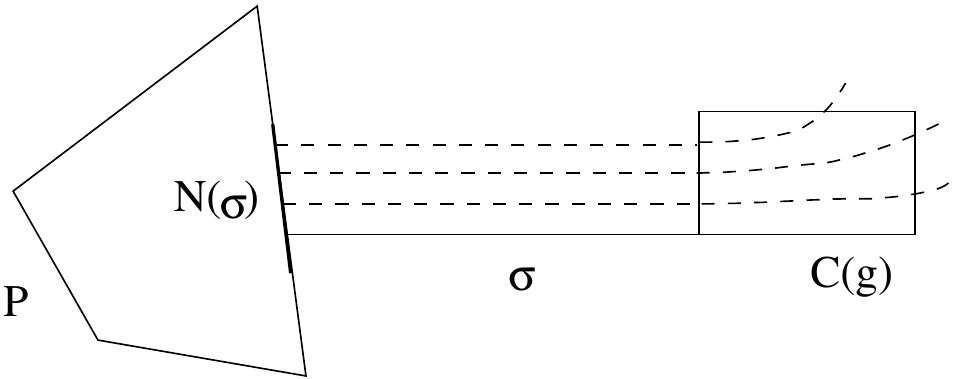}
\end{center}
\caption{Identifying $P$ and $P\cup\sigma\times[0,1]\cup C(g)$}\label{b:pushforw-xi}
\end{figure}

The push forward of $\xi$ under this isotopy remains transverse to $\II$ because it is tangent to a Legendrian vector field contained in surfaces tangent to $\II$. Now assume that $|g|$ is not too big and that 
$$
\mathrm{dist}\left(g(x),x_i^P\right)<k_{P,i}\cdot\mathrm{dist}\left(x,x_i^P\right)
$$
for a sufficiently small constant $k_{P,i}$ if $\sigma\times[0,1]$ contains a supporting vertex $x_i^P, i=1,2$.  Then the new characteristic foliation on $\partial P$ with respect to the push forward $\widehat{\xi}$ of $\xi$ is is still graphical. Since $g\le\id$, the new monodromy is decreasing by a larger amount than the original characteristic foliation.

However, since we have removed pieces of the polyhedra $P^1,\ldots, P^k$ which $\sigma\times[0,1]$ meets on its way to $H(\xi)$ the decrease in the monodromy of $\partial P$ is achieved at the expense of an increase of the monodromy in the monodromy of the boundary of other polyhedra. We now pick functions 
\begin{equation} \label{e:order g}
g_k\le \ldots \le g_1\le g\le\id
\end{equation} 
which can be realized as the holonomy of the characteristic foliation of a domain in $C$. Then \lemref{l:domain} implies $C(g_k)\supset\ldots\supset C(g_1)\supset C(g)$ and the thickenings can be chosen such that they satisfy the corresponding strict inclusion relations. The increase of the monodromy on boundaries of successive polyhedra is then outbalanced by the attachment of successive ribbons $\sigma^1\times[0,1],\ldots,\sigma^k\times[0,1]$ and domains $C(g_1),\ldots,C(g_k)$ to $P^1,\ldots,P^k$. (Later we will ask that the inequalities in \eqref{e:order g} between two functions are strict unless the functions both vanish.)

Thus the improvement of the holonomy on $\partial P,\partial P^1,\ldots,\partial P^k$ is achieved by removing pieces from the polyhedron $P^C $ containing the cylinder $C$. However, one can easily arrange that the monodromy on $\partial P^C$ with respect to $\widehat{\xi}$ is still decreasing by a definite amount when monodromy of the complement of $C(g_k)$ in $\partial C$ with $C(g_k)$ is decreasing by an amount which is bounded away from $0$. 


We attach all the ribbons together with the parts of the cylinders  obtained from \lemref{l:domain} to $P$.    For one  polyhedron $P$ we denote the result by $P^\sigma$. This is also meant to take into account the ribbons which are attached to other polyhedra and meet $P$.

\subsection{Proof of \thmref{t:transitive}} \mlabel{ss:transitive proof}

The proof of \thmref{t:transitive} has two main parts. We first determine the neighbourhood of $\xi$ and then we show that it has the desired properties. 

\subsubsection{Determining the neighbourhood of $\xi$ in \thmref{t:transitive}} \mlabel{ss:find eps}

Let $\xi$ be a transitive confoliation, $\II$ a foliation of  rank $1$ transverse to $\xi$. We fix a polyhedral decomposition of $M$ which is adapted to $\xi$ and $\II$ together with a full collection of ribbons $\sigma_i\times[0,1]$ and cylinders $\widehat{C}_i$ in $H(\xi)$. Let $\rho>0$ be such that the $3\rho$-ball around a supporting vertex meets only one ribbon, namely the ribbon $\sigma\times[0,1]$ containing the center of the ball, and the boundary of the ball meets $\sigma\times\{0\}$ in one point and the ribbon is transverse to the boundary. 

For each supporting vertex of a polyhedron $P$ we  add ribbons to the collection to ensure that every leaf of $\partial P(\xi)$ which meets $B_{3\rho}(x_i^P) \setminus B_{2\rho}(x_i^P)$ also intersects a ribbon.

When $\xi$ is replaced by a varying plane field $\zeta$, then the data associated to ribbons and Darboux domains vary as follows: 
\begin{itemize}
\item The vector fields defining the Legendrian curves on $\sigma_i^j\times\{t\},t\in[0,1],$ are replaced by their projection to $\zeta$ along $\II$. Hence the ribbons vary with $\zeta$ (with fixed initial conditions). Note that the ribbons $\sigma_i^j\times[0,1]$ which were induced by the original ribbons $\sigma_i\times[0,1]$ are now treated as being independent from each other. This is the reason why we required $\sigma_i^{j+1}\times[0,1]$ to be slightly broader than $\sigma_i^j\times[0,1]$.
\item The Legendrian vector fields $X_i$ on the cylinders $\widehat{C}_i$ where the ribbon $\sigma_i\times[0,1]$ ends are deformed in the same way. 

\end{itemize}
We do not introduce any new notation reflecting the variation of ribbons. 
The following conditions on $\eps>0$ ensure that the $\eps$-$C^0$-neigh\-bour\-hood $U_\eps(\xi)$ of $\xi$ has the stability property in \thmref{t:transitive}.
\begin{enumerate}
\item $\zeta$ is transverse to $\II$.
\item All flow lines of the deformed vector fields used in the above attachments and the associated constructions are well defined and have the same properties as before for all $\zeta$ which are $\eps$-$C^0$-close to $\xi$. In particular, $\sigma_i^{j}\times[0,1]$ meets $\sigma_i^{j+1}$ in the interior of this arcs.  
\item The polyhedral decomposition is adapted to $\zeta$ and $\II$ for all plane fields $\zeta$ which are $\eps$-close to $\xi$. In particular, the characteristic foliation remains graphical and each polyhedron with {\em all} ribbons and cylinders attached is contained in a Darboux domain 
for plane fields $\zeta$ which are $\eps$-$C^0$-close to $\xi$. 
\item If $\zeta$ is a confoliation, then it is positive (cf. \thmref{t:varela}). 
\item The holonomy of $\partial \widehat{C}_i(\zeta)$ respectively its inverse is either not defined or at least $\delta_C$-decreasing respectively increasing. 
\item For each polyhedron $P$ from the decomposition the characteristic foliation on the boundary has a monodromy which is not increasing by more than $\delta$ (we view the monodromy as a map from chosen spine of $P$ to itself). There is a positive constant $\kappa$ which depends only on $\xi$ and the geometry of the polyhedral decomposition near supporting vertices (in particular it depends on the angles between edges of polyhedra and angles between $\xi$ and edges) such that for each supporting vertex the monodromy of the characteristic foliation satisfies 
\begin{align*}
h_{P,i}(x)-x& <\kappa (x_i^P-x) & \textrm{if } &x_i^P\textrm{ is positive} \\ 
h_{P,i}(x)-x& <\kappa (x-x_i^P) & \textrm{if } &x_i^P\textrm{ is negative}
\end{align*}
on a $\rho$-ball around $x_i^P$. Here $\rho>0$ is so small that the $3\rho$-ball around $x_i^P$ intersects only one ribbon $\sigma\times[0,1]$ and $\sigma\times\{0\}$ meets the boundary of the $3\rho$-ball exactly once. The uppermost dotted line in \figref{b:support-attach} on p.~\pageref{b:support-attach} schematically depicts the graph of the upper bound for the monodromy of $\zeta$ where $\zeta$ is a smooth plane field $\eps$-close to $\xi$. (The horizontal axis is part of the spine of $P$ and the vertical axis measures the difference between a point on the spine and its image under the monodromy of the characteristic foliation of $\zeta$ on $\partial P$.)

\item Let $\sigma\times[0,1]$ be a ribbon which begins at a supporting vertex $x_i^P$. The constant $\kappa$ is so small that  after changing the characteristic foliation $\partial P(\zeta)$ in a conical neighbourhood (see \figref{b:spine}) of $\sigma\subset\partial P$ by a total amount of at most $2\kappa\mathrm{dist}\left(\cdot,x_i^P\right)$ the resulting characteristic foliation is still graphical. 
\item  Changing the characteristic foliation on $\partial P$ on the fixed neighbourhoods $N(\sigma^j_i)$ of $\sigma_{i}^j\times\{0\}$ (and hence outside of all $\rho$-balls around supporting vertices) so as to change the total monodromy by at most $2\delta$ one still obtains a graphical characteristic foliation on $\partial P$.
\item The contribution of a $\delta_C$-decreasing monodromy around $\partial \widehat{C}_i$ to the monodromy of the characteristic foliation on $\partial P$ turns that monodromy into a map which is at least $3\delta$-decreasing. (Here we use the deformed ribbons and their $\zeta$-characteristic foliations to compare the monodromy on a cylinder $\widehat{C}_i$ with that of a polyhedron which is connected to $\widehat{C}_i$ by the ribbon.)
\item The constants $\delta,\kappa$ satisfy
\begin{equation} \label{e:rhokappadelta}
\kappa\rho=\delta. 
\end{equation}
\end{enumerate}

This is a finite list of requirements which put restrictions on the $C^0$-distance of $\zeta$ from $\xi$. It can be summarized as follows: $\eps>0$ is chosen so small that ribbons and adapted polyhedral decompositions persist. The conditions on $\delta,\kappa,\delta_C$ ensure that the characteristic foliation on boundaries of polyhedra can be turned into foliations spiraling form the positive supporting vertex towards the negative supporting vertex after ribbons with suitable pieces of cylinders are attached to the polyhedra. As indicated by the extra conditions associated to supporting vertices more care is required near supporting vertices:  We will ensure that all plane fields are positive contact structures on (a priori unspecified) neighbourhoods of the supporting vertices. Furthermore, we require that characteristic foliations appearing in the constructions are such that all polyhedra can be filled by contact structures transverse to $\II$ provided that the holonomy of the characteristic foliation on the boundary is descending.

\subsubsection{Proof that the neighbourhood of $\xi$ has the desired property} \mlabel{s:transitive proof}

We will now prove that the $\eps$-neighbourhood of $\xi$ is weakly contractible. For future applications it is desirable to ensure that all contact structures/plane fields in the construction are transverse to $\II$. The proof without this control would be somewhat simpler (and the list of requirements in \secref{ss:find eps} would be shorter) and we could use \thmref{t:fill ball} instead of \lemref{l:hor fill ball}.

\begin{proof}[Proof of \thmref{t:transitive}]
We have to show that the $C^0$-neighbourhood of $\xi$ described in \secref{ss:find eps} has the following property: Given a compact family of contact structures $\xi_s,s\in S$, there is an extension of this family $\xi_s$ to a family of contact structures $\xi_{\widehat{s}}$ with 
$$
\widehat{s}=(s,t)\in \widehat{S}=S\times[0,2]/S\times\{2\}.
$$ 
Here we view $S$ as the subspace $S\times\{0\}$ of $\widehat{S}$.

The construction will be carried out in three main steps. In the first one, for $t\in[0,1]$, we will deal only with neighbourhoods of supporting vertices.  The second step takes care of the characteristic foliations on the  boundary of polyhedra. In both these steps we will be using plane fields which are not contact structures everywhere since we will only be interested in creating plane fields on $M$ together with adapted polyhedral decompositions such that the holonomy along the boundary of each polyhedron has no closed leaf. The conditions from \secref{ss:find eps} allow us to fill the interior of each polyhedra by contact structures transverse to $\II$. 
Together with Gray's theorem this concludes the proof. 

We start with the construction of the family $\xi_{\widehat{s}}$ around supporting vertices which is the most delicate step. Let $x_i^P$ be one of the supporting vertices of $P$. Recall that $P$ is a Darboux domain, so we are given a $\xi$-Legendrian vector field $X$ on $P$ and a surface $D\subset P$ intersecting every flow line of $X$ which meets $P$ exactly once. We denote the Darboux coordinates on $P$ by $(x,y,z)$ so that $X$ is the coordinate vector field of $x$.  By the conditions on $\eps$ the polyhedron $P$ is a Darboux domain for all $\xi_s, s\in S$, and the corresponding $\xi_s$-Legendrian vector fields $X_s$ vary continuously. The same is true for the characteristic foliations $D(\xi_s)$ and the Darboux coordinates $(x_s,y_s,z_s)$. The contact structure $\xi_s$ near $x_i^P$ is defined by $dz_s+f_s(x_s,y_s,z_s) dy_s$ with $\frac{\partial f_s}{\partial x_s}>0$.  

This data can be extended to families $X_{\widehat{s}}, f_{\widehat{s}}$ and  $(x_{\widehat{s}}, y_{\widehat{s}},z_{\widehat{s}})$ with $\widehat{s}=(s,t)$ with $t\in[0,1]$ defined on a neighbourhood of $x_i^P$ so that 
\begin{itemize}
\item $X_{\widehat{s}}$ is tangent to $\frac{\partial}{\partial x_{\widehat{s}}}$, 
\item $\frac{\partial f_{\widehat{s}}}{\partial x_{\widehat{s}}}>0$,  
\item $\xi_{\widehat{s}}$ is $\eps$-$C^0$-close to $\xi$, and 
\item $\xi_{(s,1)}$ does not depend on $s$. 
\end{itemize}
By compactness of $S$ we may choose the neighbourhood of $x_i^P$ to be a ball $V(x_i^P)$ which does not depend on $\widehat{s}$ and $V(x_i^P)$ is contained in the $\rho$-ball around the supporting vertex. 

We extend $\xi_{\widehat{s}}$ from $V(x_i^P)$ to a family of {\em smooth plane fields} $\zeta_{\widehat{s}}$ on $M$ such that 
\begin{itemize}
\item $\xi_s=\zeta_{(s,t)}$ outside of the $2\rho$-ball,
\item the restriction of $\zeta_{(s,t)}$ to the $\rho$-ball around $x^P_i$ is independent from $s$ when $t=1$, 
\item $\zeta_{\widehat{s}}$ remains $\eps$-$C^0$-close to $\xi$, and
\item the monodromy of the characteristic foliation of $\zeta_{\widehat{s}}$ on $\partial P$ is decreasing for all points of the spine whose distance from $x_i^P$ is at least $2\rho$.   
\end{itemize} 
Because $\zeta_{\widehat{s}}$ is only a plane field it may happen that the characteristic foliation on $\partial P$ has closed orbits in the $2\rho$-ball around $x_i^P$ (but not inside $V(x_i^P)$). This problem will be fixed by the attachment of a ribbon with a cylinder. 

For the ribbon $\sigma\times[0,1]$ containing $x_i^P$  we fix the following data. Let $I_{\sigma}$ be the reference fiber in the boundary of the cylinder where $\sigma\times[0,1]$ ends. We pick a collection of smooth functions $v_j: I_\sigma\lra (-2\delta,0], j=0,\ldots,k$, with the following properties.  
\begin{itemize}
\item $v_j$ has compact support in the interior of $I_\sigma$ unless $j=0$ (i.e. when one of the curves $(\partial\sigma)\times[0,1]$  starts at a supporting vertex of $P$).
\item $v_0\equiv \delta_C$ on all points which are connected to points in the $2\rho$-ball around $x_i^P$ by Legendrian curves in $\sigma\times[0,1]$. 
\item $-2\delta_C<v_k<v_{k-1}<\ldots <v_0\le 0$.  
\item $g_j:=\id+ v_j$ is a diffeomorphism of $I_\sigma$ for $j=0,\ldots,k$.
\end{itemize}
According to the conditions on $\eps$ it follows that when we attach the ribbon $\sigma\times[0,1]$ together with the domain $C(g_0)$ (more precisely an extension of $g_0$ with compact support in a slightly larger interval in the boundary of $C(g_0)$) to $P$ we  obtain a polyhedral complex 
such that the monodromy of the characteristic foliation is $2\delta$-decreasing for all points on the spine whose distance from $x_i^P$ is at most $2\rho$. 

Now replace $v_0$ by a function $\widehat{v}_0$ which has compact support in $I_\sigma$ such that 
 the holonomy of the characteristic foliation on the resulting polyhedral complex $P\cup \sigma\times[0,1]\cup C(\widehat{g}_0)$ is
\begin{itemize}
\item[(i)] strictly decreasing for all points in $V(x_i^P)$ (because $\zeta_{(s,t)}$ is a contact structure on $V(x_i^P)$ and $V(x_i^P)$ is a Darboux domain by construction),
\item[(ii)] at least $\delta$-decreasing for points in $B_{2\rho}(x_i^P)\setminus B_{\rho}(x_i^P)$,
\item[(iii)] at most $2\delta$-decreasing everywhere,
\item[(iv)] at least $\kappa\mathrm{dist}(x,x_i^P)$-decreasing on $B_\rho(x_i^P)\setminus V(x_i^P)$, and
\item[(v)] at most $2\kappa\mathrm{dist}(x,x_i^P)$-decreasing on $B_\rho(x_i^P)$.
\end{itemize}
Since $\widehat{v}_0$ has compact support, we can attach the ribbon $\sigma\times[0,1]$ and $C(\widehat{g}_0)$ to $P$ without having to worry about the supporting vertex.

We also replace $v_j,j>0$, by new functions so that $v_j<\widehat{v}_j<0$ on the support of  $v_j$ in order to ensure that the characteristic foliation of the pull back of $\zeta_{s,t}$ to the polyhedra remains graphical. 
(Of course we require $\widehat{v}_k\le\ldots\le\widehat{v}_1\le\widehat{v}_0\le 0$ with strict inequalities whenever one of the functions is not zero.) The conditions (ii) and (v) can be satisfied simultaneously by \eqref{e:rhokappadelta}.

Recall that the ribbons and the monodromy on the boundary of cylinders vary with $(s,t)$. Therefore we cannot choose the functions $\widehat{v}_j$ independently of $(s,t)$ but we let them vary with $(s,t)$ in such a way that for $t=1$ the monodromy of ribbons together with domains $C(\widehat{v}_j)$ (now viewed as a diffeomorphism of $\sigma_j\times\{0\}$) is independent from $s$ for $t=1$. 

Attaching the other ribbons $\sigma^j\times[0,1], j=1,\ldots,k$, together with $C(g_j)$ ensures that the holonomy of the characteristic foliation of $\zeta_{s,t}$ on the new polyhedra remains decreasing and the fact that the last polyhedron containing $C$ contains another cylinder which ensures that the holonomy of $P_C$ remains weakly decreasing after the pieces $C(g_1),\ldots,C(g_k)$ are removed from $C$ and added to other polyhedra. 

The following figure summarizes the situation near a supporting vertex.  The horizontal axis measures the distance of a point in the spine of $P$ from $x^P_i$ while the vertical axis corresponds to the displacement of a point by the holonomy of the characteristic foliation. The solid curve is a typical representative of the holonomy of $\zeta_{\widehat{s}}$ on $\partial P$ (before attachment of ribbons) while the dashed curve represents the effect of the attachment of $C(\widehat{g}_0)$ and the ribbon $\sigma\times[0,1]$. The dotted curves correspond to the conditions (i)--(v) and the thickened horizontal arc corresponds to points in $V(x^P_i)$. 
 
\begin{figure}[htb]
\begin{center}
\includegraphics[scale=0.7]{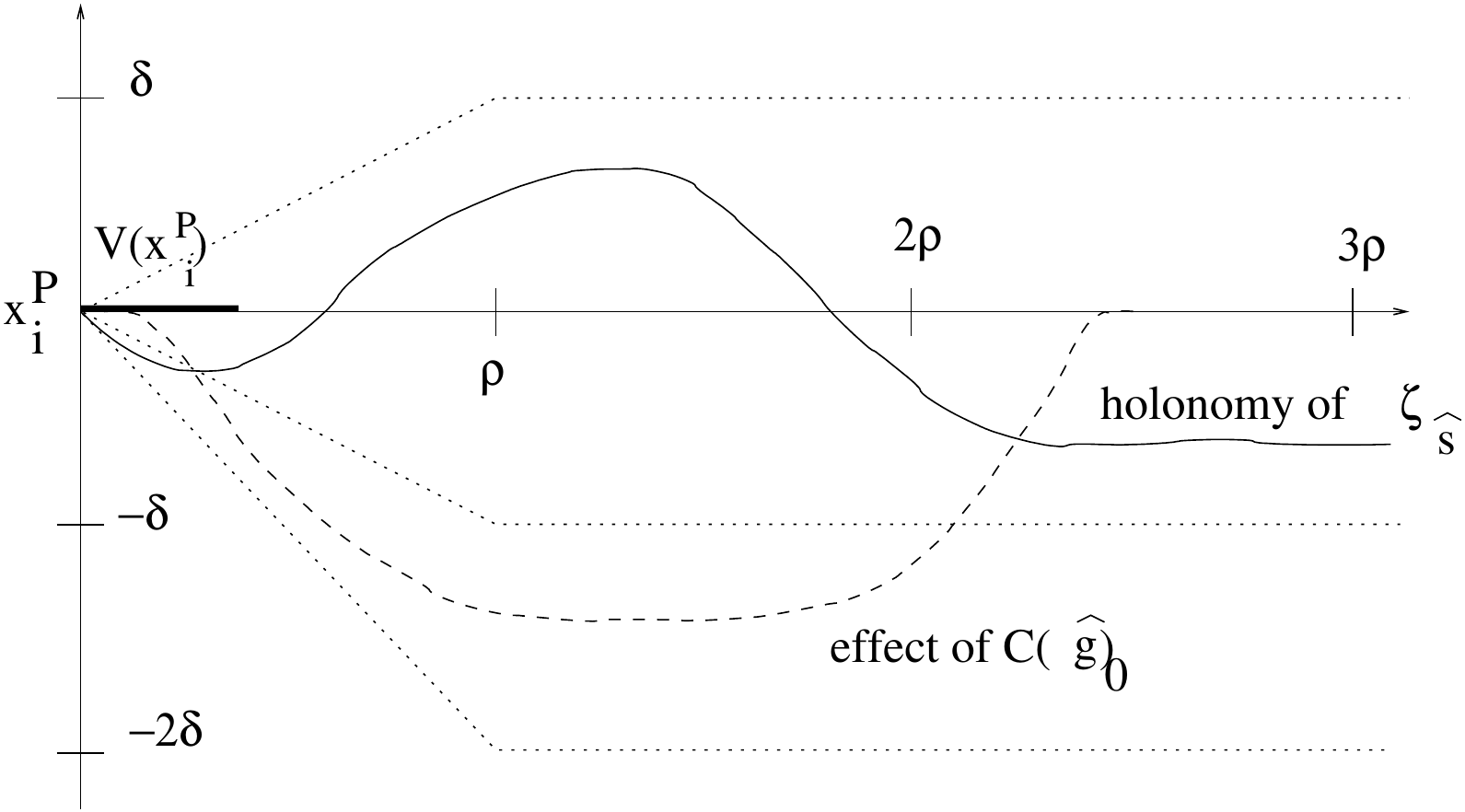}
\end{center}
\caption{Monodromy of $\zeta_{\widehat{s}}$ near supporting vertices} \label{b:support-attach}
\end{figure}

We treat all supporting vertices in the same fashion. The conditions (iii) and (v) imply that the characteristic foliation on $\partial P$ of the pull back of $\zeta_{\widehat{s}}$ remains graphical. Let $V=\cup_{i,P}V(x_i^P)$.

We next deform the plane field $\xi_{(s,1)}$ further keeping it constant on $V$. The deformation takes place only in a neighbourhood of the $2$-skeleton of the original polyhedral decomposition. We  choose the plane field $\zeta_{(s,t)}, s\in S,t\in[1,2],$ so that it remains $\eps$-close to $\xi$ and $\zeta_{(s,2)}$ does not depend on $s$ on the $2$-skeleton of the original decomposition while it coincides with $\xi_s$ near the enlarged cylinders in $H(\xi)$ where the ribbons end.  

According to our assumptions on $\eps$ we can attach ribbons and domains in $H(\xi)$ to all polyhedra such that the holonomy of the characteristic foliation on the boundary of the resulting polyhedra $P^\sigma$ is decreasing everywhere.  As before, we can achieve that the characteristic foliation on $P^\sigma$ does not depend on $s$ when $t=2$. We do not go through the details again since they already appeared in the above attachments of induced ribbons and we no longer have to worry about supporting vertices.

In this way we obtain a family of smooth plane fields $\zeta_{(s,t)}$ for $t\in[1,2]$ which is constant near supporting vertices when $t\ge 1$ and such that the characteristic foliation of $\zeta_{(s,t)}$ on the boundary of polyhedra with ribbons and domains attached to them do not depend on $s$ when $t=2$.




We now want to apply \lemref{l:hor fill ball}: From the pullback $\widehat{\zeta}_{\widehat{s}}$ of the smooth plane field $\zeta_{\widehat{s}}$ (i.e. the pullback of $\zeta_{\widehat{s}}$ by the flow of a Legendrian vector field which is guided by the ribbons attached to polyhedra) we will only remember the restriction of $\widehat{\zeta}_{\widehat{s}}$ 
\begin{itemize}
\item to the neighbourhoods $V(x_i^P)$ where $\zeta_{\widehat{s}}$ is a contact structure for all $\widehat{s}\in\widehat{S}$ and
\item to the $2$-skeleton of the polyhedral decomposition together with the attached ribbons and cylinders. 
\end{itemize}
For $t=0$ these characteristic foliations are induced by globally defined contact structures. In the deformations of the family of characteristic foliations we have ensured that the characteristic foliation on each polyhedron never has a closed leaf and is always graphical. 

By \lemref{l:hor fill ball} there is a family $\xi_{\widehat{s}}$ of contact structures on $M$ which coincides with $\widehat{\zeta}_{\widehat{s}}$ on 
\begin{itemize}
\item on the $2$-skeleton of the polyhedral decomposition, and 
\item on neighbourhoods of supporting vertices. 
\end{itemize} 
Moreover, $\xi_{\widehat{s}}$ can be chosen transverse to $\II$ for all $\widehat{s}\in\widehat{S}$. Since the boundary data provided by $\zeta_{(s,t)}$ is independent of $t$ for $t=2$, we can assume that $\xi_{(s,t)}$ has the same property. We have hence extended $\xi_s,s\in S$, to a family of contact structures $\xi_{\widehat{s}}$ with $\widehat{s}=(s,t)\in \widehat{S}$ and this finishes the proof. 
\end{proof}

\section{Exceptional minimal sets} \mlabel{s:sack}

The purpose of this section is to prove a parametric version of \thmref{t:unique} for confoliations (which have holonomy if they are foliations) such that either $\FF$ is a minimal foliation or every minimal set of $\FF$ is exceptional. 

\begin{thm} \mlabel{t:sack contract}
Let $\xi$ be a $C^2$-confoliation which has no closed leaf but is not a foliation without holonomy. Then there is a $C^0$-neighbourhood $U$ of $\xi$ such that the space of positive contact structures in $U$ is weakly contractible.
\end{thm}

\subsection{Facts about exceptional minimal sets} \mlabel{s:sack intro}

We first review the relevant definitions and results concerning Sacksteder curves.

\begin{thm}[Sacksteder, \cite{sack}] \mlabel{t:sack}
Let $\FF$ be a $C^2$-foliation and $N\subset M$ an exceptional minimal set. Then there is a leaf $L\subset N$ containing an embedded closed curve $\gamma$ such that the holonomy $h_\gamma : \tau\lra \tau$ along $\gamma$ satisfies 
$$
h'_{\gamma}(x)<1.
$$  
where $\tau$ is an interval transverse to $\FF$ containing $x\in\gamma$. 
\end{thm}
A curve with the properties of $\gamma$ in this theorem will be referred to as {\em Sacksteder curve}. Going through the proof of \thmref{t:sack} one can easily verify that it is also valid for confoliations. 
 
Sacksteder's theorem is one of the instances where the $C^2$-hypothesis is used in an essential way. As it turns out the other place where $C^2$-smoothness is used, namely the theory of foliations without holonomy and the following observation about minimal foliations with holonomy, are also based on  \thmref{t:sack}. The proof of the following result from \cite{confol} can be found in \cite{pet}. 

\begin{thm}[Ghys] \mlabel{t:minimal sack}
Let $\FF$ be a $C^2$-foliation on $M$ such that $M$ is a minimal set and $\FF$ is not a foliation without holonomy. Then there is a Sacksteder curve tangent to $\FF$. 
\end{thm}

From the fact that Sacksteder curves have non-trivial linear holonomy it follows that a Sacksteder curve cannot bound a compact subsurface of the leaf $L$ it is contained in (this holds for both orientations of $\gamma$). In order to see this  recall that 
\begin{align} \label{e:non-split}
\begin{split}
\pi_1(L) & \lra \R \\
\alpha & \lmt \log(h_\alpha'(0))
\end{split}
\end{align}
determines a cohomology class in $H^1(L,\R)$ (cf. \cite{cc} or \cite{confol}). Again this remains true when $L$ is a leaf of the fully foliated part of a confoliation. 

\subsection{Adapting definitions related to ribbons}
In the proof of \thmref{t:sack contract} we use the set up from the proof of \thmref{t:transitive}.   We describe the required changes in the following. 

Either $\xi$ is a foliation all of whose leaves are dense or all minimal sets of the fully foliated part of $\xi$ are exceptional. Fix a foliation $\II$ of rank $1$ transverse to $\xi$. Because $M$ is compact there are only finitely many exceptional minimal sets $N_1,\ldots, N_\kappa$ (this follows immediately from \thmref{t:sack}, cf. \cite{cc}). If $\xi$ is minimal we pick one Sacksteder curve $\gamma_1$, otherwise let $\gamma_1,\ldots,\gamma_\kappa$ be a collection of Sacksteder curves such that $\gamma_j\subset N_j$ for $j=1,\ldots,\kappa$. The curves $\gamma_j$ are contained in leaves $L_j$ of the fully foliated part of $\xi$. 

We choose a pairwise disjoint collection of tubular neighbourhoods of the Sacksteder curves $\gamma_1,\ldots,\gamma_\kappa$, each of these tubular neighbourhoods $N_j$ is diffeomorphic to $\gamma_j\times [-1,1]\times[-1,1]$ and the fibers of the projection 
$$
\pi_j : \gamma_j\times [-1,1]\times[-1,1] \lra \gamma_j\times [-1,1]\times\{0\}\subset L_j
$$
along the third factor are tangent to $\II$.  We fix smaller tubular neighbourhoods of $\gamma_j$ 
$$
\gamma_j\subset N_j\subset\widehat{N}_j\subset  \left(\gamma_j\times [-1,1]\times[-1,1]\right)
$$
and a polyhedral decomposition of $M$ adapted to $\xi$ and $\II$ such that the following conditions are satisfied.
\begin{itemize} 
\item[(i)] The characteristic foliation $\xi$ on $\partial N_j,\partial\widehat{N}_j$ has exactly two Reeb components.
\item[(ii)] A polyhedron which meets $N_j$ does not meet $\partial\widehat{N}_j$. 
\end{itemize}
Let 
\begin{align*}
N & :=\bigcup_{j=1}^\kappa N_j  &\widehat{N} & :=\bigcup_{j=1}^\kappa \widehat{N}_j \\
\end{align*}
We will use the following modified definition of the notion of a full collection of ribbons from the context of transitive confoliations to the present situation. 
\begin{defn} \mlabel{d:comp slice2}
A collection of ribbons $\sigma_i\times[0,1]$ is {\em full} if it satisfies the requirements of \defref{d:comp slice} with the following modifications: 
\begin{itemize}
\item No ribbon begins at a face contained in $N$. Moreover, no ribbon which enters $N$ leaves $N$ again. 
\item For all $i$, the segment $\sigma_i\times\{1\}$ is either contained in $H(\xi)$ or this segment has the following properties:
\begin{itemize}
\item $\sigma_i\times\{1\}$ is contained in the interior of $N_j$ and at the same time in a fiber of $\pi_j$ for some $j\in\{1,\ldots,\kappa\}$ so that $\sigma_i\times[0,1]$ is tangent to $\gamma_j\times\{*\}\times[-1,1]$.
\item The union of semi infinite segments of the characteristic foliation $\gamma_i\times\{*\}\times[-1,1]$ which point away from $\sigma_i\times[0,1]$ is an immersion of $\sigma_i\times[1,\infty)$ which accumulates on $\gamma_i\times\{*\}\times\{0\}$ and none of the semi infinite segments intersects $\sigma_i\times\{1\}$ twice.
\item The ribbons are tangent to $\II$. 
\item Different ribbons end in different annuli of the form $\gamma_j\times\{*\}\times[-1,1]$.
\end{itemize}
\end{itemize}
\end{defn}

\begin{lem} \mlabel{l:comp slice2}
Every foliation with holonomy and without compact leaves admits a complete collection of ribbons.
\end{lem}
\begin{proof}
The proof of \lemref{l:comp slice} carries over almost immediately. 
Since curves with non-trivial holonomy cannot separate the leaf they are contained in (see \eqref{e:non-split} at the end of \secref{s:sack intro}) there is a path tangent to $\FF$ from every point in $(\widehat{N}\setminus N)\cap L_j, j=1,\ldots,\kappa,$ to $N$ without intersecting one of the annuli $\gamma_j\times I$. 
\end{proof}

\subsection{Determining the neighbourhood of $\FF$}  \mlabel{ss:exceptional neighbourhood} 

The characteristic foliations on the annuli $\gamma_j\times\{*\}\times[-1,1]$  described in \defref{d:comp slice2} are stable with respect to $C^0$-small perturbations of $\xi$. The characteristic foliation on these annuli may have more than one closed orbit after a small perturbation of $\xi$. However, what matters to us is that all leaves of the characteristic foliation which enter the annulus  stay in the annulus even after a $C^0$-small perturbation and this is  true for sufficiently small perturbations.

The ribbons which end in $H(\xi)$ are treated as in the proof of \secref{t:transitive}. In particular, we select cylinders $C_i$ as before and holonomy maps $g_i : \sigma_i=\sigma_i\simeq I_{\sigma_i}\times\{0\} \lra\sigma_i$ which will be geometrically realized as holonomies of the characteristic foliation induced by contact structures close to $\xi$ on parts of the cylinders $C_i$. When $\xi$ is not contact near $\sigma_i\times\{1\}$ for some particular $i$, then the fact that there are integral surfaces in small cylinders near $\sigma_i\times\{1\}$ implies that we cannot realize the holonomy $g_i$ with $g_i<\id$ at given prescribed points in the interior of $\sigma_i$.  However, it is possible to realize a given decreasing holonomy with compact support in $\sigma_i=\sigma_i\times\{0\}$ if $\xi'$ is a contact structure which is sufficiently close to $\xi$ by attaching a suitable domain to the modified ribbon.  

Here sufficiently close means the following: 
Let $Y$ be a vector field on $\gamma_j\times[-1,1]$ tangent to the first factor. We denote the $\xi$-horizontal lift of $Y$ by the same letter. Finally, we also pick a compact surface $\partial_-C_i$ transverse to $Y$ which contains $\sigma_i\times\{1\}$ in its interior and which is tangent to the fibers of $\pi_j$ such that the intersection with each fiber is a connected non-degenerate interval. The notation $\partial_-C$ is meant to indicate that this surface will play the same role as  the part of the boundary of $C$ in \lemref{l:domain} which was also denoted by $\partial_- C_i$ there. We require that $\partial_-C_i$ is disjoint from its images under the flow of $Y$ on $\gamma_j\times[-1,1]\times[-1,1]$ for positive times. 

If $\zeta$ is $C^0$-close enough to $\xi$, then there is a $\zeta$-horizontal lift $Y'$ of $Y$ on $\gamma_j\times[-1,1]\times[-1,1]$ along $\II$ such that the flow lines of $Y'$ starting at points inside of $\gamma_j\times[-1,1]\times[-1,1]$ are well defined for positive times and they accumulate on closed leaves of the characteristic foliation of $\zeta$ on $\gamma_j\times[-1,1]\times[-1,1]$ as $t\to\infty$. As before, we also modify the ribbon $\sigma_i\times[0,1]$. The result of this deformation is a ribbon with the same initial conditions as $\sigma_i\times[0,1]$ such that the end opposite to the polyhedron is still contained in $\partial_-C_i$. Let $\psi_\tau'$ be the semi flow of $Y'$ in $\gamma_j\times[-1,1]\times[-1,1]$. This flow is well defined  for $\tau\in[0,\infty)$ and all flow lines in the forward direction accumulate on a closed leaf of the $\zeta$-characteristic foliation on  $\gamma_j\times\{*\}\times[-1,1]$. 

There is one final requirement which will be used to show that certain contact structures on $\widehat{N}$ are tight. For each $\gamma_j\times[-1,1]\times[-1,1]$  we fix a family of discs tangent to the third factor (ie. tangent to $\II$) with two corners so that the union of these discs is yet another tubular neighbourhood of $\gamma_j$. Since all the discs have two  corners it makes sense to require that $\xi$ and all plane field $\zeta$ are transverse to all the discs with corners. The discs are chosen such that the characteristic foliation on the discs (this is the horizontal lift of the foliation on $\gamma_j\times[-1,1]$ given by the second factor) provides an identification of the smooth pieces of the boundary of the discs for all plane fields which are close enough to $\xi$.

The above constructions are possibile for plane fields which are $\eps_0$-close to $\xi$. Now pick $0<\eps<\eps_0$ so small that $\eps$ satisfies all the requirements from \secref{ss:find eps} for those ribbons which end in $H(\xi)\setminus N$.

Before we give the proof of the uniqueness theorem, we describe how to use the non-integrability of contact structures to modify monodromies of characteristic foliations. 

We require that $\zeta$ is so close to $\xi$ that the construction above can be carried out for all ribbons. Since we will ignore all ribbons which are attached to polyhedra which meet $N$, there are only finitely many ribbons. Hence for a sufficiently small number  $\eps_0>0$ the above modification of the ribbons is possible for every contact structure $\eps_0$-$C^0$-close to $\xi$. For $w>0$ let 
$$
C'_i(w) : =\bigcup_{\tau=0}^v\psi'_\tau(\partial_-C_i). 
$$
For $w>0$ this is homeomorphic to a ball. In analogy with the previous notation we set $\partial_+C'_i(w):=\psi_w'(\partial_-C_i)$. 

If $\xi'=\zeta$ is a positive contact structure, then the non-integrability of $\xi'$ can be used as follows: As $w\to\infty$, the pull back of the characteristic foliation on $\partial_+ C'_{i}(w)$ to $\partial_-C'_i:=\partial_- C_i$ converges in a monotone fashion (because $\xi'$ is a contact structure) - and therefore uniformly by Dini's theorem -  to the line field determined by the intersection of the annuli $\gamma_j\times\{*\}\times[-1,1]$ with $\partial_-C_i$ by \eqref{e:evol-alpha}. This follows from the computation discussed at the end of \secref{s:movies}.

By \lemref{l:domain} we can prescribe the holonomy $g_i: \sigma_i=\sigma_i\times\{0\}\lra\sigma_i\times\{0\}$ corresponding to the ribbon and an attached domain when $v$ is sufficiently large. The extended ribbons do not interfere with each other.

Again this works parametrically, i.e. the domain varies smoothly when the contact structures $\xi_s'$ depend continuously (again with the $C^0$-topology) on a parameter $s$ as long as $\xi_s'$ is sufficiently close to $\xi$. By the last requirement in \defref{d:comp slice2} we can ensure that all these domains are pairwise disjoint for all $w$ by choosing $\partial_-C_i$ sufficiently thin. 

We fix holonomy maps $g_i : \sigma_i\lra \sigma_i$ in the same way as in \secref{s:transitive}. As before we assume that $g_i$ has compact support in the interior of the ribbon except near ends which meet supporting vertices of the polyhedral decomposition. Again we require that every closed orbit of the characteristic foliation of the polyhedra outside of $N$ and every supporting vertex outside of $N$ intersects the support of one of the monodromies $g_i$.

\subsection{The proof of the stability theorem for confoliations with holonomy and without closed leaves} \mlabel{s:sack proof}

We shall construct a family of contact structures $\xi_{(s,t)}$ using ribbons and smoothings of polyhedra etc. Also, the domains we attach to ends of ribbons depend on $s$ and $t$. None of these dependencies will be reflected in the notation. 

\begin{proof}
Let $\xi_s$ be a family of contact structures with compact parameter space $S$ so that $\xi_s$ lies in the $\eps$-neighbourhood of $\xi$.  As before, we want to extend $\xi_s$ to a family  of contact structures $\xi_{\widehat{s}}$ with parameter space $\widehat{S}$. This will be done in three steps, so it is convenient to put $\widehat{S}:=S\times[0,3]/S\times\{3\}$. We denote elements of $\widehat{S}$ by $\widehat{s}=(s,t)$. 

For $t\in[0,2]$ we deal with the polyhedra which do not meet $N$. We want to obtain a family of contact structures $\xi_{(s,t)}$ so that the restriction to the complement of $N$ is independent from $s$ when $t=2$. In order to apply \thmref{t:solid torus} we also want to control the characteristic foliation on $\partial\widehat{N}$ when $t=2$. 

We fix a particular perturbation of $\xi$ into a contact structure: By the proof of the approximation theorem of Eliashberg and Thurston (\thmref{t:elth}) there is a contact structure - which we shall denote by $\widetilde{\xi}$ - in the $\eps$-neighbour\-hood of $\xi$ so that the characteristic foliation on each connected component of $\partial\widehat{N}$ has exactly two closed orbits on each connected component (one of them attractive, the other one repulsive). 

In the first two steps we construct a family of contact structures $\xi_{(s,t)}$ with $t\in[0,2]$ so that outside of $\widehat{N}$ we have $\xi_{(s,2)}=\widetilde{\xi}$. This is possible at the expense of losing some control over the contact structure $\xi_{(s,2)}$ inside $\widehat{N}$. However, we will show that $\xi_{(s,2)}$  is tight on $\widehat{N}$. Then we use \thmref{t:solid torus} to find a homotopy $\xi_{(s,t)}, t\in[2,3]$ so that in the end $\xi_{(s,3)}=\widetilde{\xi}$. This will conclude the proof of \thmref{t:sack contract}.

Now we turn to the first two steps. For $t\in[0,1]$ we construct a family of plane fields $\zeta_{(s,t)}, t\in[0,1]$ which coincides on a neighbourhood of supporting vertices in the complement of $N$ with $\widetilde{\xi}$ when $t=1$. In particular, it is then independent of $s$ near supporting vertices. This is done in the same way as in \secref{s:transitive} and as before  we do not create closed leaves on boundaries of polyhedra. 

To obtain a homotopy of contact structures  we fill polyhedra outside of $N$ using \lemref{l:hor fill ball}. 
For  the interior of $N$ we use the restriction of $\xi_s$ to the polyhedra pulled back by isotopies guided by ribbons as in \secref{s:transitive proof}.  In particular, we consider only ribbons starting at polyhedra which do not intersect $N$ and the attachment of $\sigma\times[0,1]$ to a polyhedron in the complement of $N$ induces only an a priori bounded number of ribbons $\sigma^j\times[0,1]$. In this way we obtain a family of contact structures $\xi_{(s,t)}$ which coincides with $\widetilde{\xi}$ near supporting vertices outside of $N$ for $t=1$.  

In the second step we move the characteristic foliation on the faces of polyhedra which do not intersect $N$ to the one induced by $\widetilde{\xi}$. In this way we obtain a plane fields $\zeta_{(s,t)}$ such that the characteristic foliation on the boundary of the polyhedra not intersecting $N$  coincides with the one induced by $\widetilde{\xi}$ (it is understood that the ribbons are attached to the polyhedra). Therefore we can apply \lemref{l:hor fill ball} in such a way that the contact structure $\widehat{\xi}_{s,t}$ obtained from the smooth plane fields $\widehat{\xi}_{s,t}$ coincides with $\xi_2$ for $t=2$ on the complement of $\widehat{N}$. The polyhedra which are not contained in $N$ are filled using \lemref{l:hor fill ball} while $N$ is filled with $\xi_s$. Hence for $t=2$ we can assume that the contact structure $\xi_{(s,t)}$ coincides with $\widetilde{\xi}$ outside of $\widehat{N}$.

On the interior of $\widehat{N}$ the contact structures $\xi_{(s,2)}$ depend on $s$ while the characteristic foliation on the boundary is constant. Once the following claim is proved we can use \thmref{t:solid torus} to extend the homotopy of contact structures by a family $\xi_{(s,t)}, t\in[2,3]$, so that $\widetilde{\xi}=\xi_{(s,3)}$ is independent from $s$.

{\bf Claim: $\xi_{(s,2)}$ is tight on $\widehat{N}$ for all $s$.}

We now use the tubular neighbourhoods of $\gamma_j$ which were formed as unions of discs with two corners and which contain $\widehat{N}_j$. By construction $\xi_{(s,2)}$ is transverse to $\II$ and therefore the characteristic foliation on the discs with corners consists by connected arcs which pass from one smooth piece of the boundary of the disc to the other. Thus $\xi_{(s,2)}$ can be extended to a contact structure on $\R^2\times\R$ such that the second factor corresponds to $\II$ and the extended contact structure is transverse to the second factor and defines a complete connection on $\R^2\times\R\lra\R^2$. Therefore $\xi_{(s,2)}$ is tight for all $s\in S$.
\end{proof}

\section{Uniqueness with closed leaves: Tori} \mlabel{s:tori}

The proof of \thmref{t:unique2} is easier when the only closed leaves of $\FF$ are stable tori. We therefore give this proof first before proceeding to leaves of higher genus in the next section. The main results discussed in \secref{s:fixing closed tori} apply to all orientable surfaces and will be used later, too.

\subsection{Fixing neighbourhoods of closed leaves} \mlabel{s:fixing closed tori}

The purpose of this section is to introduce part of the data we shall use to determine the neighbourhood $U$ of the confoliation $\xi$ in \thmref{t:unique}. 

In contrast to exceptional minimal sets, whose number is always finite, a foliation can have uncountably many compact leaves. However, according to a fundamental theorem of A.~Haefliger \cite{haefliger} the set of compact leaves of a foliation of codimension $1$ is a closed subset of $M$. This result does not need $C^2$-smoothness assumption. Moreover, if $\FF$ is coorientable, then the union of leaves of a given diffeomorphism type is compact. In our situation this implies that there is an integer $g_{max}$ so that the genus of a given closed leaf of $\FF$ is at most $g_{max}$. In order to give a more precise description of the union of closed leaves of $\FF$ we recall the following definition from \cite{bonfi}.

\begin{defn} \mlabel{d:equclass}
Let $\Sigma_0$ and $\Sigma_1$ be two closed leaves of $\xi$. These leaves are {\em equivalent} if there is an immersion 
$$
\psi : \Sigma\times[0,1] \lra M 
$$
with the following properties:
\begin{itemize}
\item[(i)] The restriction of $\psi$ to $\Sigma\times\{t\}$ is an embedding for all $t\in[0,1]$. 
\item[(ii)] $\psi(\Sigma\times\{0\})=\Sigma_0$ and $\psi(\Sigma\times\{1\})=\Sigma_1$.
\item[(iii)] For all $p\in\Sigma$ the curve $\psi(p,\cdot)$ is transverse to $\FF$.
\end{itemize}
\end{defn}
Clearly, equivalent leaves are diffeomorphic. A diffeomorphism is provided by the holonomy of the image of the foliation by the second factor on $\Sigma\times[0,1]$. \defref{d:equclass} has an obvious generalization to all foliations of codimension $1$. The closed leaves of the foliation on $S^1\times[-1,1]$ shown in \figref{b:equiv} (the $S^1$-factor is horizontal) which lie in the center of the figure are all equivalent while the other two closed leaves are not equivalent to any other closed leaf in the figure.

\begin{figure}[htb]
\begin{center}
\includegraphics[scale=0.9]{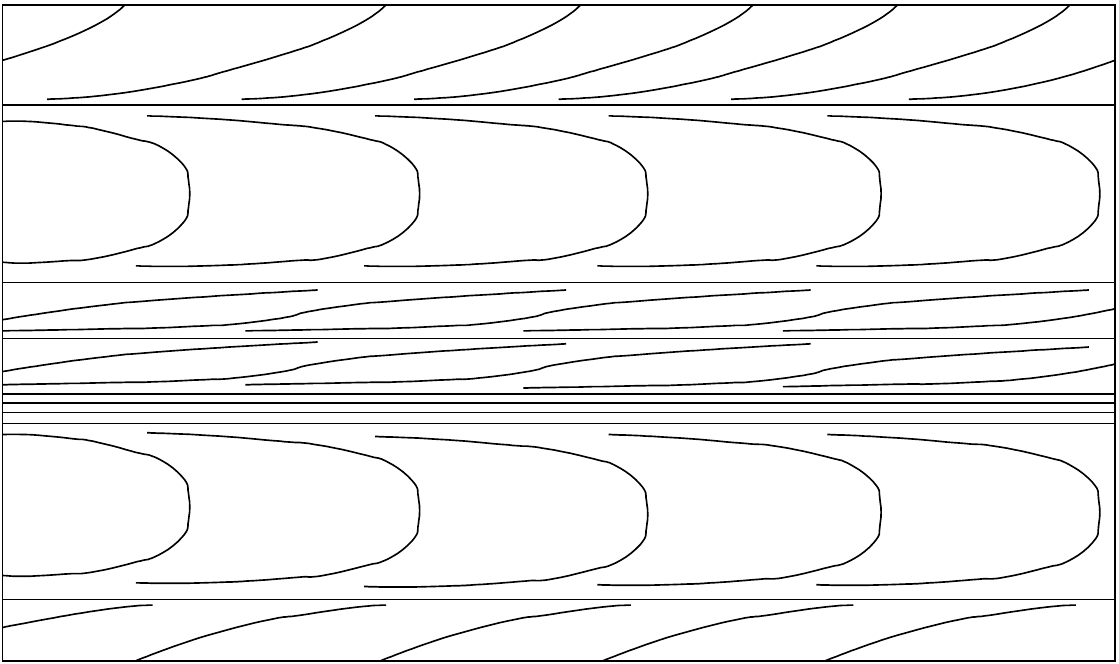}
\end{center}
\caption{Foliation on $S^1\times[-1,1]$ with three equivalence classes of closed leaves}\label{b:equiv}
\end{figure}

Haefliger's compactness theorem implies that there is only a finite number of equivalence classes of closed leaves \cite{bonfi}. Because $\FF$ is not a foliation without holonomy we can actually assume that $\psi$ is an embedding and extend it to a tubular neighbourhood for both $\Sigma_0$ and $\Sigma_1$. Using this terminology, the assumption (i') in \thmref{t:unique2} can be replaced by the slightly weaker requirement
\begin{itemize}
\item[(i'')] The union of all torus leaves is covered by a finite collection of embeddings  $\psi: T^2\times[0,1]\lra M$ as in \defref{d:equclass} each of which has attractive holonomy, i.e. there is a simple non-separating closed curve $\gamma\subset T^2$ so that the holonomy along $\psi(\gamma\times\{i\}) $ in the torus $T^2\times\{i\}$ is attractive on the side not contained in $\psi(T^2\times[0,1])$ for $i\in\{0,1\}$. 
\end{itemize}

The two upper equivalence classes of closed leaves in \figref{b:equiv} have attractive holonomy, the closed leaf at the bottom  does not.

Before we proceed with the proof of \thmref{t:unique2} let us recall the following result of M.~Hirsch \cite{hirsch}. This result is reproved as Theorem 3.f.1 in \cite{bonfi}.
\begin{thm} \mlabel{t:hirsch}
Let $\FF$ be a transversely coorientable foliation of codimension $1$ on $M$ and $L$ a closed leaf with Abelian fundamental group so that $\FF$ is not a foliation by fibers of a fibration $M\lra S^1$. Fix a tubular neighbourhood $N(L)$ of $L$. 
\begin{itemize}
\item[(i)] If $L$ has attractive holonomy along some simple closed curve $\gamma\subset L$, then there is a $C^0$-neighbourhood of $\FF$ in the space of plane fields such that every foliation in that neighbourhood has a closed leaf diffeomorphic to $L$ inside $N(L)$. 
\item[(ii)] If there is no curve in $L$ with attractive holonomy, then every $C^1$-neighbourhood of $\FF$ contains a foliation which has no closed leaf inside $N(L)$. 
\end{itemize}
\end{thm}
This theorem explains the terminology {\em stable/unstable} torus leaf and \thmref{t:unique2} is true only for stable torus leaves.  The case when $\FF$ is given by the fibers of a torus  fibration over $S^1$ is also understood: According to a theorem of J.~Plante \cite{pl} a fibration with torus fibers has a $C^0$-neighbourhood such that every foliation in that neighbourhood has a torus leaf close to an original fiber if and only if the map 
$$
\phi_{M} : H_1(T^2,\Z) \lra H_1(T^2,\Z)
$$
induced by the monodromy of the torus bundle does not have a positive real eigenvalue. However, it will turn out in \secref{s:torus bundles} that not only these fibrations have a neighbourhood with the properties described in \thmref{t:unique2}. Therefore the analogy between stable torus leaves and stably isotopic approximating contact structures is not perfect. 
\subsection{Determining the neighbourhood in the space of plane fields} \mlabel{s:determine U for T}

For the definition of the neighbourhood $U$ in \thmref{t:unique2} of the confoliation $\xi$ in the space of plane fields we proceed as in \secref{s:sack}. In order to simplify the presentation we assume that $\xi$ has a unique minimal set which is a closed torus leaf with attractive holonomy. How to treat the case when there are several minimal sets, either exceptional ones or other torus leaves, will then be clear.  Since we have already considered one situation where both $H(\xi)$ and the fully foliated set of $\xi$ are not empty we will assume from now on that $\xi=\FF$ is a foliation. As usual $\II$ is a line field transverse to $\xi$. 

Let $T\subset M$ be the unique torus leaf and $\gamma$ a simple closed curve in $T$ with attractive holonomy. 

There is a foliation $\GG$ on $T$ by simple closed curves such that $\gamma$ is a leaf and the holonomy along each leaf of $\GG$ is attractive (in the case of confoliations this is true by  \lemref{l:confol holonomy}).  
Now we fix a polyhedral decomposition of $M$ adapted to $\FF$ and $\II$ together with a pair of tubular neighbourhoods
\begin{equation} \label{e:N(T)}
\widehat{N}(T)\simeq T\times[-1,1]\supset N(T) \simeq T\times[-1/2,1/2]
\end{equation}
of $T$ such that the following conditions are satisfied:
\begin{itemize}
\item[(i)] The foliation $\II$ is tangent to the second factor in \eqref{e:N(T)} and the characteristic foliation on $\gamma\times[-1,1]$ has one closed orbit and all other leaves enter this annulus through its boundary and accumulate on $\gamma$ such that the characteristic foliation is transverse to $\partial N(T)\cap (\gamma\times[-1,1])$. The same requirement is supposed to hold for  the other leaves of $\GG$. 
\item[(ii)] $\partial N(T)$ and $\partial \widehat{N}(T)$ are both transverse to all faces and edges of the polyhedral decomposition. 
\item[(iii)] The characteristic foliations $\partial \widehat{N}(T)(\FF)$ and $\partial N(T)(\FF)$ contain no $2$-dimensional Reeb component. 
\item[(iv)] No polyhedron which intersects $\partial\widehat{N}(T)$ meets $\partial N(T)$.
\end{itemize}
We fix a full collection of ribbons for all faces of polyhedra which do not meet $N(T)$, i.e. we choose a finite collection of ribbons $\sigma_j\times[0,1]$ satisfying the following conditions:
\begin{itemize}
\item[(i)] The end of each ribbon which is not contained in the face of a polyhedron is contained in $N(T)$ and the ribbons are pairwise disjoint. Each closed leaf of the characteristic foliation of a polyhedron outside of $N(T)$ meets the attaching arc $\sigma_j$ of a ribbon. 
\item[(ii)] For each ribbon the projection of the part of the ribbon lying in $N(T)$  along $\II$ to the torus is contained in a single leaf of $\GG$.
\item[(iii)] The characteristic foliation on the annuli tangent to $\II$ and containing leaves of $\GG$ extends the ribbons to semi-infinite ribbons accumulating on a leaf of $\GG$. 
\end{itemize}
In the following we sometimes view $\GG$ as a foliation on $\partial\widehat{N}(T)$. Then we denote this foliation by $\widehat{\GG}$. 

We are now in a position to choose $\eps>0$ which determines the $C^0$-neighbourhood of $\FF$ in \thmref{t:unique2} in the present context (i.e. no closed leaves of higher genus). 
\begin{itemize}
\item $\widehat{\GG}$ remains transverse to $\zeta$ for every plane field $\zeta$ which is $\eps$-$C^0$-close to $\FF$.
\item $\zeta$ is transverse to $\II$.
\item The characteristic foliation of $\zeta$ on $\gamma'\times[-1,1]$ remains transverse to the boundary and inward pointing for all leaves $\gamma'$ of $\GG$. In particular, all leaves of $\gamma'\times[-1,1]$ entering through the boundary accumulate on a closed leaf. It is irrelevant to our discussion how many closed leaves this  characteristic foliation has or whether or not they are non-degenerate.
\item All ribbons can be lifted to ribbons adapted  to $\zeta$ while the ribbons still have the necessary properties (pairwise disjointness, ending in $N(T)$, etc.) explained in  \secref{ss:find eps} and \secref{ss:exceptional neighbourhood}. 
\end{itemize}

\subsection{The proof of \thmref{t:unique2} in the absence of closed leaves of higher genus} \mlabel{s:unique proof only tori}

It remains to show that with $\eps>0$ from the previous section the $\eps$-$C^0$-neighbourhood of $\FF$ has the desired properties. 

\begin{proof}[Proof of \thmref{t:unique2}]
We start with two contact structures  $\xi,\xi'$ which are $\eps$-close to $\FF$. By the procedure from \secref{s:transitive proof} and \secref{s:sack proof} we can connect $\xi$ to a contact structure $\widehat{\xi}$ which coincides with $\xi'$ on all polyhedra which do not meet $N(T)$. 

Using \lemref{l:hor fill ball} we can therefore ensure that the contact structure remains transverse to $\II$. After a $C^\infty$-small perturbation we may assume that $\partial\widehat{N}(T)$ is convex and non-singular. Then the dividing set on $\partial N(T)$ contains no homotopically trivial component.  (By \lemref{l:tight extension} $\widehat{\xi}$ and $\widehat{\xi}'$ are tight on $N(T)$.) 

Because the characteristic foliation of $\widehat{\xi}$ on $\partial\widehat{N}(T)$ is the same as the one induced by $\xi'$, there is no Reeb component in $\partial\widehat{N}(T)$. Now we apply \thmref{t:tight on TxI}. For this we have to check that for both contact structures $\widehat{\xi},\xi'$ there is a torus in the interior of $\widehat{N}(T)$ isotopic to $T$ and whose characteristic foliation is a foliation by closed leaves.

But this follows from the fact that for each leaf $\gamma'$ of $\GG$ the characteristic foliation on $\gamma'\times[-1,1]$ has a closed leaf in the interior. By the results from \secref{s:movies} the union of these closed Legendrian curves is an embedded torus with the desired properties. Thus $\widehat{\xi}$ and $\widehat{\xi}'$ are stably isotopic by \thmref{t:tight on TxI}.
\end{proof}

If $T$ is a torus leaf (stable or unstable) of a (con-)foliation $\FF$, then $\FF$ can be $C^0$-approximated by confoliations $\FF_n$ containing a domain foliated by tori. It is then easy to approximate $\FF_n$ by a contact structure with arbitrarily large Giroux torsion along $T$. 

Finally, let us mention two points where the above proof fails for unstable torus leaves. When $T$ is unstable, then
\begin{enumerate}
\item we are no longer sure that our ribbons can be extended to semi-infinite ribbons in $\widehat{N}(T)$ and
\item we can no longer guarantee that (iii) of \thmref{t:tight on TxI} is satisfied. 
\end{enumerate}
As indicated in the bottom part of \figref{b:un-stable} (on p.~\pageref{b:un-stable}) there may be sheets of the contact structure which connect the two boundary components of $\widehat{N}(T)$. If this happens then according to \cite{gi-bif}  there are infinitely many contact structures on $\widehat{N}(T)$ with vanishing Giroux torsion which are pairwise non-isotopic and still satisfy assumptions (i),(ii) of \thmref{t:tight on TxI}. For these contact structures, the sheets connect the two boundary components of $\partial\widehat{N}(T)$. The conditions on $\eps$ formulated in \secref{s:higher genus neighbourhood} ensure - among other things - that no sheet  of the contact structure $\eps$-$C^0$-close to $\FF$ will connect the boundary components of  a tubular neighbourhood of the closed leaf.

In \exref{ex:unstable torus} we show that in this case it may happen that any neighbourhood (it will turn out that we may even take a $C^\infty$-neighbourhood) of a foliation with unstable torus contains two positive contact structures which are not stably isotopic on $M$.

\section{Uniqueness in the presence of closed leaves with higher genus} \mlabel{s:higher}

In this section we shall complete the proof of \thmref{t:unique} for foliations with holonomy.

The previous sections have covered the situation when $\FF$ is a foliation (or a positive confoliation) which belong to one of the following classes:
\begin{itemize}
\item $\FF$ is a foliation such that every leaf is dense and there is holonomy.
\item $\FF$ is a (positive con-) foliation all of whose minimal sets are either exceptional or stable torus leaves.
\end{itemize}
In this section we deal with closed leaves of genus $g\ge 2$. As in the case of torus leaves the set of closed leaves of a fixed genus is not finite but at least it is compact. However, the discussion from the beginning of \secref{s:fixing closed tori} applies. In order to simplify the presentation, we assume that there is exactly one minimal set which is the closed leaf $\Sigma$. 

Of course, \thmref{t:hirsch} does not cover the case of higher genus surfaces. Actually it was shown by T.~Tsuboi \cite{tsuboi} that every $C^1$-neighbourhood of a foliation contains a foliation without closed leaves of higher genus.

\subsection{Geometry of surfaces of higher genus} 

Let $\Sigma$ be a closed surface of genus $\ge 2$. We fix a hyperbolic metric on $\Sigma$ and a universal covering $\HH^2\lra\Sigma$. On all surfaces covering $\Sigma$ we use the pulled back metric and $\partial \HH^2$ denotes the ideal boundary of the hyperbolic plane.
\begin{lem}\mlabel{l:Sigma ator}
There is a constant $K$ which depends only on the hyperbolic surface $\Sigma$ with the following property.

Let $\gamma_t,t=[0,1]$, be a family of homotopically essential simple closed curves and let $\widetilde{\gamma}_t$ be a lift of the isotopy to $\HH^2$. Then there is a pair of points $p_0\in\widetilde{\gamma}_0$ and $p_1\in\widetilde{\gamma}_1$ such that the distance between the points is smaller than $K$.  

If $\Sigma'\lra\Sigma$ is an Abelian covering, then the same constant can be used for $\Sigma'$. 
\end{lem}

The only interesting case is when $\widetilde{\gamma}_1$ lies entirely on one side of $\widetilde{\gamma}_0$ in $\HH^2$.  This will be the case in our applications of this lemma. The statement about Abelian coverings will be relevant only in the proof of \thmref{t:unique} when $\FF$ is a foliation without holonomy and all leaves are dense.

\begin{proof}[Proof of \lemref{l:Sigma ator}]
We will use some facts from the geometry of hyperbolic surfaces which can be found for example in the first chapters of \cite{fm}. 

Because $\gamma_0$ is a simple closed curve it is isotopic to a unique closed geodesic $\gamma$ which is also simple and non-trivial because $\gamma_0$ is not null-homotopic. We fix a lift $\widetilde{\gamma}$ of $\gamma$ in the universal covering. In order to prove the lemma we will show that there is a constant $K'$ and a point $\widetilde{x}\in\widetilde{\gamma}$ such that the lift of every curve isotopic to $\gamma$ with the same endpoints on the ideal boundary contains a point whose distance from $\widetilde{x}$ is at most $K'$. Then there is pair of points on $\widetilde{\gamma}_0$ and $\widetilde{\gamma}_1$ such that the distance between the two points is smaller than $K=2K'+\max\{\mathrm{length}(a_i)\}$. Here $a_1,\ldots,a_l\subset\Sigma$ is a collection of null-homologous homotopically essential simple closed curves such that the complement of the curves is a union of discs (as indicated in \figref{b:cutting}).

\begin{figure}[htb]
\begin{center}
\includegraphics[scale=0.75]{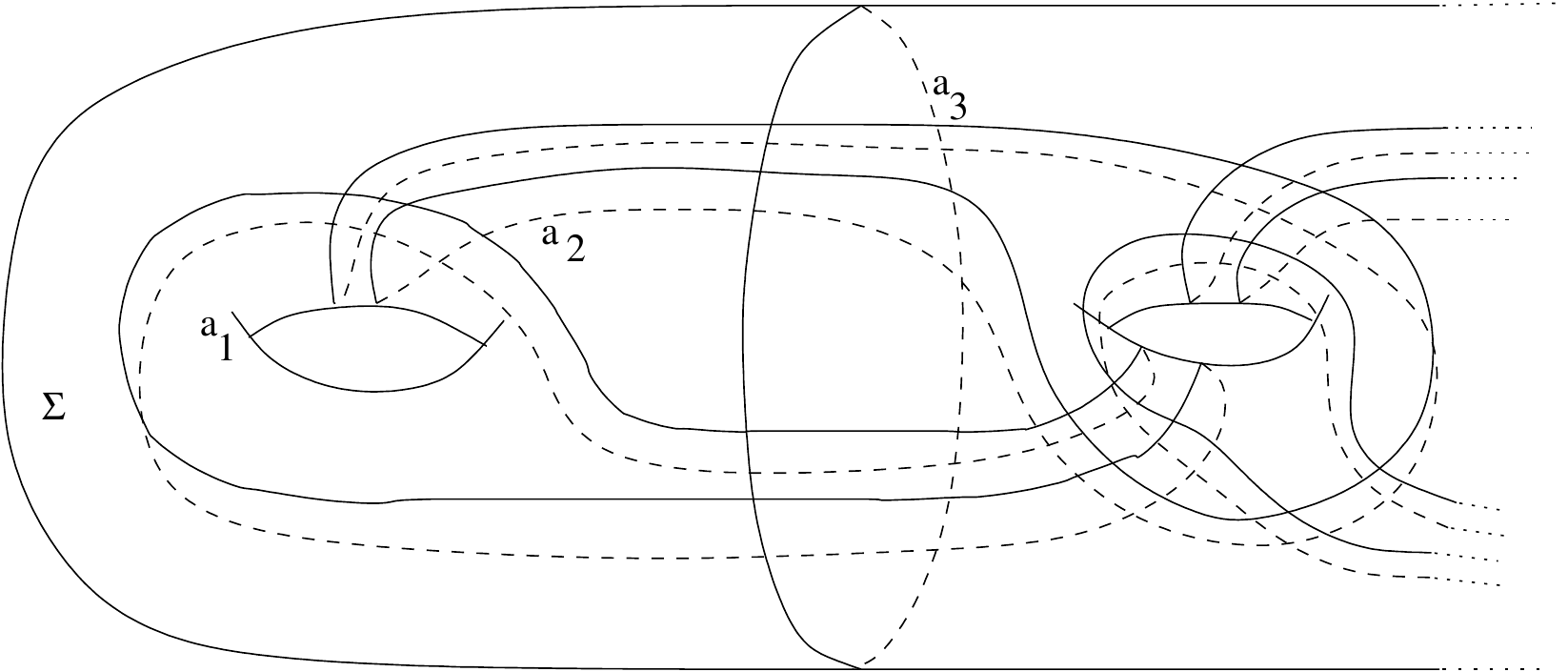}
\end{center}
\caption{Collection of null-homologous curves in $\Sigma$ }\label{b:cutting}
\end{figure}
We assume that the curves $a_i$ are geodesics. Let $B_R$ be a large disc in $\HH^2$ such that the restriction of the universal covering map to the disc is surjective. 

For each lift $\widetilde{a}_i$ of $a_i$ intersecting $B_R$ we pick connected neighbourhoods $\widetilde{J}_i^{\pm}\subset\partial\HH^2$ of the endpoints of $\widetilde{a}_i$ such that these neighbourhoods are pairwise disjoint. (There are only finitely many lifts of $a_i$ which intersect $B_R$.)  Given a lift $\widetilde{\alpha}$ of an oriented closed geodesic $\alpha$ in $\Sigma$ we denote the corresponding isometry of $\HH^2$ by $f_{\widetilde{\alpha}}$. 

For each $\eps>0$ there is a number $N=N(\eps)$ with the following property: If $\widetilde{\alpha}$ is a geodesic in $\HH^2$ which intersects $\widetilde{a}_i$ in a point $y\in B_R$ such that 
$$
\left|\measuredangle_y(\widetilde{\alpha},\widetilde{a}_i)\right|>\eps,
$$
then both endpoints of $f_{\widetilde{a}_i}^{\pm N}(\widetilde{\alpha})$ are contained in the same interval $\widetilde{J}_i^{\pm}$ (the thickened arcs in \figref{b:H22} correspond to two such intervals). In the following we choose $\eps>0$ close to zero (and the corresponding integer $N$) such that if a geodesic $\gamma$ intersects one curve $a_j$ and the angle at that intersection is smaller than $\eps$, then the absolute value of the angle at the intersection points of $\gamma$ and $\bigcup_{i\neq j} a_i$ which lie next to $x$ on $\gamma$ is bigger than 
$$
\frac{1}{2}\min\left\{\left|\measuredangle_y(a_i,a_j)\right|\,|\, y\in a_i\cap a_j \textrm{ and } i\neq j\right\}.
$$  

Pick two intersection points $x,x'$ of $\gamma$ with $a_1,\ldots, a_l$ which are either consecutive along $\gamma$ and the angle of the two geodesics at both intersection points is superior to $\eps$ or $x,x'$ are separated by exactly one other intersection of $\gamma$ with $\bigcup_ia_i$ where the angle between the two curves is smaller than $\eps$. 

Now consider preimages $\widetilde{x},\widetilde{x}'\in\widetilde{\gamma}$ of $x,x'$ such that there is either no or exactly one intersection point of the segment of $\widetilde{\gamma}$ between $\widetilde{x},\widetilde{x}'$ and lifts of the curves $a_i$. The distance between $\widetilde{x}$ and $\widetilde{x}'$ is bounded by $2K_1$ where $K_1$ is the maximal diameter  of the discs $\Sigma\setminus\bigcup_i a_i$. This bound is independent from $\gamma$.

Let $\widetilde{a}, \widetilde{a}'$ be lifts of two curves from our collection passing through $\widetilde{x},\widetilde{x}'$. Since our collection of curves is finite, we can consider 
$$
K' : = \max\left.\left\{\mathrm{dist}\left(y,f_{\widetilde{a}_i}^N(y)\right)\,\right|\,\mathrm{dist}(y,\widetilde{a}_i)\le 2K_1 \textrm{ and } \widetilde{a}_i\textrm{ is a lift of } a_i\right\}.
$$
This number depends only on the maximal displacement of the isometries associated to our system of curves but not on $\gamma$. The point $\widetilde{x}$ and the number $K'$ have the desired properties. 

In order to see this assume that $\widetilde{\gamma}'$ is the lift of a simple closed curve isotopic to $\gamma$ with the same endpoints in the ideal boundary of $\HH^2$ which does not meet the $K'$-ball around $\widetilde{x}$. Then, as indicated in \figref{b:H22}, some images of $\widetilde{\gamma}'$ under isometries $f_{\widetilde{a}_i}^{\pm N}$ intersect. But this contradicts the assumption that $\gamma'$ is simple. 
\begin{figure}[htb]
\begin{center}
\includegraphics[scale=0.7]{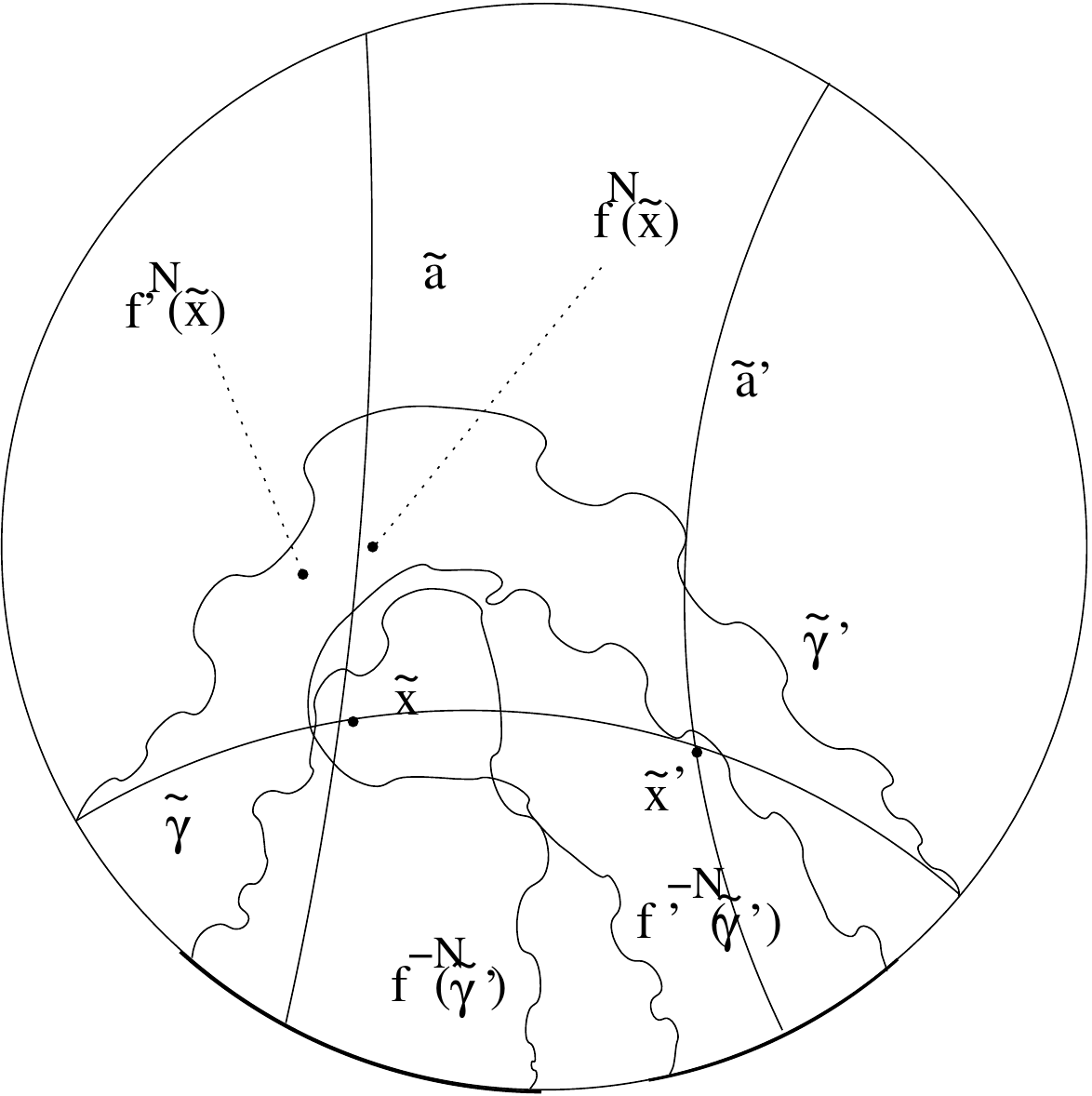}
\end{center}
\caption{Non-disjoint lifts of $\gamma'$ to $\HH^2$} \label{b:H22}
\end{figure}

If $\Sigma'\lra\Sigma$ is an Abelian covering of $\Sigma$, then we lift the collection of separating curves and the hyperbolic metric to the covering. Because the covering is Abelian the lifted curves are still  closed. The discs obtained by cutting $\Sigma'$ along all lifts of $a_1,\ldots,a_l$ are isometric to the discs obtained from $\Sigma$. Since $K$ is determined by this data we can use the same constant on all Abelian coverings of $\Sigma$.
\end{proof}

This lemma is  not true if  $\Sigma$ is a torus $T^2\simeq S^1\times S^1$ since one can use the flow along the first  circle direction to displace $\{1\}\times S^1$ from itself in such a way that the distance becomes unbounded when everything is lifted to the universal covering. The crucial point in the proof of \lemref{l:Sigma ator} is that lifts of isotopies of closed curves do not move the endpoints in the ideal boundary. 

In order to see that \lemref{l:Sigma ator} does not hold for homotopies (instead of isotopies) consider a closed geodesic $\gamma\subset \Sigma$ and a lift $\widetilde{\gamma}\subset\HH^2$. Now let $\widetilde{\gamma}_t$ be a family of curves in $\HH^2$ consisting of points whose distance $t$ from $\widetilde{\gamma}$ is $t$. This family of curves projects to a homotopy $\gamma_t$ of closed curves in $\Sigma$ and this homotopy violates the conclusion of \lemref{l:Sigma ator} since $t$ can be chosen arbitrarily large.

\subsection{Determining the neighbourhood in the space of plane fields} \mlabel{s:higher genus neighbourhood}

Let $\Sigma$ be a closed leaf of $\FF$ with genus $g\ge 2$. As usual we will assume that it is the unique minimal set of $\FF$ and we fix a foliation $\II$ of rank $1$ transverse to $\FF$. For a hyperbolic metric on $\Sigma$ we fix the constant $K$ from \lemref{l:Sigma ator}. 

We fix a tubular neighbourhood $\widehat{N}(\Sigma)=\Sigma\times[-2,2]$ such that the second factor is tangent to $\II$. If $\Sigma$ is not the only closed leaf then we consider a neighbourhood of an equivalence class of closed leaves. We require that all plane fields are transverse to $\II$.

Loosely speaking we ask that $\FF$ -horizontal lifts along geodesics of length $\le K+1$ starting in $\Sigma_{\pm 1}$ are defined $g$ times. More precisely, we suppose that the following procedure is possible.  Starting with $\Sigma_{-1}$ we obtain $2g+1$ levels 
\begin{equation} \label{e:levels}
-2<\delta_{-(g+1)}^-<\delta_{-g}^-<\ldots<\delta_0^{-}=-1<\ldots<\delta_{g+1}^-<0 
\end{equation}
in $(-2,0)$ as follows. We start with $\Sigma_{-1}$ and consider $\FF$-horizontal lifts of all geodesics in $\Sigma$ whose length is smaller than $K+1$ with initial point lying in $\Sigma_{-1}$. The space of such geodesics is compact, hence we can define
\begin{equation} \label{e:level2}
\delta_1^-=\max\left\{ t\in(-1,0)\left|\begin{array}{ll} \textrm{there is a geodesic of length }\le K+1\textrm{ whose} \\ \textrm{lift starting in } \Sigma\times\{\delta_{0}^-\}\textrm { ends in }\Sigma\times\{t\} \end{array} \right\}.\right. 
\end{equation}
Then we define $\delta_i^-$  inductively, the levels lying below $\Sigma_{-1}$ are determined by replacing $\max$ by $\min$ in \eqref{e:level2}. For the part of $\Sigma\times[-2,2]$ lying above $\Sigma_0$ we fix an analogous sequence 
\begin{align*}
0<\delta_{g+1}^+<\delta_{g}^+<\ldots<\delta^+_{0}=1<\ldots<\delta_{-(g+1)}^+<2.
\end{align*}
If $\Sigma_{\pm_1}$ are sufficiently close to $\Sigma_0$, then we can fix levels $\delta_i^\pm$ defined as above. 

This is used in one of our requirements for $\eps$, namely we ask that $\eps>0$ is so small that there is no geodesic in $\Sigma$ of length $\le K$ whose $\zeta$-horizontal lift connects $\Sigma\times\{\delta_i^-\}$ to $\Sigma\times\{\delta_{i+1}^-\}$ or $\Sigma\times\{\delta_{i-1}^-\}$ for $i=-g,\ldots,g$, where $\zeta$ is any smooth plane field whose $C^0$-distance to $\FF$ is smaller than $\eps$. Of course, we ask that $\eps$ satisfies the analogous requirement associated with respect to the levels $\delta^+_i$. Let 
\begin{equation} \label{e:delta}
\delta:=\frac{\min\{|\delta_{g+1}^-|,\delta_{g+1}^+\}}{2}.
\end{equation}
and $N(\Sigma)=\Sigma\times[-\delta,\delta]$. Inside $N(\Sigma)$ we choose yet another smaller neighbourhood $\widehat{N}'(\Sigma)$ and constants $\widetilde{\delta}^\pm_i$ and $\delta'$ which are analogous to the constants $\delta^\pm_i$ and $\delta$  (defined in \eqref{e:levels},\eqref{e:delta}). Finally let $N'(\Sigma)=\Sigma\times[-\delta',\delta']$. Using this data we obtain additional restrictions on $\eps$ analogous to those above.

Since $\Sigma$ is an isolated closed leaf, there are simple closed non-separating curves $\gamma_+,\gamma_-$ embedded in $\Sigma$ such that the holonomy along $\gamma_+$ respectively $\gamma_-$ is attractive on the side lying above respectively below $\Sigma$. Note that $\gamma_+,\gamma_-$ are not isotopic or disjoint in general. At least  we can choose both of  them non-separating because if the holonomy of $\Sigma$ is trivial on one side for all non-separating curves, then all leaves of $\FF$ in a neighbourhood of $\Sigma$ which meet the same side of the neighbourhood are compact and equivalent to $\Sigma$. But we assumed that $\Sigma$ is isolated. Hence there are annuli $A_{\pm}$ containing $\gamma_{\pm}$ in their interior such that 
\begin{itemize}
\item one boundary component (above $\Sigma$ for $A_+$, below $\Sigma$ for $A_-$) is transverse to $\FF$,
\item $A_\pm$ are both contained in the interior  $\Sigma\times[\delta,\delta]$, and
\item $A_\pm$ is tangent to $\II$. 
\end{itemize}
We pick product neighbourhoods of $A_{\pm}$ such that every annulus in that family has the same properties as the original annulus $A_\pm$.

As in \secref{s:determine U for T} there is another condition related to polyhedral decompositions and ribbons $\eps>0$ has to satisfy. We fix the following data on $M$:
\begin{itemize}
\item A polyhedral decomposition of $M$ adapted to $\FF$ and $\II$ such that no 
polyhedron which meets $\Sigma\times[\delta,\delta]$ meets the complement of $\Sigma\times(-2\delta,2\delta)$.
\item No polyhedron which meets the complement of $\Sigma\times(-2,2)$ meets $\Sigma\times\{\delta_{-(g+1)}^\pm\}$. 
\item A complete collection of ribbons $\sigma_i\times[0,1]$ for all polyhedra in the decomposition which meet the complement of $\Sigma\times[-2\delta,2\delta]$. The ribbons are assumed to be tangent to $\II$. 
The ends of the ribbons opposite to the faces lie in $\widehat{N}'(\Sigma)$, they are tangent to annuli parallel to $A_\pm$ and enter annuli near  $A_+$ respectively $A_-$ through the boundary component above respectively below $\Sigma$. 
\end{itemize}
We require that $\eps$ is such that all the properties of polyhedra and ribbons vary continuously  when the plane field varies in the $\eps$-$C^0$-ball around $\FF$ in the space of smooth plane fields. Moreover, the boundary components of the annuli $A_\pm$ which are transverse to $\FF$ are also transverse to $\zeta$ for all plane fields $\eps$-close to $\FF$. 

Comparing this with the torus case $\widehat{N}(\Sigma)=\Sigma\times[-2,2]$ will play the role of $\widehat{N}(T)$ and $N(\Sigma)=\Sigma\times[-\delta,\delta]$ will be the analogue of $N(T)$ (as indicated by the notation). Since we made no assumptions on the holonomy of $\Sigma$ we need one more ingredient since after the perturbation of $\FF$ into a contact structure we can't be sure that the ribbons ending in the annuli $A_\pm$ can still be extended to semi-infinite ribbons. (Thus if $\gamma_+=\gamma_-$ and this curve has attractive holonomy on both sides the neighbourhoods $\widehat{N}'(\Sigma)\supset N'(\Sigma)$ are not needed for the proof of \thmref{t:unique}).

\subsection{The proof of \thmref{t:unique} in the presence of closed leaves of higher genus} \mlabel{s:higher genus proof}

Let $\xi_0,\xi_2$ be two positive contact structures $\eps$-$C^0$-close to $\FF$. 
\begin{enumerate}
\item We isotope $\xi_0,\xi_2$ inside $N'(\Sigma)$ such that after the isotopy the annuli $A_\pm$ (and the annuli parallel to $A_\pm$ fixed above)  contain a closed leaf in their interior.  For this we show that the contact structure on $\widehat{N}'(\Sigma)\setminus N'(\Sigma)$ determines the contact structure on $\widehat{N}'(\Sigma)$ up to isotopy. The resulting contact structures are still transverse to $\II$ and they are still denoted by $\xi_0,\xi_2$. (This step is not needed if $\Sigma$ has attractive holonomy).
\item We construct a homotopy $\xi_s,s\in[0,1],$ of contact structures on $M$ such that $\xi_1=\xi_2$ outside of $\Sigma\times[-2\delta,2\delta]$ and such that the contact structures $\xi_1,\xi_2$ are tight on $\Sigma\times[\delta^-_{-(g+1)},\delta^+_{-(g+1)}]$. 
In this step we use the restrictions on $\eps$ related to polyhedra and ribbons.
\item Using the constraints on $\eps$ coming from the holonomy of $\FF$ near $\Sigma$, we show that the restrictions of $\xi_1$ and $\xi_2$ to $\widehat{N}(\Sigma)$ are isotopic relative to the boundary. For this we show that the contact structures are determined up to isotopy by their restriction to the smaller space $\widehat{N}(\Sigma) \setminus N(\Sigma).$ 
\end{enumerate}
This is slightly more complicated than in the case of torus leaves. If there is a curve in $\Sigma$ such that the holonomy of $\FF$ along that curve is attractive, then one can omit the first step and proceed as in \secref{s:tori}. 

For Step (1) and (3) we first show the following proposition (formulated for the pair $N(\Sigma)\subset \widehat{N}(\Sigma)$). 
\begin{prop} \mlabel{p:outside determines inside}
Let $\xi$ be a contact structure $\eps$-close to $\FF$ on $\Sigma\times([-2,2]\setminus (-\delta,\delta))$.

Up to isotopy there is a unique tight contact structure on $\Sigma\times[-2,2]$ which coincides with $\xi$ on $\Sigma\times([-2,2]\setminus(-\delta,\delta))$. 
\end{prop}
Since the pair $N'(\Sigma)\subset \widehat{N}'(\Sigma)$ has analogous properties, \propref{p:outside determines inside} also holds for this pair:  A tight contact structure on $\widehat{N}'(\Sigma)$ which is $\eps$-close to $\FF$ is determined up to isotopy by its restriction to the collar $\widehat{N}'(\Sigma)\setminus N'(\Sigma)$. 

We postpone the proof of \propref{p:outside determines inside} and explain first how it implies \thmref{t:unique}. In order to finish Step (1) it suffices to extend the contact structure on $\widehat{N}'(\Sigma)\setminus N'(\Sigma)$ in such way that $A_\pm\cap  N'(\Sigma)$ contain closed leaves. This can be done as in \exref{ex:periodic cont}. One can also use Proposition 6.2. of \cite{hkm}.  

Then Step (2) works as in the case of torus leaves or the case of Sacksteder curves. Finally, \propref{p:outside determines inside} finishes the proof of \thmref{t:unique} for confoliations which are not foliations without holonomy.

Let us summarize the main differences between the case of torus leaves and leaves of higher genus before we prove \propref{p:outside determines inside}:
\begin{enumerate}
\item The relative Euler class essentially determines the tight contact structure  on $\Sigma\times [-1,1]$ up to isotopy if $\Sigma$ is not a torus. 
\item If $\Sigma=T$ is a torus, then we cannot change the contact structure on $N'(T)\subset \widehat{N}'(T)$ in order to ensure the existence of closed leaves of the characteristic foliation on annuli transverse to $T$.
\item The last problem occurs for contact structures which have sheets connecting the two boundary components of $\widehat{N}(T)$. If $T$ does not have attractive holonomy we cannot prevent this by reducing $\eps$ since there is no analogue of \lemref{l:Sigma ator} for tori.  
\end{enumerate}

In order to prove \propref{p:outside determines inside} we want to apply the results outlined in \secref{s:Sigma class}. We need to arrange that the dividing sets on $\Sigma_{-1}$ and $\Sigma_{+1}$ have exactly two connected components which are  non-separating. This is done in \secref{s:clean boundary}. In \secref{s:det contact} we determine the relative Euler class and if the relative Euler class vanishes we determine which basic slice embeds into $\Sigma\times[-2,2]$ as stated in \propref{p:vanishing rel Euler}. 

\subsubsection{Correcting the boundary of the neighbourhood of the closed leaf} \mlabel{s:clean boundary}

In this section we explain how to find a domain in $\widehat{N}(\Sigma)$ which contains $N(\Sigma)$ and satisfies the assumptions of  \thmref{t:Sigma class}. 

\begin{lem} \mlabel{l:no passing}
The constant $\eps>0$ and the levels $\delta_i^{\pm}, i=-(g+1),\ldots,g+1,$ have the following property: Let $\xi$ be a contact structure $\eps$-close to $\FF$, $i=-g,\ldots,g,$ and $\beta\subset \Sigma
\{\delta_i^{\pm}\}$  a closed attractive leaf of the characteristic foliation. Then the sheet $A(\beta)$ does not meet $\Sigma\times\{\delta_{i+1}^{\pm}\}$ or $\Sigma\times\{\delta_{i-1}^{\pm}\}$  
\end{lem}

\begin{proof}
We consider the case $\beta\subset\Sigma_{-1}$. As shown in \secref{s:sheets} there is an embedded annulus $A(\beta)$ in $\Sigma\times[-1,1]$ so that $\beta\subset\partial A(\beta)$ and  $A(\beta)$ is foliated by closed Legendrian curves parallel to $\beta$ which are contained in one of the surfaces $\Sigma_t, t>-1$. 
The sheet $A(\beta)$ has the following properties: 
\begin{itemize}
\item[(i)] $A(\beta)$ is transverse to both $\II$ and $\xi$. 
\item[(ii)] When a connected component $A(\beta)\cap \Sigma_t$ is an attractive closed leaf, then at each point $p$ of that leaf $\xi$ is steeper than the tangent space of $A(\beta)$ at that point (cf. \eqref{e:order} on p.~\pageref{e:order}).
\item[(iii)] The projection of the closed Legendrian curves foliating $A(\beta)$ to $\Sigma$ along $\II$ provide an isotopy of simple closed curves $\beta_\tau\subset\Sigma$  starting with the curve $\beta=\beta_0$. We lift this isotopy of curves to an isotopy $\widetilde{\beta}_\tau$ with fixed points on the ideal boundary of the universal covering $\HH^2\lra \Sigma$. 
\end{itemize}
By \lemref{l:Sigma ator} $\widetilde{\beta}_\tau$ contains a point which is connected to a point in $\widetilde{\beta}_0$ by a geodesic $\widetilde{\gamma}_\tau$ which is shorter than the constant $K$ from \lemref{l:Sigma ator}. Consider the characteristic foliation of the lifted contact structure $\widetilde{\xi}$ on $\widetilde\Sigma\times[-1,1]$ on the surface $\widetilde{\Gamma}=\widetilde{\gamma}_\tau\times[-1,1]$. Let $\widetilde{\xi},\widetilde{\FF}$ denote the lifts of $\xi,\FF$ to the universal covering $\HH^2\times[-2,2]$ of $\Sigma\times[-2,2]$. 

If $A(\beta)$ meets $\Sigma\times\{\delta_1^-\}$, then by the discussion in \secref{s:sheets} (in particular \eqref{e:order})  the above properties of $A(\beta)$ imply that the leaf of the characteristic foliation on  $\widetilde{\gamma}_\tau\times[-2,2]$ lies {\em above} the intersection of $A(\beta)$ with $\widetilde{\gamma}_\tau\times[-2,2]$. Therefore, if $A(\beta)$ meets $\Sigma\times\{\delta_1^-\}$ then so does the leaf of $\widetilde{\Gamma}(\widetilde{\xi})$ which starts at $\beta_0$. But this is excluded by the choice of $\eps$. This proves the lemma.
\end{proof}
\begin{lem} \mlabel{l:remove separating}
There is a convex surface $\Sigma'$ in $\Sigma\times[\delta_{-g+1}^-,\delta_{g-1}^-]$ transverse to $\II$ such that the dividing set on $\Sigma'$ has no separating component.
\end{lem}

\begin{proof} 
After a $C^\infty$-small perturbation of $\xi$ we may assume that $\Sigma_{-1}$ is convex. 
Assume that the dividing set on $\Sigma_{-1}$ has separating components. Then there is a separating attractive closed leaf of $\Sigma_{-1}$. Let $A(\beta)$ be the sheet of the movie $(\Sigma\times[-1,1],\xi)$ whose boundary contains $\beta$.  Then $A(\beta)$ is an annulus and we put
$$
t_{max}(\beta)=\max\{t\in[-1,1]\,|\,A(\beta)\cap \Sigma_t\neq\emptyset\}.
$$
This is the highest level the sheet $A(\beta)$ reaches. Since the movie has no negative singularities $A(\beta)\cap\Sigma_{t_{max}(\beta)}$ is a degenerate closed curve $\widehat{\beta}$. 
 
We assume that $A(\beta)$ has the property that $t_{max}:=t_{max}(A(\beta))$ is maximal among the levels $t_{max}(\widetilde{\beta})$  for the finite number of separating attractive closed leaves $\widetilde{\beta}$  of $\Sigma_{-1}(\xi)$ isotopic to $\beta$. We can also arrange that the sheet $A(\beta)$ is as simple as possible, i.e. there is no degenerate closed orbit on $A(\beta)$ between $\beta$ and $\widehat{\beta}$. By \lemref{l:no passing} 
$$
t_{max}(A(\beta))<\delta_{1}^-.
$$ 

We now replace $\Sigma_{-1}$ by a smooth surface $\Sigma'$ transverse to $\II$ and close to,  but never below the surface consisting of
\begin{itemize}
\item[(i)] the connected component  of $\Sigma_{-1}\setminus\beta$ on the side of $\beta$ determined by the coorientation of $\beta$ inside $\Sigma_{-1}$,
\item[(ii)] the part of the sheet $A(\beta)$ which lies between $\beta$ and $\widehat{\beta}$, and 
\item[(iii)] the connected component $\widehat{\Sigma}$ of $\Sigma_{t_{max}}\setminus \widehat{\beta}$ on the repelling side of $\widehat{\beta}$. 
\end{itemize}
If $\widehat{\Sigma}$ contains an attractive closed curve isotopic to $\widehat{\beta}$, then we pick the curve closest to $\widehat{\beta}$ and denote it by $\beta'$. There are two possibilities:
Either $\beta$ is parallel to $\beta'$ or not. 

In the first case, the sheet $A(\beta')$ is part of a sheet containing $A(\beta)$. This contradicts the definition of $t_{max}(\beta)$. Hence $\beta'$ is anti-parallel to $\widehat{\beta}$. Notice that the sheet $A(\beta')$ does not intersect $\Sigma_{-1}$ since this would contradict the maximality of $t_{max}(\beta)$. We proceed by replacing $\Sigma'$ by a surface $\Sigma''$ consisting of 
\begin{itemize}
\item[(i)] the part of $\Sigma'$ lying on the side of $\beta'$ which contains $\widehat{\beta}$,
\item[(ii)] the part of the sheet of $A(\beta')$ below $\Sigma'$ and $\Sigma_{-1}$ and the degenerate closed curve $\widehat{\beta}'$ which is contained in $t_{min}(\beta')\in (-1,t_{max})$ (the definition of this number is by now obvious), and
\item[(iii)] the part of $\Sigma_{t_{min}(\beta')}$ on the side of $\widehat{\beta}'$ which is determined by the coorientation of $\widehat{\beta}'$ in $\Sigma_{t_{min}(\beta')}$. 
\end{itemize}

Now we repeat this process. The next attractive closed curve we encounter in the process is part of a sheet containing $A(\beta)$. Therefore we cannot pass the level $\Sigma_{t_{max}}$ in the next step and the same applies to all steps that follow it. This procedure terminates since otherwise we would find a degenerate closed curve isotopic to $\beta$ which is the limit of degenerate closed curves isotopic to $\beta$. 
The resulting surface will be called $\Sigma'$. By construction, $\Sigma'$ has no attractive closed curves isotopic to $\beta$ which  lie in the part of $\Sigma'\setminus\widehat{\beta}$ not containing $\beta$.

The remaining attractive closed curves parallel to $\beta$ are easy to eliminate since the sheets such curves can be completed to sheets which are properly embedded, i.e. both of whose boundary components are contained in $\Sigma_{-1}$. The attractive closed curves which are anti-parallel to $\beta$ are dealt with in the same way explained above. This step does not interfere with what we have achieved already since it all happens in the half of $\Sigma'$ which is contained in $\Sigma_{-1}$. 

We have removed all separating attractive closed curves which are parallel or anti-parallel to $\beta$ and we want to eliminate the remaining separating attractive closed curves. 
We proceed with the non-separating curve $\widetilde{\beta}$  which, together with $\beta$, bounds a subsurface containing no other attractive separating leaf. Now $\widetilde{\beta}$ can be treated essentially in the same way as $\beta$, except that we leave the part of the surface on the side of $\widetilde{\beta}$ which contains $\beta$ as it is. Depending on whether the other side of $\widetilde{\beta}$ coincides with the side determined by the coorientation of $\xi$ or not, the surface is shifted towards higher or lower levels. 

\figref{b:rem-degen} shows this schematically. The isotoped surface is thickened. Only parts of sheets where the corresponding curve in $\Sigma_t$ is attractive and neighbourhoods of degenerate closed curves are shown. \begin{figure}[htb]
\begin{center}
\includegraphics[scale=0.9]{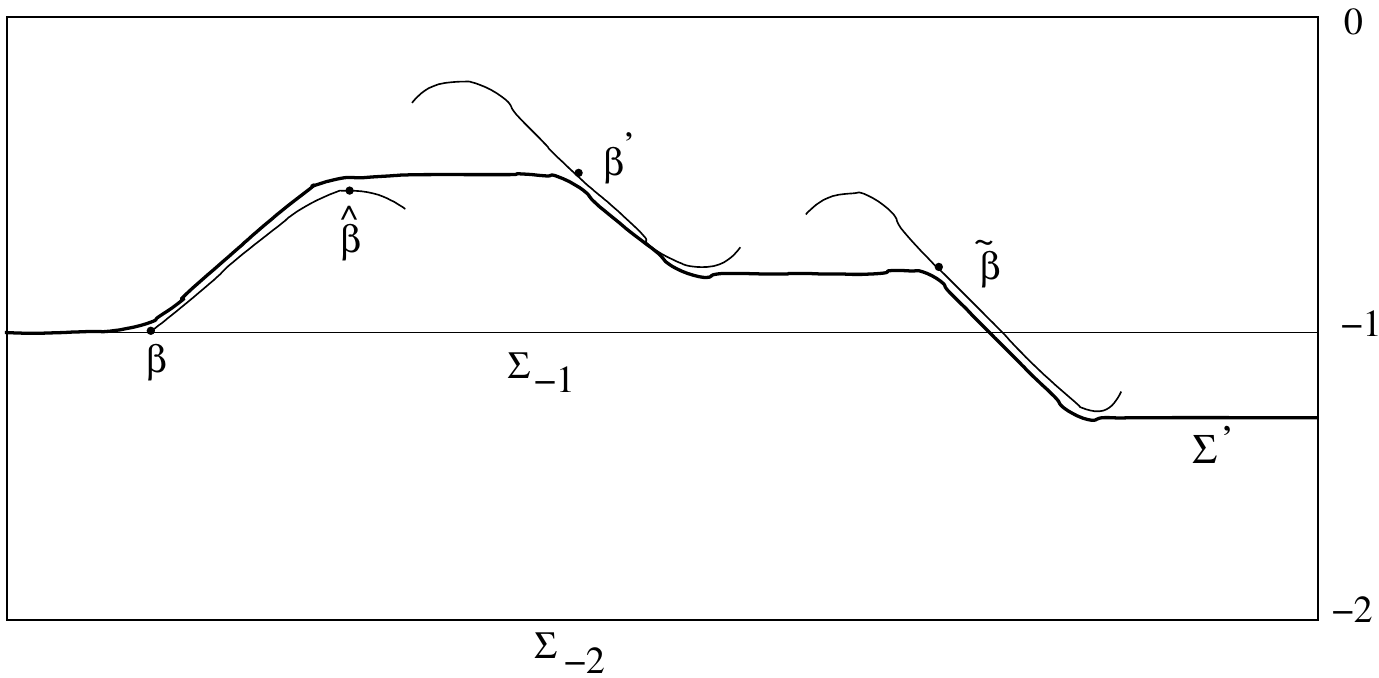}
\end{center}
\caption{Elimination of separating curves}\label{b:rem-degen}
\end{figure}

Since a pairwise disjoint collection of separating homotopically essential simple closed curves contains at most $g-1$ (where $g$ is the genus of $\Sigma$) isotopy classes of non-oriented curves this process stops after a finite number of steps and we have found the desired surface. By our choice of $\eps$ the resulting surface is contained in $\Sigma\times[\delta_{-g+1}^-,\delta_{g-1}^-]$.
\end{proof}

After the elimination of all separating attractive curves we obtain a convex surface $\Sigma'$ in $\Sigma\times[\delta_{-g+1}^-,\delta_{g-1}^-]$ such that no dividing curve is separating. The procedure from the proof of \lemref{l:remove separating} allows us prove that the contact structures are tight. 
\begin{prop} \mlabel{p:tight xi1}
The restrictions of $\xi_1$ and $\xi_2$ to $\Sigma\times[\delta^-_{-g},\delta^+_{-g}]$ are tight.
\end{prop}

\begin{proof}
By construction, $\xi_1,\xi_2$ are transverse to $\II$ on $\Sigma\times[\delta_{-g}^-,\delta_{-g}^+]$. 
After a $C^\infty$-small perturbation we may assume that $\Sigma\times\{\delta^{\pm}_{-g}\}$ is convex. We want to apply \lemref{l:tight extension}. For this we have to show that no connected component of the dividing set bounds a disc. Assume that there is a component of the dividing set which bounds a disc. By \remref{r:disc bounding tight} there is a closed attractive orbit $\beta$ bounding a disc. Let $A(\beta)$ be the corresponding sheet. 

After going through the procedure in the proof of \lemref{l:remove separating} we find a disc on a surface $\Sigma_t$ in $\Sigma\times[-2,2]$ such that this surface contains no closed orbits although the characteristic foliations point inwards along the boundary. But this requires the presence of a negative singularity. However if $\xi$ is $\eps$-close to $\FF$, then there are no such singularities. Hence \lemref{l:tight extension} proves the claim. 
\end{proof}

We are now in a position to eliminate all but one pair of dividing curves from the surface $\Sigma'$  obtained in \lemref{l:remove separating}.


 \begin{lem} \mlabel{l:remove components}
 Let $\Sigma\subset\Sigma\times[\delta_{-g+1}^-,\delta_{g-1}^-]$  be a convex surface transverse to $\II$ such that all dividing curves are non-separating. Then $\Sigma$ is isotopic to a convex surface $\Sigma'$ in $\Sigma\times[\delta_{-g}^-,\delta_{g}^-]$ 
 \end{lem}
 
 \begin{proof}
 We assume that there are at least four dividing curves, all of which are non-separating. Fix an attractive closed leaf $\beta$ in $\Sigma(\xi)$ so that $t_{max}(\beta)$ is minimal among the finitely many attractive closed leaves in $\Sigma_{-1}$. Since $\beta$ is non-separating and there is another attractive closed orbit, we can use the theory of convex surfaces to change the characteristic foliation on $\Sigma$  so that
 \begin{itemize}
 \item there is a repulsive closed leaf $\beta'$ parallel to $\beta$ on the side of $\beta$ opposite to the side determined by the coorientation of $\xi$, and
 \item all leaves of the characteristic foliation which do not lie in the annulus bounded by $\beta$ and $\beta'$ accumulate on an attractive closed curve different from $\beta$.
\end{itemize}
This can be done without changing any of the sheets containing closed attractive leaves of $\Sigma(\xi)$. 

By our assumptions on $\eps$ and the choice of neighbourhoods of the closed leaf $\Sigma$ the sheet $A(\beta)$ containing $A(\beta)$ does not enter $\Sigma\times[-\delta,\delta]$. We choose an identification of the region bounded by $\Sigma$ and $\Sigma_{-\delta}$ with $\Sigma\times[-1,-\delta]$ so that 
\begin{itemize}
\item $\Sigma$ corresponds to $\Sigma_{-1}$,
\item the foliation corresponding to the second factor is tangent to $\II$, and
\item the Legendrian foliation on the sheets containing closed attractive curves of $\Sigma(\xi)$ is tangent to the level surfaces of the product decomposition $\Sigma\times[-1,-\delta]$. 
\end{itemize}
Now we apply the pre-Lagrangian extension lemma (\lemref{l:completing sheets}) to $\beta$ relative to the sheets containing other closed leaves of $\Sigma_{-1}(\xi)$. We obtain a properly embedded sheet $A'(\beta)$ of a contact structure $\xi'$ isotopic to $\xi$ so that $A'(\beta)$ connects $\beta$ to $\beta'$. 
We now replace $\Sigma_{-1}$ by a surface $\Sigma'$ close to the union of 
\begin{itemize}
\item[(i)] the sheet $A'(\beta)$ and 
\item[(ii)] $\Sigma_{-1}$ with the annulus bounded by $\beta'$ and $\beta$ removed 
\end{itemize}
If $\Sigma'$ is sufficiently close to this union and lies above it, then no new closed curves/negative singular points have appeared on $\Sigma'$ and we have eliminated one pair of dividing curves. This process can be iterated as long as there are at least two pairs of parallel dividing curves.
 \end{proof}
 

\subsubsection{Identification of the contact structure} \mlabel{s:det contact}

By the lemmas from the previous section, we find a domain $N(\Sigma_0)$  diffeomorphic to $\Sigma\times[-1,1]$ inside $M$ such that the boundary has the following properties:
\begin{itemize}
\item It is convex and contained inside $\Sigma\times([-2,2]\setminus[-\delta,\delta])$. 
\item One boundary component lies above $\Sigma_0$ while the other boundary component lies below $\Sigma_0$. 
\item The dividing set of the characteristic foliation of $\xi$ consists of two non-separating closed curves on each boundary component. 
\end{itemize}
Since the contact structures $\xi_1,\xi_2$ are tight on $N(\Sigma_0)$ 
we can  apply \thmref{t:Sigma class} to determine the isotopy class of $\xi$. By our assumptions on $\eps$, both contact structures have the following property: The sheet containing an attractive closed leaf of the characteristic foliation on a boundary component of $N(\Sigma_0)$ does not enter $\Sigma\times[-\delta,\delta]$. 

From now on we identify $N(\Sigma_0)$ with $\Sigma\times[-1,1]$ in such a way that sheets containing attractive closed leaves of the characteristic foliation on the boundary are preserved and nothing changes on $\Sigma\times[-\delta,\delta]$.

First we determine the relative Euler class of a contact structure $\xi$ on $\Sigma\times[-1,1]$. 
For this we apply \lemref{l:completing sheets} to $A(\beta_{\pm 1})$ where $\beta_{\pm}$ is the unique closed attractive leaf of $\Sigma_{\pm1}(\xi)$. 

We obtain a boundary elementary contact structure such that the boundary component $\beta'_{-1}$ of $A(\beta_{-1})$ which is different from $\beta_{-1}$ lies on the side of $\beta_{-1}$ opposite to the coorientation of $\beta_{-1}$ determined by the coorientation of $\xi$. For $A(\beta_{+1})$ the situation is opposite. 

The map $\varphi$ in the following lemma is an automorphism of the surface isotopic to a left-handed Dehn twist along $\beta_{-1}$.

\begin{lem} \mlabel{l:bypass}
Let $\xi$ be a boundary elementary tight contact structure on $\Sigma\times[-1,1]$ with the properties from the previous paragraphs. Let $\alpha\subset\Sigma$ be a simple closed oriented curve with $\alpha\cdot\beta_{-1}=1$.  Then for sufficiently large $k>0$ there is an annulus $S(\alpha_k)$ with primped Legendrian boundary isotopic to $\alpha_k:=\varphi^k(\alpha)$ connecting  $\Sigma_{-1}$ to $\Sigma_{+1}$ so that there is a Legendrian curve $\widehat{\alpha}_k$ in the interior of $S(\alpha_k)$ containing no singular point of $S(\alpha_k)(\xi)$. 
\end{lem}

\begin{proof}
We start with $\Sigma_{-1}$ and arrange that $\partial A(\beta_{-1})\subset\Sigma_{-1}$ consists of two circles of singularities (negative along $\beta_{-1}$ and positive at the other end of $A(\beta_{-1})$). Using \lemref{l:convex families}, we can moreover arrange that $\alpha$ is a Legendrian curve in $\Sigma_{-1}$ which intersects the annulus bounded by $\partial A(\beta_{-1})$ in a single arc so that  $\alpha$ is primped (cf. \secref{s:Sigma class}) and the only negative singularity on $\alpha$ is $\alpha\cap\beta_{-1}$. 

Now consider the smooth surface $\Sigma''$ obtained from $\Sigma_{-1}$ after replacing the annulus bounded by $\partial A(\beta_{-1})$ in $\Sigma_{-1}$ by $A(\beta)$ and then smoothing out the non-smooth points.  Using a vector field $X$ transverse to the surface, we push $\Sigma''$ into the interior of $\Sigma\times[-1,1]$. We may assume that $X$ is a contact vector field outside of a small neighbourhood of $\partial A(\beta_{-1})$. Note that the part of $\Sigma''$ contained in $A(\beta_{-1})$ is foliated by closed leaves of the characteristic foliation, so $\Sigma''$ is certainly not convex. The closed Legendrian curve close to $\beta$ is repelling while the closed Legendrian curve at the opposite end of $A(\beta_{-1})$ is repulsive.

By \lemref{l:gi-birth-death} the collection of closed Legendrian curves forming on $\Sigma''$ curves disappears as we push this surface into $N(\Sigma_0)$ using the flow of $X$ and leaves of the characteristic foliation close to $\beta_{-1}$ get connected to leaves of the characteristic foliation on the opposite side of $A(\beta_{-1})$. If the flow runs for an appropriate time, the characteristic foliation on the pushed-off surface  connects the two arcs of $\alpha$ which lie in the part of the surface further away from $A(\beta)$. (The sequence of instances where this happens has $0$ as an accumulation point.) 

For an appropriate push-off we obtain a surface $\Sigma'$ containing a closed curve isotopic to a  curve obtained from $\alpha$ by applying a sufficiently high power of a left-handed Dehn twist $\varphi$  along $\beta$.  The annulus $\beta$ is now obtained by isotoping $\Sigma_{\pm 1}$ such that there is a Legendrian curve $\widehat{\alpha}_k$ isotopic to  $\varphi^k(\alpha)$ which intersects $\beta_{-1}$ exactly once. 

The annulus $S(\alpha_k)$ is then obtained by picking an annulus bounding the Legendrian curves isotopic to $\alpha_k$ in $\Sigma_{\pm 1}$ and containing $\widehat{\alpha}_k$. Since the twisting of $\xi$ along $\widehat{\alpha}_k$ is zero with respect to the framing determined by $\Sigma'$ this twisting vanishes also with respect to the framing determined by $S(\alpha_k)$. Therefore we can eliminate the singular points of $S(\alpha_k)(\xi)$ which lie on $\widehat{\alpha}_k$. 
\end{proof}

We now decompose $\Sigma\times[-1,1]$ as follows. Start with $\Sigma'$ and modify this surface using the pre-Lagrangian extension lemma in a similar way as in the proof of \lemref{l:remove components} to reduce the number of dividing curves to $2$ without introducing any new ones. For this recall the following facts.
\begin{itemize}
\item All dividing curves intersect $\beta_{-1}$ at least once and always with the same sign and thus there are no null-homologous dividing curves in $\Sigma'$.
\item From the way we obtained $\Sigma'$ it follows that we may assume that the domain between $\Sigma'$ and $\Sigma_{-1}$ does not contain negative singularities except those along $\beta_{-1}\subset\Sigma_{-1}$ (these singularities lie on the Legendrian curve $\partial S(\alpha_k)$.)   
\end{itemize} 
The resulting surface is called $\widehat{\Sigma}$ and the dividing set on this surface is such that $\alpha_k$ is isotopic to a curve disjoint from the dividing set (which consists of two parallel copies of the curve $\gamma$).

Hence if $\widehat{S}$ is an annulus in the domain bounded by $\widehat{\Sigma}$ and $\Sigma_{-1}$, whose boundary is isotopic to $\alpha_k$, then 
\begin{equation} \label{e:which basic}
\langle\widetilde{e}(\xi),\widehat{S}\rangle = -1 =  (\pm \gamma+\beta_{-1} )\cdot \alpha_k = - \alpha\cdot\beta_{-1}
\end{equation}
if we use the coorientation of $\beta_{-1}$ in order to orient $\widehat{S}$. In other words $\widehat{S}$ is oriented so that $-\alpha_k$ is part of the oriented boundary of $\widehat{S}$. This means that we have determined the sign in front of $\beta_{-1}$ in the expression of the Poincar{\'e} dual of $\widetilde{e}(\xi)$ in \eqref{e:sum}. The coefficient in front of $\beta_{+1}$ is determined in the same way looking at the other boundary component $\Sigma_{+1}$. Since the pre-Lagrangian annulus $A(\beta_{+1})$ now lies on the side of $\beta_{+1}$ determined by the coorientation of $\xi$, we get a minus sign. Hence
\begin{equation} \label{e:which basic2}
\mathrm{PD}\left(\widetilde{e}(\xi)\right)= -\beta_{+1}+\beta_{-1}.
\end{equation}

Thus if $\beta_{-1}$ and $\beta_{+1}$ are not homologous, then we have determined the contact structure up to isotopy because by \thmref{t:Sigma class} it is determined by the relative Euler class. 

If $\beta_{-1}$ and $\beta_{+1}$ are homologous, then \eqref{e:which basic2} implies $\widetilde{e}(\xi)=0$ and we have to study which basic slice admits a contact embedding into $\Sigma\times[-1,1]$ such that one boundary components gets mapped to $\Sigma_{-1}$ in an orientation preserving fashion. If $\widehat{\alpha}_k$ is attractive, then we have already found a basic slice. If $\alpha_k$ is repulsive, then after folding we obtain a surface with $4$ dividing curves. After removing the dividing curves which do not come from this folding procedure we end up again with a basic slice
$$
\llbracket \beta_{-1},\beta'; -\beta'+\beta_{-1} \rrbracket
$$
with $\beta'$ isotopic to $\widehat{\alpha}_k$. (The folding procedure provides us with a pre-Lagrangian annulus below $\Sigma'$ which lies on the side of $\alpha_k$ determined by the coorientation of $\xi$).  

Therefore the contact structure on $\Sigma\times[-1,1]$ is completely determined by the contact structure on the complement of $\Sigma\times[-\delta,\delta]$. This completes the proof of \propref{p:outside determines inside} and with it the proof of \thmref{t:unique} for all confoliations except foliations without holonomy.

\section{Foliations without holonomy} \mlabel{s:without holonomy}

The proof of \thmref{t:unique} is almost finished, what is left open is the case of foliations without holonomy which will be discussed in this section. We shall make use of the structure theory of foliations without holonomy which was developed in particular by R.~Sacksteder and S.~Novikov.  The following theorem can be found in \cite{cc} (recall that  we assume that $M$ is closed). 
\begin{thm} \mlabel{t:fibered or dense}
Let $\FF$ be a $C^2$-foliation without holonomy on $M$. Either every leaf of $\FF$ is dense or the leaves of $\FF$ are the fibers of a fibration $M\lra S^1$. 
\end{thm}
In particular, a foliation without holonomy is automatically taut since non-compact leaves in closed manifolds always have a closed transversal. In the following we consider only such neighbourhoods of $\FF$ which are so small that every contact structure in that neighbourhood is automatically universally tight. Such a neighbourhood exists by \thmref{t:taut tight}. 

We shall deal first with the case that every leaf is closed and we will determine precisely which torus fibrations satisfy the conclusion of \thmref{t:unique2}. Then we will finally prove \thmref{t:unique} for foliations without holonomy all of whose leaves are dense.

\subsection{Every leaf of $\FF$ is a torus} \mlabel{s:torus bundles}

Let $\mathrm{pr} : M \lra S^1$ be an orientable torus fibration over the circle. The diffeomorphism type of $M$ is determined by the action of the monodromy of the fibration on the homology of a fiber
\begin{align*}
\phi_M : H_1(T^2,\Z)\lra H_1(T^2,\Z).
\end{align*}
Thus we may assume that $M=T^2\times\R/\simeq$ with $(x,t)\simeq (Ax,t+1)$ for $A\in\mathrm{Sl}(2,\Z)$ since $M$ is orientable.

There are infinitely many isotopy classes of positive contact structures on $M$: On the universal cover $\R^2\times\R$ of $M$ consider the $1$-form 
$$
\cos(\varphi(t))dx_1-\sin(\varphi(t))dx_2
$$ 
where $\varphi$ is a strictly increasing function. According to \cite{gi-inf} for each integer $n\ge 0$ one can chose $\varphi$ such that the corresponding contact structure on $\R^3$ is invariant under the action of $\pi_1(M)$ and 
\begin{equation*}
2n\pi<\sup_{t\in\R}(\varphi(t+1)-\varphi(t))\le 2(n+1)\pi.
\end{equation*}
The resulting contact structures do not depend on the particular choice of $\varphi$ but the induced contact structures $\xi_n$ on $M$ are not isotopic for different integers $n$.  

The universally tight contact structures on $M$ have been classified by K.~Honda \cite{honda2} and E.~Giroux \cite{gi-bif, gi-inf}. For our purposes the following statement of their results is sufficient:
\begin{thm} \mlabel{t:tight torus bundles}
If $\left|\mathrm{tr}(\phi_M)\right|\neq 2$, then all positive universally tight contact structures on $M$ are isotopic to one of the contact structures $\xi_n$ defined above. 
\end{thm}
This implies that the conclusion of \thmref{t:unique2} also holds for torus bundles over $S^1$ whose monodromy is elliptic or hyperbolic. If the monodromy $A\in\mathrm{Sl}(2,\Z)$ satisfies $\mathrm{tr}(A)=2$, then this is not the case as will be shown in \exref{ex:tr2}.

\begin{thm} \mlabel{t:torus bundles}
Let $\FF$ be the foliation defined by a torus bundle $\mathrm{pr} : M\lra S^1$. Then there is a $C^0$-neighbourhood of $\FF$ in the space of plane fields such that all positive contact structures in that neighbourhood are stably isotopic if and only if $+1$ is not an eigenvalue of $A$. 
\end{thm}

\begin{proof}
If the monodromy is elliptic or hyperbolic, then $+1$ is not an eigenvalue of $A$ and  \thmref{t:tight torus bundles} proves the claim. The case $\mathrm{tr}(A)=+2$ will be treated in \exref{ex:tr2} below. Thus we are left with the case when $\mathrm{tr}(A)=-2$ in which the classification of universally tight contact structures is slightly more complicated and we have to prove our claim more directly:

Let $U$ be the neighbourhood of $\FF$ determined by the requirement that all plane fields on $U$ are transverse to the   line field $\partial_t$ on $M$. \thmref{t:tight connection} implies that all contact structures in $U$ are universally tight. 

Consider a positive contact structure $\xi$ in $U$ and consider the characteristic foliations on the fibers of $M$. After a $C^\infty$-small perturbation of $\xi$ we may assume that the fiber $T_0=\mathrm{pr}^{-1}(0)$ is convex. Let $\beta$ be an attractive closed leaf of $T_0(\xi)$ and $A(\beta)$ the sheet containing $\beta$. 

Assume that $A(\beta)$ connects the two boundary components of $\overline{M\setminus T_0}$. Then $T(\xi)$ has another closed attractive leaf $\beta'$  isotopic to $\beta$ with the orientation reversed (because $-1$ is the only eigenvalue of the monodromy $A$). Now consider the cyclic  covering $\widehat{M}$ of $M$ given by   

$$
\xymatrix{
\widehat{M} \ar[r]^{\widehat{\mathrm{pr}}} \ar[d]& \R\ar[d] \\
M \ar[r]^{\mathrm{pr}} & S^1.      }
$$
Let $\widehat{T}=\widehat{\mathrm{pr}}^{-1}(0)$ and consider the maximal sheets $\widehat{A}(\beta),\widehat{A}(\beta')$ containing lifts of $\beta,\beta'\subset\widehat{T}$. Let
\begin{align*} 
\widehat{t}(\widehat{A}(\beta)) & = \sup\left\{\widehat{t}\in\R\,\left|   \widehat{A}(\beta)\cap\widehat{\mathrm{pr}}^{-1}(\widehat{t})\neq\emptyset\right\}\right. \\
\widehat{t}(\widehat{A}(\beta')) & = \sup\left\{\widehat{t}\in\R\,\left|   \widehat{A}(\beta')\cap\widehat{\mathrm{pr}}^{-1}(\widehat{t})\neq\emptyset\right\}\right..
\end{align*}
One of these numbers is finite because otherwise the sheets $\widehat{A}(\beta)$ and $\widehat{A}(\beta')$ would intersect. This is impossible unless they actually coincide, but this impossible by  \remref{r:sheet prop} and the fact that the characteristic foliations on the fibers of $M\lra S^1$ have no negative singularities. Hence we can reduce the number of dividing curves on $T$ as in \secref{s:clean boundary} and following \cite{gi-bif} we find an embedded torus $T'\subset M$ isotopic to the fiber such that $T'(\xi)$ has no singularities and no closed orbits.  

The classification of universally tight contact structures on $T^2\times[0,1]$ (cf. \thmref{t:tight on TxI}) such that the characteristic foliation on the boundary is of the same type as $T'(\xi)$  implies the claim. 
\end{proof}

\subsection{Every leaf of $\FF$ is compact and has genus $g\ge 2$} \mlabel{s:without hol closed}

Fix a hyperbolic metric on a fiber $\Sigma_0$ and let $K>0$ be the constant from \lemref{l:Sigma ator}. We identify a foliated tubular neighbourhood of $\Sigma_0$ with $\Sigma\times(-\delta,\delta)$ and a foliation $\II$ transverse to the leaves of $\FF$ such that $\II$ is tangent to the fibers of $\Sigma\times(-\delta,\delta)\lra\Sigma$.  

We assume that $\eps>0$  satisfies the following condition: 

For every path  $\gamma$ of length at most $K+1$ and $i=-g+1,\ldots,g-1$ and every smooth plane field $\zeta$ which is $\eps$-$C^0$-close to $\FF$, the $\zeta$-horizontal lift of $\gamma$ with starting point in $\Sigma_{i\delta/g}$ does not meet $\Sigma_{(i-1)\delta/g}$ or $\Sigma_{(i+1)\delta/g}$. Moreover, we require that $\eps$ is so small that every contact structure $\eps$-close to $\FF$ is tight (such an $\eps$ exists because $\FF$ is taut).

The first three steps are the same as those used in the proof of \propref{p:outside determines inside}.  We do not go through the details again but here are the steps: Let $\xi_0$ and $\xi_2$ be two contact structures $\eps$-close to $\FF$. 
\begin{enumerate}
\item After a $C^\infty$-small perturbation of $\xi_0$ and $\xi_2$ we may assume that $\Sigma_0$ is convex for both surfaces. 
\item After a finite number of isotopies  isotopies of the surface $\Sigma_0$ we obtain a surface $\widehat{\Sigma}^{(0)}$ respectively $\widehat{\Sigma}^{(2)}$ which is convex with respect to $\xi_0$ respectively $\xi_2$ in $\Sigma\times(-(g-1)\delta/g,(g-1)\delta/g)$and whose dividing set has no separating component. (In this step we use \lemref{l:remove separating}.)
\item After sufficiently many applications of \lemref{l:remove components}  we end up with surfaces $\Sigma^{(0)}, \Sigma^{(2)}\subset(-\delta,\delta)$ whose dividing set consists of exactly two non-separating closed curves which bound a single annulus.\end{enumerate}

Now cut $M$ along $\Sigma^{(0)}$. The resulting manifold is diffeomorphic to $\Sigma\times[0,1]$ and the boundary satisfies the assumptions of \thmref{t:Sigma class}. Moreover we have shown that the contact structure on $\Sigma\times[0,1]$ is determined by the orientations of the attractive closed leaves of $\Sigma^{(0)}(\xi)$ (see \lemref{l:bypass} and the discussion following it):

Let $\gamma_0$ denote   the attractive closed leaf of $\Sigma^{(0)}(\xi_0)$. The relative Euler class of $\xi_0$ on the cut open manifold is Poincar{\'e}-dual to $\phi_*(\gamma_0)-\gamma_0$  where $\phi : \Sigma^{(0)}\lra \Sigma^{(0)}$ denotes the monodromy of the fibration. If the relative Euler class vanishes, then the contact structure is still determined by embedding properties of basic slices. The remaining problem is that the dividing sets of $\Sigma^{(0)}(\xi_0)$ and $\Sigma^{(2)}(\xi_2)$ are not isotopic in general. 

But as in the first step of the proof of \thmref{t:unique} in \secref{s:higher genus proof} we can find a contact structure $\xi_1$ isotopic to $\xi_2$ and a convex surface $\Sigma_{\gamma^{(2)}}$ in $(M\setminus\Sigma^{(0)},\xi_0)$  transverse to $\II$ such that the dividing set on $\Sigma_{\gamma^{(2)}}$ is isotopic to $\Sigma^{(2)}(\xi_2)$. For this we use that the part of the curve complex of $\Sigma$ whose vertices are non-separating is connected, cf. Proposition 3.3 in \cite{hkm} and the construction from \exref{ex:periodic cont}.  After another isotopy which does not affect the sheet in $(M\setminus\Sigma^{(0)},\xi_1)$ which contains the unique attractive closed leaf in $\Sigma_{\gamma^{(2)}}$ (note that this leaf is oriented parallel to the unique closed attractive leaf of $\Sigma^{(2)}(\xi_2)$). As in \secref{s:higher genus proof} this implies that 
$$
(M\setminus\Sigma^{(2)},\xi_2) \textrm{ and } (M\setminus\Sigma_{\gamma^{(2)}},\xi_1)
$$   
are isotopic after the surfaces $\Sigma^{(2)}$ and $\Sigma_{\gamma^{(2)}}$ are identified using the leaves of $\II$.

\subsection{Every leaf of $\FF$ is dense} \mlabel{s:facts without}


First, we review standard facts from the theory of foliations one without holonomy and codimension one and explain how a general foliation without holonomy can be understood in terms of a foliated fibre bundle on a manifold of higher dimension. This uses results of Novikov \cite{nov} and our presentation follows \cite{mortsu} closely. We will always assume that the underlying manifold is closed and that the foliation is at least $C^2$-smooth. We fix a simple closed curve $C$ transverse to $\FF$ (such a curve exists because there are non-compact leaves). 
\begin{thm}[Sacksteder \cite{sack}] \mlabel{t:sacksteder without}
If $\FF$ is a $C^2$-foliation without holonomy then there is a $C^0$-flow 
\begin{equation} \label{e:sf}
\psi : M\times\R \lra M 
\end{equation}
which preserves the leaves of $\FF$ and acts transitively on the leaves. The flow can be chosen tangent to any previously fixed foliation $\LL$ of rank $1$ transverse to $\FF$. Given a closed curve transverse to $\FF$ we can choose $\psi$ such that $C$ becomes a flow line. 
\end{thm}

In order to prove this theorem, one establishes the existence of a  holonomy invariant transverse measure $\mu$. The relationship between $\mu$ and $\psi$ is $\mu(\psi([0,s],x))=s$. The transverse measure determines a group homomorphism 
\begin{align*}
\varphi_\mu : \pi_1(M) & \lra \R \\
[\gamma] & \lmt \int_{S^1} \gamma^*(d\mu) .
\end{align*}
The image $\mathcal{P}(\mu)$ of $\varphi_\mu$ is called the group of periods of $\FF$. It  consists of those times $s$ where  $\psi(s,\cdot)$ maps one (or equivalently every) leaf to itself. The leaves of $\FF$ are fibers of a fibration if and only if $\mathcal{P}(\mu)$ is discrete. 


Let $\widetilde{M}\lra M$ be the universal covering and $\widetilde{\FF}$ the induced foliation, $L$ a leaf of $\FF$ and $\widetilde{L}\lra L$ the universal covering of $L$. Since $M$ is taut, $\widetilde{L}$ is a leaf of $\widetilde{\FF}$ and lifting $\psi$ to $\widetilde{M}$ we obtain a diffeomorphism 
$$
\widetilde{\psi} : \widetilde{L}\times\R  \lra \widetilde{M}.
$$
We denote the projection onto the second factor of $\widetilde{L}\times \R$ by $\pi$.  We get a representation 
\begin{align} \label{e:rep G} 
\begin{split}
q: \pi_1(M) = \mathrm{Deck}(\widetilde{M}) & \lra \mathrm{Diff}_+^2(\R) \\
\alpha &  \lmt \left( x\lmt \pi\left((p_0,x)\cdot\alpha^{-1}\right)\right).
\end{split}
\end{align} 
Here $p_0$ is a base point in $\widetilde{L}$ and we abbreviate $\{p_0\}\times\R\subset\widetilde{L}\times\R$ by $\R$. Novikov proves the following facts about $q$:
\begin{itemize}
\item[(i)] If $q(\alpha)$ has a fixed point, then $q(\alpha)=\mathrm{id}$.
\item[(ii)] The image $G$ of $q$ is an Abelian group (it is obviously free and finitely generated).
\end{itemize}
The action of $\pi_1(M)$ on $\widetilde{M}$ by deck transformations and by $q$ on $\R$ determines a foliated $\R$-bundle 
\begin{equation} \label{e:fol bundle}
E=\widetilde{M}\times\R/\pi_1(M).
\end{equation} 
We denote the induced foliation on $E$ by $\FF_E$. By the definition of $q$ the embedding
\begin{align*}
\widetilde{\sigma} : \widetilde{M} & \lra \widetilde{M}\times\R\\
x & \lmt (x,\pi(x)) 
\end{align*}
is $\pi_1(M)$-equivariant and the resulting embedding $\sigma: M \lra E$ is transverse to $\FF_E$ and $\FF=\sigma^*\FF_E$. 

The element $q(C)=f_0$ is non-trivial and $f_0(x)>x$. Then according to \cite{stern} we may (smoothly) reparameterize $\R$ so that $f_0(x)=x+1$. Because the other elements of $G$ commute with $f_0$ we have that $f(x+1)=f(x)+1$ for all $f\in G$. Therefore the elements of $G$ have the same properties as lifts of orientation preserving $C^2$-diffeomorphisms of the circle $S^1=\R/\Z$ (we call such diffeomorphisms $1$-periodic). 

The outline we have given is somewhat misleading since Novikov's result can be used to prove Sacksteder's theorem. This is explained in chapter 9 of \cite{cc} where the representation \eqref{e:rep G} is obtained without using \thmref{t:sacksteder without} and the holonomy invariant transverse measure is then obtained by an averaging procedure.

\subsubsection{$1$-periodic diffeomorphisms of $\R$}

Let $\phi : \R \lra\R$ be a $1$-periodic homeomorphism of $\R$, we denote the group consisting of such diffeomorphisms by $\mathrm{Diff}^1(\R)$. It is well known (and explained in many places, e.g. \cite{godb}) that the sequence $\frac{\phi^n-id}{n}$ converges uniformly to a constant $\tau(\phi)$. We call this number the {\em translation number} of $\phi$. The fractional  part of this number is the rotation number $\rho\in\R/\Z$ of $\varphi$ when $\phi$ is a lift of $\varphi  : S^1\lra S^1$ to the universal covering $\R\lra S^1=\R/\Z$. All the following statements are consequences of well-known properties of the rotation number. 

Obviously, if $r_\alpha(x)=x+\alpha$ then $\tau(r_\alpha)=\alpha$. The rotation number depends continuously on $\phi$ with respect to the uniform topology and if $\phi_1,\phi_2$ commute, then $\tau(\phi_1\circ\phi_2)=\tau(\phi_1)+\tau(\phi_2)$. 

The following theorem is just a translation of a fundamental result in the theory of circle diffeomorphisms to the context of $1$-periodic diffeomorphisms of $\R$.

\begin{thm}[Denjoy] \mlabel{t:denjoy}
If the translation number $\tau$ of $\phi$ is irrational and $\phi$ is $C^2$-smooth, then there is a $1$-periodic homeomorphism $h$ of $\R$ such that $h\circ\varphi\circ h^{-1}=r_\tau$. The centralizer of $r_\tau$  in the group of $1$-periodic homeomorphisms of $\R$ consists of all translations of $\R$. 
\end{thm} 
In particular, the conjugating homeomorphism $h$ is unique up to composition with a translation. Moreover, note that if $\phi_1,\phi_2,\ldots,\phi_{n}$ are pairwise commuting homeomorphisms and one of them is conjugate to a translation, then one can use the same conjugating homeomorphism in order to conjugate $\phi_1,\ldots,\phi_{n}$ to translations simultaneously.

\subsubsection{Diophantine approximations} \mlabel{s:Diophant}

The theory of diffeomorphisms of the circle has strong connections to the theory of Diophantine approximations \cite{herman}. However, our use of the following theorem from the theory of Diophantine approximations is much more modest. A reference for the following result is p. 27 of \cite{schmidt}.\begin{thm}[Dirichlet] \mlabel{t:Diophant}
Let $\alpha_1,\ldots,\alpha_n$ be real numbers so that at least one of them is irrational. Then there are infinitely many $n+1$-tuples $(q,p_1,\ldots,p_n)$ such that $\mathrm{gcd}(q,p_1,\ldots,p_n)=1$ and 
\begin{equation} \label{e:irr}
\left|\alpha_i-\frac{p_i}{q}\right|<\frac{1}{q^{1+1/n}}.
\end{equation}
\end{thm}
The Thue-Siegel-Roth theorem states that if $\alpha$ is a real algebraic number and $\delta>0$, then
$$
\left|\alpha-\frac{p}{q}\right|<\frac{1}{q^{2+\delta}}
$$
has only finitely many solutions $p,q$ where $p,q$ are coprime integers. Therefore one cannot expect to improve the exponent in \eqref{e:irr}.

It is not known to the author whether one can arrange $\mathrm{gcd}(q,p_i)\in\{1,q\}$ for all $i$, maybe at the expense of replacing the function $q^{-1-1/n}$ in \eqref{e:irr} by another function $f(q)$ such that $qf(q)\to 0$ as $q\to\infty$. This would allow us to reduce the case of minimal foliations without holonomy to the previous case. i.e. we could find a fibration such that $\FF$  lies in a neighbourhood of the fibration in which the uniqueness theorem holds. 

\subsubsection{The uniqueness theorem for minimal foliations without holonomy}

The following fact is used in the proof of \thmref{t:elth}: Every foliation without holonomy can be $C^0$-approximated by a fibration. This is explained in Section 1.2.2 of \cite{confol} and later we shall find particular sequences of fibrations converging to $\FF$. We fix one fibration transverse to the flow lines of $\psi$ (cf. p.~10 of \cite{confol}) and a fiber $\Sigma_0$.  There are two cases:

{\bf The fibers have genus $\le 1$:}
If $\Sigma_0$ has genus $0$, then the foliation $\FF$ is a foliation by spheres and there is an $\eps$-neighbourhood of $\FF$ which does not contain any contact structure. When the fibers have genus $1$, then the leaves of $\FF$ are either all cylinders or they are all planes since the leaves of $\FF$ cover the fibers of the fibration. These are the cases excluded in \thmref{t:unique} and as we will explain in \secref{s:T3} there are infinitely many contact structures with vanishing Giroux torsion in any neighbourhood of a linear foliation by planes or cylinders on $T^3$. 

{\bf The fibers have genus $\ge 2$:}
Our first goal is to find a very good approximation of the foliation by a fibration using \thmref{t:Diophant}. For future reference we fix a hyperbolic metric on $\Sigma_0$.

Let $f_0,\ldots,f_{n}$ be generators of $G$ (the image of the map defined in \eqref{e:rep G}). As explained above, $f_i\circ f_j=f_j\circ f_i$ and we may assume $f_0(x)=x+1$. Because no map $f_i$ has a fixed point and all these maps are $1$-periodic they are conjugate to translations (either because they are all lifts of periodic circle diffeomorphisms, or because of \thmref{t:denjoy}). Since we assumed that $\FF$ is minimal, there is one generator, say $f_1$, with irrational translation number.  

The representation $q$ is conjugate to a representation $q'$ with $(q'(\alpha))(x)=x+\tau(q(\alpha))$ via a $1$-periodic homeomorphism $h$ of $\R$. We consider the foliated bundle induced by this representation. Using \thmref{t:Diophant} we approximate the numbers $\tau_i:=\tau(f_i), i=1,\ldots,n$, by rational numbers (recall that $\tau(f_1)$ is irrational). 

If $q$ is big enough, an approximation $p_1/q,\ldots,p_n/q$ of $\tau_1,\ldots,\tau_n$  from \thmref{t:Diophant} automatically satisfies the relations between $\tau_1,\ldots,\tau_n$ which come from relations in $\pi_1(M)$ not yet reflected by the fact that $G$ is Abelian. Fix a handle decomposition of $M$. We can $C^0$-approximate the foliation $\FF_E$ on $E\lra M$ on the preimage of $1$-handles by a foliation whose monodromy along a curve in $M$ is $h'\circ\tau_{p_i/q}\circ h'^{-1}$ when the original monodromy along the curve was $f_i$ (here $h'$ is a diffeomorphism sufficiently close to the homeomorphism $h$).  

Since the maps $h'\circ\tau_{p_i/q}\circ h'^{-1}$ satisfy the relations in $\pi_1(M)$ coming from the $2$-handles we can extend the approximating foliation to the union of preimages of the $2$-handles. The extension over $3$-handles is no problem since a $1$-periodic foliation on $S^2\times \R$ transverse to the $\R$-fibers is a product foliation by the Reeb stability theorem (\thmref{t:Reeb}). We obtain a sequence of foliations $\FF_q$ on $E$. The extensions can be chosen so that the tangent spaces of the foliations $\FF_q$ converge uniformly to the original foliation $\FF_E$ on $E$ as $q\to\infty$. 

By construction, the foliations $\FF_q$ are proper, ie. the leaves of $\FF_q$ are properly embedded submanifolds. Hence the pullback of $\FF_q$ under $\sigma$ is a foliation by compact leaves as soon as $\sigma$ is transverse to $\FF_q$. This happens when $q$ is sufficiently large. 

So far we have only showed that $\FF$ can be approximated by foliations all of whose leaves are compact. Such foliations of codimension $1$ define fibrations over the circle. We will now use \eqref{e:irr}. Because  $\mathrm{gcd}(q,p_1,\ldots,p_n)=1$ the distance between two distinct points of a leaf of $\FF_q$ which lie on the same $\R$-fiber can be chosen in the interval  $[1/2q,2/q]$ (the factor $2$ accounts for the additional approximations e.g. of $h$). Because $(p_1/q,\ldots,p_n/q)$ satisfy \eqref{e:irr}, $\FF_q$ can be chosen such that the angle between leaves of $\FF_q$ and $\FF_E$ is bounded by a constant proportional to $1/q^{1+1/n}$. 

When $q$ is chosen big enough, then the angle between $\FF$ and the fibers of $\sigma^*\FF_q$ is bounded by a constant proportional to $1/q^{1+1/n}$ while the distance of two points on a fiber along leaves of the Sacksteder flow decreases linearly. 

If $q$ is sufficiently large, then the distance between the endpoints of a curve in the fiber whose length is smaller than $K+1$ and the endpoint of the $\FF$-horizontal lift along curves of the Sacksteder flow does never intersect {\em all} fibers of $\FF_q$. We say that $\FF_q$ is {\em well-approximating}. The fibers of the fibration $\sigma^*\FF_q$ of $M$ are Abelian coverings of a $\Sigma_0$.
 
{\bf Fixing the neighbourhood of $\FF$:} We choose $\eps>0$ so that 
\begin{itemize}
\item every plane field which is $\eps$-close to $\FF$ is transverse to the flow lines of the Sacksteder flow, and 
\item for every geodesic of length $\le K$ in a fiber of a well-approximating fibration the $\zeta$-horizontal lift with starting point in that fiber does not meet the same fiber so that the geodesic and its lift do not form a null-homotopic closed curve. 
\end{itemize}

We now show that the constant $\eps$ has the desired property. Let $\xi_0,\xi_2$ be $C^\infty$-generic positive contact structures $\eps$-close to $\FF$. We consider the movie of characteristic foliations on the fibers of the well-approximating foliation. From the condition on $\eps$ it follows that no sheet containing an attractive closed curve of the characteristic foliation of a fiber can be a closed torus in $M$. Therefore, when $\Sigma$ is a fiber of the well-approximating fibration so that it is a convex surface with respect to $\xi_0$ then we can isotope $\Sigma$ along the flow lines of the Sacksteder flow and thereby reduce the dividing set to a single non-separating pair using \lemref{l:remove separating} and \lemref{l:remove components}. The fact that one sheet which arises in the constructions of these two lemmas may hit the surface we are about to isotope more than once does not represent a problem since several pieces of the surface can be isotoped at the same time. (From the definition of the sheets it follows that two sheets either coincide or are disjoint.)

We now proceed as in the proof of \thmref{t:unique} when $\FF$ is a fibration over the circle. Using the construction from \exref{ex:periodic cont} together with the classification of tight contact structures we conclude that $\xi_0$ is isotopic to $\xi_2$.

\section{Applications and examples} \mlabel{s:appex}

In this section we apply \thmref{t:unique} to prove results about the topology of the space of taut foliations. Moreover, we give a few examples of approximations of foliations by contact structures where the foliation violates the assumptions of \thmref{t:unique} and \thmref{t:unique2} and every neighbourhood of the foliation contains contact structures which are not (stably) isotopic.

\subsection{Homotopies  through atoral foliations} \mlabel{s:homotopies}

\begin{defn} \mlabel{d:atoral}
A foliation $\FF$ is {\em atoral} if there is no torus leaf, not every leaf is a plane and not every leaf is a cylinder. 
\end{defn} 
In other words, atoral foliations are just those foliations which satisfy the assumptions of \thmref{t:unique}. This definition, like the following, makes sense for positive confoliations. However, in the next two sections we shall focus on foliations. According to a result from \cite{goodman}, a foliation on a closed $3$-manifold without torus leaves is taut. On the class of atoroidal manifolds, {\em atoral} foliations coincides with {\em taut} foliations. 

\begin{defn} \mlabel{d:close}
A contact structure $\xi$ {\em approximates} a foliation $\FF$ if every $C^0$-neighbourhood of $\FF$ contains a contact structure isotopic to $\xi$. 
\end{defn}
\thmref{t:unique} then just says that there is a unique positive contact structure approximating $\FF$ whenever $\FF$ is atoral. We have the following simple consequence:

\begin{thm} \mlabel{t:atoral homotopy}
Let $\FF_t, t\in [0,1]$, be a $C^0$-continuous family of atoral $C^2$-foliations. Then the positive contact structures $\xi_0$ respectively $\xi_1$ approximating $\FF_0$ respectively $\FF_1$ are isotopic. 
\end{thm}

\begin{proof}
By \thmref{t:unique} for each $t\in [0,1]$ there is a $C^0$-neighbourhood $U_t$ of $\FF_t$ in the space of plane fields so that all positive contact structures in $U_t$ are pairwise isotopic. By compactness we can cover the path $\FF_t$ by finitely many $C^0$-open sets $U_0,\ldots,U_N$ such that $U_i\cap U_{i+1}$ contains a foliation from the family $\FF_t$. According to \thmref{t:elth} there is a positive contact structure $\xi_i'$ in $U_i\cap  U_{i+1}$ since this is a $C^0$-open neighbourhood of a foliation. By the choice of $U_i$ 
$$
\xi_0\simeq\xi_0'\simeq\xi_1'\simeq\ldots\simeq \xi_{N-1}'\simeq\xi_1. 
$$
\end{proof}
We obtain an obstruction for two foliations $\FF_0$ and $\FF_1$ being homotopic through taut foliations when the underlying manifold is atoroidal: If the positive contact structures approximating the foliations are not isotopic, then there is no homotopy through taut foliations connecting $\FF_0$ and $\FF_1$.
This is of interest because of the following {\em h}-principle due to H.~Eynard (\cite{eynard}, building on \cite{larcanche}) reduces the question when two taut foliations are homotopic through foliations to a purely homotopy theoretic problem.
\begin{thm} \mlabel{t:eynard}
Two taut foliations $\FF_0$ and $\FF_1$ on $3$-manifolds are homotopic through foliations if and only if the corresponding plane fields are homotopic. 
\end{thm}
The first step in the proof of this theorem is the introduction of Reeb components. 

The following example shows that one can indeed use \thmref{t:atoral homotopy} to show that a pair of taut foliations is not homotopic through foliations without Reeb components foliations (also, they are not homotopic through taut foliations)  although they are homotopic through foliations. The example therefore shows that the introduction of Reeb components in Eynard's proof of \thmref{t:eynard} is necessary. To best of the knowledge of the author this is the first example of this kind. 

More information concerning the question which contact structures appear in neighbourhoods of foliations can be found in J.~Bowden's preprint \cite{bojo}.

\begin{ex} \mlabel{ex:2311}
We consider the Brieskorn homology sphere
$$
M=\Sigma(2,3,11)=\{(x,y,z)\in \mathbb{C}^3\,|\, x^2+y^3+z^{11}=0\}\cap S^5.
$$  
This manifold is a Seifert fibered space over $S^2$ with three singular fibers.  In terms of Seifert invariants this manifold is often denoted by $M(\frac{1}{2},-\frac{1}{3},-\frac{2}{11})$ (cf. \cite{gs}). Since $M$ is a homology sphere it does not fiber over $S^1$. 

 The tight contact structures on $M$ were classified by Ghiggini and Sch{\"o}n\-en\-berger in \cite{gs}. They showed that this manifold carries exactly two positive tight contact structures up to isotopy.  From the surgery description in in Section 4.1.4. of \cite{gs} of the two contact structures it follows that if $\xi$ is a tight contact structure on $M$, then $\overline{\xi}$ (this is $\xi$ with its orientation reversed) represents the other isotopy class of tight contact structures. 

In \cite{jn} (together with \cite{naimi}) it is shown that $M$ admits a smooth foliation $\FF$ transverse to the fibers. (The conjecture formulated on p.~398 of \cite{jn} is proved in \cite{naimi} and the Seifert fibered space we are considering satisfies the condition formulated in \cite{jn} with $(a=3,m=5)$). Note that since $M$ is a homology sphere, no taut foliation on $M$ has a closed leaf. In particular, there are no torus leaves and since $M$ does not fiber over $S^1$ there are no smooth foliations without holonomy. Hence $\FF$ satisfies the conditions of \thmref{t:unique}. 

Now let $\xi$ be a contact structure in a sufficiently small $C^0$-neighbourhood of $\FF$.  Notice that $\overline{\xi}$ approximates $\overline{\FF}$. By \thmref{t:atoral homotopy} the foliations $\FF$ and $\overline{\FF}$ are not homotopic through foliations without torus leaves. Since $M$ is a homology sphere this is equivalent to saying that $\FF$ and $\overline{\FF}$ are not homotopic through taut foliations. 

This is non-trivial since $\FF$ and $\overline{\FF}$ are homotopic as plane fields as can be shown using the invariants from \cite{gompf} which form a complete invariant for oriented plane fields on $3$-manifolds up to homotopy. These invariants are: 
\begin{itemize}
\item The Euler class $e(\xi)\in H^2(M,\Z)$.
\item A rational number $\theta(\xi)$ defined by 
\begin{equation} \label{e:theta}
\theta(\xi)= \left(\mathrm{PD}\left(c_1(X,J)\right)\right)^2-2\chi(X)-3\sigma(X)\in \Q
\end{equation}
if $e(\xi)$ is a torsion class. Here $(X,J)$ is a $4$-manifold oriented by an almost complex structure $J$ with signature $\sigma(X)$ and Euler characteristic $\chi(X)$ such that $M=\partial X$ as oriented manifolds, $J(\xi)=\xi$ and $J$ induces the original orientation of $\xi$.  
\item An element $\Theta_f(\xi)\in\Z/(2d)$, where $d$ is the divisibility of the Euler class and $f$ is a framing of a curve representing the Poincar{\'e} dual of $e(\xi)$, if $e(\xi)$ is not a torsion class. 
\end{itemize}
If the plane field is a contact structure given by a Legendrian surgery diagram, then these invariants can be computed effectively.

In the case at hand it is clear that the Euler class of $e(\xi)=-e(\overline{\xi})$ vanishes since $M$ is a homology sphere. It follows from \eqref{e:theta} and $c_1(X,-J)=-c_1(X,J)$ 
that $\theta(\xi)=\theta(\overline{\xi})$. Hence $\xi$ and $\overline{\xi}$ are homotopic as oriented plane fields and the same is true for $\FF$ and $\overline{\FF}$. So by \thmref{t:eynard} the foliations $\FF$ and $\overline{\FF}$ are homotopic through foliations but by \thmref{t:atoral homotopy} this homotopy has to contain torus leaves. One can easily show that every 
foliation without Reeb components on $M$ is taut. 

\end{ex}

\begin{ex} \mlabel{ex:gy}
The work of Ghys and Sergiescu \cite{ghys} provides another class of pairs of foliations which are not homotopic through atoral foliations. These foliations are the stable and unstable foliations $\FF^s,\FF^u$ of Anosov flows on $T^2$-bundles over $S^1$ which are suspensions of $A\in\mathrm{Gl}(2,\Z)$ such that $|\mathrm{tr}(A)|>2$. The reason why these foliations are not homotopic through atoral foliations is the fact that atoral foliations on suspensions of orientation preserving Anosov diffeomorphisms of $T^2$ admit a classification up to diffeomorphism: Each smooth atoral foliation is smoothly equivalent to $\FF^s$ or $\FF^u$ \cite{ghys}. In many instances these two types of foliations are not even diffeomorphic. It is obvious from \cite{confol, mit} that both foliations have isotopic approximating contact structures and that the corresponding plane fields are homotopic (through positive confoliations). Note that  $\FF^s$ and $\FF^u$ are both homotopic to the foliation by the torus fibers of the fibration if $\mathrm{tr}(A)>2$. The homotopy is given by 
$$
\alpha_s = sdt+(1-s)\lambda^{-t}\beta_\lambda
$$ 
where $s\in[0,1]$, $\beta_\lambda$ is a $1$-form on $T^2$ with constant coefficients such that $A^*(\beta_\lambda)=\lambda\beta_\lambda$. The coordinate on the base circle is denoted by $t$. This eigenvalue $\lambda$ is positive since $\mathrm{tr}(A)>2$. Thus $\FF^u$ and $\FF^s$ are even homotopic through taut foliations.
\end{ex}
We conclude from the last example that there are foliations which are not homotopic through atoral foliations despite of the fact that the approximating contact structures are isotopic. 

\subsection{Parabolic torus fibrations and linear foliations on $T^3$} \mlabel{s:T3}

In this section we discuss foliations given by fibers of particular torus fibrations over $S^1$ and we show that there are foliations with the property that every isotopy class of contact structures has positive $C^0$-distance from the foliation. Clearly, these foliations have to belong to the classes where \thmref{t:unique} does not apply.

\begin{ex} \mlabel{ex:tr2}
Let $M$ be the total space of a torus bundle over the circle whose monodromy $\phi_M$ satisfies $\mathrm{tr}(\phi_M)=+2$ (cf. \secref{s:torus bundles} for the notation). According to \cite{hector} these are exactly those orientable manifolds which admit $C^2$-foliations all of whose leaves are cylinders.  We may assume that $\phi_M=A=\left(\begin{array}{cc} 1 & 0 \\ k & 1 \end{array}\right)$ with $k\in \Z$. We view $M$ as  fibration over $T^2$ with the bundle projection 
\begin{align*}
M &\lra T^2 \\ 
(x_1,x_2,t) & \lmt (x_1,t). 
\end{align*}
and typical fiber $S^1$. Let $\FF_0$ denote the foliation by tori defined by $dt$ on $M$. Thus the leaves are the fibers of $M\lra S^1$. Let $\xi_{\eps,m}$ be the contact structure defined by 
\begin{equation} \label{e:xi2}
dt+\eps\big((|k|+1)\cos(2\pi mt)dx_1-\sin(2\pi mt)(dx_2-ktdx_1)\big)
\end{equation}
with $m$ a positive integer and $\eps>0$. This $1$-form is a contact form on $M$ and as $\eps\to0$ the corresponding plane field converges to $\FF_0$. The contact structures $\xi_{\eps,m}$ are distinguished by their Giroux torsion. 

However, there are other contact structures in every neighbourhood of $\FF_0$: Let $\eta_{\eps,m}$ be the positive contact structure defined by the $1$-form 
\begin{equation*}
\alpha_{\eps,m} = dx_1+\eps\left(\sin(2\pi mx_1)(dx_2-ktdx_1) + |k+1|\cos(2\pi mx_1) dt\right)
\end{equation*}
where $m$ is a positive integer and $\eps>0$ (these contact forms are taken from \cite{gi-inf}). Then $\eta_{\eps,m}$ converges to the foliation $\FF_1$ defined by $dx_1$ on $M$. Now consider the following automorphism of $M$ 
\begin{align*}
\psi_p : M & \lra M \\
(x_1,x_2,t) & \lmt \left(x_1+pt, x_2+\frac{kp}{2}t(t+1), t\right).
\end{align*}
This map covers the $p$-fold Dehn twist of $T^2$ given by $(x_1,t)\lmt (x_1+pt,t)$. We consider the foliations defined by
\begin{align*}
\psi_p^*(dx_1) & = dx_1+pdt.
\end{align*}
As $p\to \infty$ these foliations converge to $\FF_0$. Hence the contact structures $\eta_{\eps,m,p}$ defined by $\psi^*_p(\alpha_{\eps,m})$ form a sequence of positive contact structures converging to $\FF_0$. It is shown in \cite{gi-inf} that the contact structures $\eta_{\eps,m,p}$ and $\eta_{\eps',m',p'}$ are isotopic if and only if $m=m'$ and $p=p'$. Moreover, they are not isotopic to any of the contact structures $\xi_{\eps,m}$ defined by \eqref{e:xi2}.  
\end{ex}

A foliation $\FF$  on $T^3$ is {\em linear} if $\FF=\ker(\beta=a\,dx+b\,dy+c\,dz)$ with $a,b,c\in\R$. In the following we establish restrictions on the contact structures lying in a given $C^0$-neighbourhood of $\FF$. For this we first recall the classification of tight contact structures on $T^3=\R^3/\Z^3$ (with corresponding coordinates $x,y,z$ and oriented by the volume form $dx\ww dy\ww dz$).

\begin{thm}[Kanda, Giroux] \mlabel{t:T3}
A positive tight contact structure on $T^3$ is diffeomorphic to 
\begin{equation} \label{e:T3contact}
\xi_m=\ker(\alpha_m= \cos(2\pi mz)dx-\sin(2\pi mz)dy)
\end{equation}
for a unique $m\in\{1,2,\ldots\}$. Two tight contact structures $\xi,\xi'$ are isotopic if and only if there are contactomorphisms 
\begin{align*}
\psi : (T^3,\xi) & \lra (T^3,\xi_m)\\
\psi' : (T^3,\xi') & \lra (T^3,\xi_{m'}) 
\end{align*}
such that $m=m'$ and the pre-Lagrangian tori $\psi^{-1}(\{z=z_0\}),\psi'^{-1}(\{z=z_0'\})$ are isotopic.
\end{thm} 

This theorem implies in particular, that every oriented tight contact structure $\xi$ on $T^3$ is isotopic to $\overline{\xi}$ since $\xi_m$ is isotopic to $\overline{\xi}_m$ via the isotopy 
\begin{align*}
h_s : T^3 & \lra T^3 \\
   (x,y,z) & \lmt (x,y,z+\pi s/m)  
\end{align*}
with $s\in[0,1]$. Moreover we can associate a pre-Lagrangian torus $T_{\xi}$ which is well defined up to isotopy to an isotopy class of  contact structures $\xi$. Using this we will establish the following result:

\begin{prop} \mlabel{p:T3}
Let $\FF$ be a linear foliation on $T^3$ and $\xi$ a tight contact structure. 
\begin{itemize}
\item[(a)] If $\FF$ is not a foliation by closed leaves, then there is a neighbourhood $U_\xi$ of $\FF$ which does not contain any contact structure isotopic to $\xi$.  
\item[(b)] If $\FF$ is a foliation by closed leaves, then every $C^0$-neighbourhood of $\FF$ contains a contact structure isotopic to $\xi$ if and only if $T_{\xi}$ is isotopic to a leaf of $\FF$.  
\end{itemize}
\end{prop}

\begin{proof} 
Assume first that $\FF$ is a foliation by planes or cylinders and pick coordinates such that $\FF$ is defined by $\beta=dz+adx+bdy$. We consider $\FF$ and contact structures close to $\FF$ as connections of 
\begin{align*}
T^3 & \lra T^2 \\
(x,y,z) & \lmt (x,y).
\end{align*}

We fix a linear foliation $\GG$ by tori transverse to $\FF$ (lets say $\GG=\ker(dz)$) and two linear curves $\gamma_1$ and $\gamma_2$ in a leaf of $\GG$ so that both of these curves are transverse to $\FF$. In order to prove the claim we first show the following statement.

{\bf Claim:} Let $\xi$ be a contact structure satisfying the following conditions. 
\begin{itemize}
\item[(i)] $\xi$ is transverse to a foliation $\GG$ by tori transverse to $\FF$. 
\item[(ii)] $\xi$ is transverse to the fibers of $T^3\lra T^2$ and to all $\GG$-horizontal lifts of $\gamma_1$ and $\gamma_2$.
\end{itemize}
Then there are $\xi$-Legendrian curves  $\widehat{\gamma}_i$ embedded  in $T_i=\{(x,y,z)\,|\,(x,y)\in\gamma_i\}$  whose slope is the slope of the curve obtained by intersecting the linear torus isotopic to the pre-Lagrangian tori of $\xi$ and $T_i$. 
\medskip

Since these curves are $\xi$-horizontal lifts of linear curves in $\{z=z_0\}$ their slope is controlled by $\FF$ and the $C^0$-distance between $\xi$ and $\FF$. 
Now if one of the $\FF$-horizontal curves has irrational slope, then for every isotopy class of tori there is a neighbourhood of $\FF$ such that the fixed torus cannot be the pre-Lagrangian torus of a contact structure in that neighbourhood. Hence the claim above proves both (i) and (ii) of the proposition. 
\medskip

We now prove the claim: Given $\xi$ as above we consider the movie $T_z(\xi)$ on the leaves of $\GG$. The condition that $\xi$ is transverse to all $\GG$-horizontal copies of $\gamma_i$ together with \thmref{t:T3} ensures that the characteristic foliation of $\xi$ on the torus $T_z$ is never linear. (If it were, then $T_z$ would be a pre-Lagrangian torus and the slope of the characteristic foliation on $T_{z'}$ would have to make at least one full twist as $z'\in S^1$ varies.) Since there are no singular points the characteristic foliation has several parallel closed leaves and other leaves accumulating on closed leaves. The slope of the closed leaves of $T_z(\xi)$ is independent from $z$.

Let $T_z(\xi)$ be a generic leaf of $\GG$. Because attractive and repulsive leaves on $T_z$ alternate, the manifold consisting of the closed leaves of $T_z(\xi), z\in S^1$ has at least two connected components $T,T'$. We assume that the region bounded by $T$ and $T'$ contains no other connected component of the surface formed by closed leaves of $T_z(\xi)$. Now consider the intersection of $T_1$ with $T$ and $T'$. This intersection consists of transverse curves on $T_i$ such that the characteristic foliation either points into or out of the annulus bounded by these two curves (the annulus is the intersection of the region between $T$ and $T'$ with $T_i$) but the behavior is the same along both transverse curves. Since there are no singular points on $T_i$, there has to be a closed leaf in that annulus. This is the closed leaf we have been seeking. 
\end{proof} 
A similar argument should show that foliations by cylinders on torus bundles over the circle as in \exref{ex:tr2} (defined by $dx_1+adt$ with $a\in\R\setminus\Q$ in the coordinates used in \exref{ex:tr2}) have analogous properties as the foliations by planes considered in \propref{p:T3}.    
\begin{cor}
Let $\FF$ be a linear foliation on $T^3$ with non-compact leaves. Then $\FF$ cannot be deformed to a contact structure, i.e. there is no family of plane fields $\xi_s$ with $\xi_0=\FF$ and $\xi_s$ a contact structure for $s>0$.
\end{cor}
\begin{proof}
If $\xi_s$ is a deformation of $\FF$ into a contact structure then Grays theorem implies that every neighbourhood of $\FF$ contains a contact structure isotopic to $\xi_1$. But according to \propref{p:T3} a sufficiently small neighbourhood of $\FF$ does not contain a contact structure isotopic to $\xi_1$.  
\end{proof}

A different type of example for this phenomenon was found much earlier by J.~Etnyre \cite{etnyreCk}. His example is slightly  different since it refers to Giroux torsion rather than pre-Lagrangian tori but both examples make essential use of pre-La\-grang\-ian  tori. It is natural to ask to whether all atoral foliations can be deformed into contact structures. 

Finally, we give an example of a foliation on $T^3$ with an unstable torus leaf such that every $C^\infty$-neighbourhood contains positive contact structures which are not stably isotopic. 

\begin{ex}\mlabel{ex:unstable torus}
Let $Z=f(z)\partial_z$ be a smooth vector field on $S^1=\R/\Z$ with $f(z)>0$ for $z\neq 0$ such that $f(z)=z^2$ on a neighbourhood of $0$. We denote the flow of $Z$ by $\varphi_t$. Mapping the two generators of $\pi_1(T^2)$ to $\varphi_t$ and $\varphi_{t'}$ with $0<t<t'$ we obtain a foliation on $T^3$ transverse to the fibers of $T^3\lra T^2$ and the torus $z=0$ is the only minimal set of $\FF_0$. By construction this torus is unstable. 

In order to show that this example has the desired properties we proceed as follows: Approximate $Z$ by $\widetilde{Z}=\widetilde{f}(z)\partial_z$ such that $\widetilde{f}(z)>0$ for all $z\in S^1$. We denote the flow of $\widetilde{Z}$ by $\widetilde{\varphi}$. Replacing $\varphi_t,\varphi_{t'}$ in the representation $\pi_1(T^2)\lra\mathrm{Diff}_+(S^1)$ by $\widetilde{\varphi}_t,\widetilde{\varphi}_t'$ we replace  $\FF$ by $\widetilde{\FF}$. If the rotation numbers $\widetilde{\rho},\widetilde{\rho}'$ of $\widetilde{\varphi}_t,\widetilde{\varphi}_t'$ are rationally independent the foliation $\widetilde{\FF}$ is a foliation by planes. Moreover, if there are integers $(c,d)$ with $c>0$ and $d>2$ such that 
\begin{equation} \label{e:dio}
\left|\widetilde{\rho}-\frac{p}{q}\right|>\frac{c}{q^{d}} 
\end{equation}
for all $q\in\N^+$ and $p\in \Z$, then according to Herman's thesis \cite{herman} $\widetilde{\FF}$ is smoothly conjugate to one of the linear foliations discussed in \propref{p:T3}. Note that $\widetilde{\varphi}_t$ and $\widetilde{\varphi}_{t'}$ commute.  A more explicit construction could avoid the use of Herman's work. The numbers satisfying the Diophantine condition \eqref{e:dio} are dense. Therefore we can approximate $\FF$ by foliations $\widetilde{\FF}_n$ all of whose leaves are planes and every neighbourhood of $\FF_n$ contains non-isotopic positive contact structures with vanishing Giroux torsion. 
\end{ex}

\subsection{Further applications} \mlabel{s:further} 
V.~Colin showed in \cite{col2} that foliations without Reeb components can be approximated by tight contact structures and asks whether or not this is true for {\em every} contact structure in a sufficiently small neighbourhood of the foliation. \thmref{t:unique2} together with the gluing results from \cite{col-glue} provide the following partial answer to this question.
\begin{prop} \mlabel{p:tight approx}
Let $\FF$ be a $C^2$-foliation without Reeb components such that all torus leaves have attractive holonomy. Then $\FF$ has a neighbourhood such that all contact structures in that neighbourhood are universally tight. 
\end{prop}
As in \secref{s:tori} the assumption on the holonomy of the tori can be weakened slightly. However, our methods do not suffice to remove the stability condition for torus leaves completely and the question whether every contact structure in a sufficiently small $C^0$-neighbourhood of a foliation without Reeb components is tight remains open.
 
Recall the following theorem from \cite{rigflex}: 
\begin{thm} \mlabel{t:star approx}
Let $(M,\xi)$ be a confoliation admitting an overtwisted star. Then $\xi$ can be $C^0$-approximated by overtwisted contact structures. 
\end{thm}

Together with the uniqueness result \thmref{t:unique} has the following consequence for confoliations which are not s-tight. 

\begin{cor}
Let $M$ be a closed manifold $\xi$ be a confoliation without torus leaves which is not $s$-tight. Then there is a $C^0$-neighbourhood of $\xi$ so that every contact structure in that neighbourhood is overtwisted.  
\end{cor} 
It is not clear whether or not the conclusion of the corollary also holds in the presence of incompressible torus leaves.


The techniques developed in the \secref{s:higher} and \secref{s:prelag extend}  have further applications. For example, combining Lemma 2.17 of \cite{gi-bif} with the pre-Lagrangian extension lemma, the construction from \exref{ex:periodic cont} and the methods from \secref{s:higher} one can prove the following slight extension of \thmref{t:Sigma class}.
\begin{thm} \mlabel{t:more Sigma class}
Let $\Sigma$ be a surface of genus $\ge 2$ and $\Gamma_0$ respectively $\Gamma_1$ sets consisting of $n_0>0$ respectively $n_1>0$  pairs of simple closed curves in $\Sigma$ such that each pair of curves bounds an annulus in $\Sigma$, all curves in $\Gamma_0$ and $\Gamma_1$ are non-separating and the annuli bounded by pairs of curves in $\Gamma_0$ and $\Gamma_1$ are pairwise non-isotopic. 

Then there are $2^{n_0+n_1}$ isotopy classes of tight contact structures on $\Sigma\times[0,1]$ such that $\Sigma_i$ is convex and divided by $\Gamma_i$ for $i=0,1$. The contact structures are distinguished by their relative Euler class and they are all universally tight. 
\end{thm}
This author believes that it is possible to give a more complete classification of extremal contact structures on $\Sigma\times[0,1]$  and hopes to elaborate on this in another paper. 

\end{document}